\documentclass[preprint,authoryear,11pt]{elsarticle}
\usepackage{natbib}
\usepackage{apalike}
\RequirePackage[T1]{fontenc}
\usepackage{color}
\usepackage[utf8]{inputenc}
\usepackage[T1]{fontenc}
\usepackage[english]{babel} 
\RequirePackage{amsfonts,amsthm,amsmath}

\RequirePackage{amsfonts,amsthm,amsmath}
\RequirePackage[colorlinks,citecolor=blue,urlcolor=blue]{hyperref}

\RequirePackage{hypernat}
\usepackage{a4wide}
\usepackage{here}
\usepackage{float}
\usepackage{dsfont}

\newtheorem{theorem}{Theorem}

\newtheorem{lemm}{Lemma}
\newtheorem{cor}[theorem]{Corollary}
\newtheorem{example}{Example}
\newtheorem{prop}[theorem]{Proposition}
\newtheorem{remi}{Remark}
\numberwithin{theorem}{section}
\numberwithin{lemm}{section}
\numberwithin{remi}{section}
\numberwithin{example}{section}

\newcommand{\zak}{\nobreak \ifvmode \relax \else
     \ifdim\lastskip<1.5em \hskip-\lastskip
     \hskip1.5em plus0em minus0.5em \fi \nobreak
     $\Box$\fi\\}
\def\Y{{\mathcal Y}}
\def\G{{\cal G}}
\def\EE{\mathbb{E}}

\usepackage{amsmath,amssymb,graphicx}
\usepackage{amsthm}
\usepackage{enumerate}

\allowdisplaybreaks

\newcommand{\eps}{\varepsilon}

\newcommand{\nbR}{\mathbb{R}}
\newcommand{\nbN}{\mathbb{N}}

\newcommand{\nbC}{\mathbb{C}}

\newcommand{\nbu}{\mathbb{1}}

\newcommand{\nbP}{\mathbb{P}}
\newcommand{\nbZ}{\mathbb{Z}}

\newcommand{\nbE}{\mathbb{E}}

\def\PP{{\mathbb P}}

\def\EE{{\mathbb E}}

\newcommand{\N}{\ensuremath{\mathbb{N}}}

\newcommand{\R}{\ensuremath{\mathbb{R}}}

\newcommand{\1}{\ensuremath{\mathbb{1}}}

\def\eps{{\epsilon}}

\def\1{{\mathbb 1}}

\def\var {\mathop{\rm Var}\nolimits}    
\def\cov{\mathop{\rm Cov}\nolimits}    
\usepackage{graphicx}
5

\begin{document}
\begin{frontmatter}

\title{Estimation of subcritical Galton Watson processes with correlated immigration
}

\author[uphf]{{{Yacouba}}
{{Boubacar Ma\"{\i}nassara}}
} 
\address[uphf]{
Université Polytechnique Hauts-de-France,
INSA Hauts-de-France, CERAMATHS - Laboratoire de
Matériaux Céramiques et de Mathématiques, F-59313 Valenciennes, France}
\ead{mailto:Yacouba.BoubacarMainassara@uphf.fr}
\author[ufc]{{{Landy }}
{{Rabehasaina}}
}
\address[ufc]{ Universit\'e Bourgogne Franche-Comt\'e, Laboratoire de math\'ematiques de Besan\c{c}on, UMR CNRS 6623, 16 Route de Gray, 25030 Besan\c{c}on, France}
\ead{mailto:lrabehas@univ-fcomte.fr}

%

%
\begin{abstract}

We consider an observed subcritical Galton Watson process $\{Y_n,\ n\in \nbZ\}$ with correlated stationary immigration process $\{\epsilon_n,\ n\in \nbZ \}$. Two situations are presented. The first one is when $\mbox{Cov}(\epsilon_0,\epsilon_k)=0$ for $k$ larger than some $k_0$: a consistent estimator for the reproduction and mean immigration rates is given, and a central limit theorem is proved. The second one is when $\{\epsilon_n,\ n\in \nbZ \}$ has general correlation structure: under mixing assumptions, we exhibit an estimator for the the logarithm of the reproduction rate and we prove that it converges in quadratic mean with explicit speed. In addition, when the mixing coefficients decrease fast enough, we provide and prove a two terms expansion for the estimator. Numerical illustrations are provided.
\end{abstract}
\end{frontmatter}
\normalsize

\section{Introduction}\label{sec:Intro}
Estimation of parameters in a Galton Watson process with immigration has a long history: we refer to the seminal paper \cite{KN78} for laying the ground and expliciting conditional least square estimators for the expectation of the reproduction and immigration sequences in the subcritical case. A central limit theorem for these estimators was later proved in \cite{V82}, using a time series point of view. Note that the link between such processes and the so called integer-valued times series INAR$(1)$ processes has been exploited, see e.g. \cite{AOA87} which studied such a process with particular distributions for the reproduction and immigration sequences. See also \cite{MCK2003} who reviews the literature on models for integer-valued time series.
These above processes play an important role in many scientific disciplines and applied fields such as epidemiology, economics, finance, to name a few. Note also that the link between such processes and the so called Hawkes  processes introduced by \cite{H1971a,H1971b} as a model for contagious processes such as measles infections or hijackings  has been exploited in the literature. See \cite{DVJ2003} for a reference book that covers many aspects of the Hawkes process. For instance, \citet[Example 6.3(c)]{DVJ2003} show that each immigrant of Hawkes process has the potential to produce descendants whose numbers in successive generations constitute a subcritical Galton–Watson branching process with Poisson offspring distribution (see also \cite{HO1974}). Recently, it was proved in \citet[Theorem 3]{K2016} that Hawkes prcesses can be approximated by integer-valued autoregressive models of infinite order (INAR($\infty)$) processes, which in turn are the limit of INAR$(p)$ processes as $p\to \infty$ see \citet[Proposition 2.5]{K2017}. A certain number of extensions for the model were later devised and studied. \cite{WW90} considered the general critical and supercritical case and proved central limit theorems for (modified) weighted least square estimators. In \cite{barczy2021}, the specific case where the immigration sequence has a regular variation distribution is considered, leading to asymptotic normality of the reproduction mean when the immigration mean is known. Generalization to two types processes have been recently investigated in \cite{ispany2014, kormendi2018}, including estimations for the criticality parameter. Note that these references have one of the following constraint on the immigration process: first, some of them assume that its expectation is known (and appears in the expression of the estimator for the reproduction sequence); second, this immigration process is assumed to be a sequence of {\it independent} random variables, in addition to be identically distributed. 
%

We aim in this paper at considering a process which is stationary, but where the immigration process is dependent. To our knowledge, this particular feature appears not to have been studied in the literature. 
Thus, 
there are two major contributions in this work. The first one is the relaxing of the standard independence assumption on  the immigration process in order to extend considerably the range of application of the Galton Watson  models and to show that the moment estimation procedure can be extended to the Galton Watson  models with correlated immigration to estimate the unknown mean reproduction and immigration parameters $\lambda_0$ and $m_0$. This goal is achieved thanks to Proposition \ref{prop_consistency} and  Theorem \ref{theo_CLT_vector_k_0} in which the consistency
and the asymptotic normality are stated when the immigration sequence is no longer correlated from a certain instant. Still in this context, another contribution is to propose a weakly consistent estimator of the unknown 
 asymptotic variance matrix (see Theorem \ref{convergence_Isp}), which may be very different from that obtained in the standard framework (see for instance \citet[Theorem 2.2]{KN78} or \citet[Section 4.2]{AOA87}). Thanks to this estimation of the asymptotic variance matrix, we can for instance construct a confidence region for the estimation of the parameters (the reproduction and mean immigration rates). The second major contribution is when the immigration process has general correlation structure: we exhibit an estimator for the logarithm of the reproduction rate and we prove that it converges in quadratic mean with explicit speed at most of the order ${1}/{\ln(n)}$ (see Theorem \ref{main_theo_general1}). Convergence  is thus very slow, which in particular explains why we need to pick very large $n$ in the numerical illustrations. Under mildly stronger assumption, we provide and prove a two terms expansion for the estimator (see Theorem \ref{main_theo_genera2}) and thus deduce that the speed of convergence is exactly of the order ${1}/{\ln(n)}$.

The paper is organized as follows. Section \ref{sec:assumption}
presents the Galton Watson process that we consider here. A series of Propositions and Lemmas which are going to be repeatedly used are presented in Section \ref{sec:preliminary}. It is shown in Section \ref{sec:ultimate} that the moment estimators of the unknown mean reproduction and immigration parameters $\lambda_0$ and $m_0$ are 
consistent and asymptotically normally distributed when the immigration sequence is no longer correlated from a certain instant and  satisfies mild mixing assumptions. The proof relies on using Herrnorf's result (see \cite{H84}) that involves a truncated expansion of the stationary Galton Watson with correlated immigration process. These results use tools used previously in the time series literature in the case of uncorrelated but non-independent error terms ({\em i.e.} weak white noise), see \cite{FZ98}, except for some major differences: first, the model is different from the one in a time series setting and requires a different approach for determining the estimators, and second, the asymptotic normality concerns estimators which are different from least square estimators in \cite{FZ98}. The asymptotic variance of the moment estimators may be very different in the correlated and independence immigration cases. 
The estimation of this unknown asymptotic covariance matrix is done in Section \ref{estimOmega}. We consider in Section \ref{sec:general} a more general case of correlated immigration process. Here the approach is completely novel to our knowledge. The core results of the section are the quadratic convergence of the estimator as well as a two terms stochastic expansion for the estimator of $\ln\lambda_0$. The ingredients for the proofs relies on fine estimates for the estimator coupled with a central limit theorem for mixing triangular arrays of variables whose dependence is allowed to grow with the sample size  proved in \cite{FZ05}. Note that in the aforementioned results all estimators are determined from the observation of the base Galton Watson process only; in particular, the expectation of the immigration is unknown, and actually also estimated. Numerical illustrations are given in Section \ref{sec:numerics}. Finally, proofs of intermediate technical results are in Appendix \ref{sec:app}.

\section{Model, notation and Assumptions}\label{sec:assumption}
We consider the following Galton Watson process with immigration as the stationary sequence $\{Y_n,\ n\in\nbZ \}$ satisfying
\begin{equation}\label{model}
Y_{n+1}=\sum_{k=1}^{Y_n} \xi_{n+1,k}+\eps_{n+1},\quad n\in \nbZ ,
\end{equation}
for some sequences $\{\xi_{n,k},\ n\in\nbZ,\ k\in \nbN \}$ and $\{\eps_n,\ n\in\nbZ \}$, named thereafter the {\it reproduction} sequence and the {\it immigration} sequence, which are such that
\begin{itemize}
\item $\{\xi_{n,k},\ n\in\nbZ,\ k\in \nbN\}$ and $\{\eps_n,\ n\in\nbZ \}$ are independent sequences,
\item the reproduction sequence $\{\xi_{n,k},\ n\in\nbZ,\ k\in \nbN \}$ is an i.i.d. (doubly indexed) sequence, with distribution of a generic r.v. denoted by $\xi$,
\item the immigration process $\{\eps_n,\ n\in\nbZ \}$ is stationary and ergodic, with distribution that of a generic r.v. denoted by $\eps$,
\item the moments $\lambda_0:=\nbE(\xi)$ and $m_0:=\nbE(\eps)$ of $\{\xi_{n,k},\ n\in\nbZ,\ k\in \nbN \}$ and $\{\eps_n,\ n\in\nbZ \}$ are unknown.
\end{itemize}
The aim of the paper is devoted to the problem of estimating the unknown reproduction rate $\lambda_0$ as well as the mean immigration $m_0$ from the observed sequence $\{Y_n,\ n\in\nbZ \}$. 

Furthermore, in order for the existence of the stationary process $\{Y_n,\ n\in\nbZ \}$ to exist and be unique, it is required that the subcritical case holds, namely that
 \begin{equation}\label{stab}
\lambda_0<1 . 
 \end{equation}
To control the serial dependence of the stationary process $\{\eps_n,\ n\in\nbZ \}$, we introduce the strong mixing coefficients $\alpha_{{\epsilon}}(h)$ defined by
$$\alpha_{{\epsilon}}\left(h\right)=\sup_{A\in\mathcal F^n_{-\infty},B\in\mathcal F_{n+h}^{\infty}}\left|\mathbb{P}\left(A\cap
B\right)-\mathbb{P}(A)\mathbb{P}(B)\right|,$$
where $\mathcal F_{-\infty}^n=\sigma ({\epsilon}_u, u\leq n )$ and $\mathcal F_{n+h}^{\infty}=\sigma ({\epsilon}_u, u\geq n+h )$. Note that $\alpha_{{\epsilon}}(h)$ does not depend on $n\in \nbZ$ thanks to the stationarity of $\{\eps_n,\ n\in\nbZ \}$.

To establish the asymptotic properties of the proposed estimator   of $\lambda_0$, the following assumptions are required.

\begin{itemize}
\item[\hspace*{1em} {\bf (A1):}]
\hspace*{1em}   The mixing coefficient $\alpha_{{\epsilon}}(\cdot)$ 
verifies the following  summability condition. There exists $\beta>2$ such that
$$
\sum_{h=0}^\infty \lambda_0^{-h} \alpha_{\epsilon}(h)^{1-2/\beta}<\infty .
$$
\end{itemize}
We next make an integrability assumption on the moments and covariances of the immigration process $\{\eps_n,\ n\in\nbZ \}$ and the reproduction sequence $\{\xi_{n,k},\ n\in\nbZ,\ k\in \nbN \}$.
We use $\|\cdot\|$ to denote the Euclidean norm  of a vector and for any (potentially matrix valued) random variable $X$, we will set $\|X\|_p^p:= \nbE\|X\|^p$ its $\mathbb{L}^p$ norm, with $p\ge 1$. 
\begin{itemize}
\item[\hspace*{1em} {\bf (A2):}]
\hspace*{1em}  The following moment conditions hold:
$$
\|\eps\|_{2\beta} = \left[\nbE\left|\eps\right|^{2\beta}\right]^{1/(2\beta)}<\infty\quad\text{and}\quad \|\xi\|_{2\beta} = \left[\nbE\left|\xi\right|^{2\beta}\right]^{1/(2\beta)}<\infty .
$$
\end{itemize}
\begin{itemize}
\item[\hspace*{1em} {\bf (A3):}]
\hspace*{1em}  
 The covariance of the immigration process $\nu_h:=\mbox{Cov}(\eps_0,\eps_h)$ verifies:
$$
\sum_{h=0}^\infty h \lambda_0^{-h}\left|\nu_{h+1}\right|<\infty .
$$
\end{itemize}
Assumption $\mathbf{ (A3)}$ may be removed if we replace $\mathbf{ (A1)}$ by $\sum_{h=0}^\infty h\lambda_0^{-h} \alpha_\epsilon(h)^{1-2/\beta}<\infty$ thanks to Davydov's inequality (see \citet[Inequality (2.2)]{D68}) and $\mathbf{ (A2)}$. Also note that {\bf (A1)} implies that the mixing coefficient $\alpha_{{\epsilon}}(\cdot)$ decreases at least exponentially, with decay rate depending on the unknown parameter $\lambda_0$. This unknown decay rate may be problematic because {\bf (A1)} may thus not be easy to verify it in practice. Note that, in the case when one imposes $\lambda_0$ to lie in a known interval $[\lambda_-,\lambda_+] $ where $0<\lambda_-<\lambda_+<1$ (as is the case later on in Section \ref{sec:general}), then {\bf (A1)} can be replaced by the stronger assumption $\sum_{h=0}^\infty \lambda_-^{-h} \alpha_{\epsilon}(h)^{1-2/\beta}<\infty $ which may be easier to check in practice.

We need to introduce the following notation in the sequel:
\begin{eqnarray}
u_k&:=& \nbE(Y_0 Y_{k+1})=\nbE(Y_1 Y_{-k}),\quad k\in \nbN ,\label{def_u_k}\\
f_\nu(x) &:=& \sum_{j=0}^\infty \nu_{j+1} x^j,\quad |x| \le  \lambda_0^{-1} 
\label{def_f_nu}\\
\chi_k &:=& -\frac{1}{\lambda_0^k}\sum_{j=k}^\infty \lambda_0^j\nu_{j+1}, \quad k\in \nbN , \label{def_chi}
\end{eqnarray}
Note that the existence of $\chi_k$ defined in \eqref{def_chi} is an easy consequence of the convergence of the series in Assumption $\mathbf{(A3)}$, and that $f_\nu$ is a power series with convergence radius larger than $\lambda_0^{-1}$ for the same reason. Observe also that $\chi_0=- f_\nu(\lambda_0)$.
Furthermore, it will be proved in  Proposition \ref{prop_stationarity} that, under Condition \eqref{stab} and Assumption $\mathbf{ (A2)}$, the process $\{Y_n,\ n\in\nbZ \}$ is a second-order stationary process of which the first and second-order moments are finite and defined by:
\begin{equation}\label{def_moments}
C_1:= \nbE(Y_0),\quad C_2:= \nbE(Y_0^2).
\end{equation}
Furthermore, taking the expectation in \eqref{model} with $n=0$ and using the independence of $\{ \xi_{1,j},\ j\in \nbN\}$'s and $Y_0$ imply the explicit expression of the first order moment of the process $\{Y_n,\ n\in\nbZ \}$:
\begin{equation}\label{expr_first_moment}
C_1=\frac{m_0}{1-\lambda_0}.
\end{equation}
The expression of $C_2$ is also available and will be derived later on in Lemma \ref{lemma_v_k}. 
\section{Preliminary results}\label{sec:preliminary}
\subsection{Expansion for $Y_n$}
Let us introduce for all $n\in \nbZ$ the random operator $\theta_n:\nbN\longrightarrow \nbN$ defined as
$$
\theta_n\circ k:=\sum_{i=1}^k \xi_{n,i}
$$
for all $k\in \nbN$. We note that $\theta_n$ depends on the $\xi_{n,i} $, $i\in\nbN$, so that the operators $(\theta_n)_{n\in\nbN}$ are independent. Following \citet[Relation (4)]{KW21}, iteration of \eqref{model} yields the representation for the stationary distribution
\begin{equation}\label{rep_stationary}
Y_0=\sum_{i=0}^\infty \theta^{(i)}\circ\epsilon_{-i}
\end{equation}
provided that the series converges, where $\theta^{(i)}$, $i\in \nbN$, are independent random operators, such that $\theta^{(0)}=\mbox{Id}$ and $\theta^{(i)}\stackrel{\cal D}{=}\theta_0\circ \theta_{-1}\circ ...\circ \theta_{-i+1}$, $i\ge 1$, and independent from the immigration sequence $\{ \epsilon_i,\ i\in\nbZ\}$. Even more generally and more precisely, let us introduce a family of i.i.d. Galton Watson processes $\{ \{\theta_{j,j'}(i),\ i\in \nbN \},\ (j, j')\in \nbN^2\}$ starting from $\theta_{j,j'}(0)=1 $ and with same offspring distribution as $\{\xi_{n,k},\ n\in\nbZ,\ k\in \nbN\}$, and independent from the immigration process $\{\eps_n,\ n\in\nbZ \}$. We may write
\begin{equation}\label{proof_step_2_CTL_k0}
Y_n=\sum_{i=0}^\infty \theta^{(i)}_n\circ\epsilon_{n-i},\quad n\in \nbZ,
\end{equation}
where $\{\theta^{(i)}_n,\ i\in \nbN,\ n\in \nbN\}$ are a set of operators defined from $\nbN$ to $\nbN$, written as
\begin{equation}\label{op_stationary}
\theta^{(i)}_n\circ k := \sum_{j=1}^k \theta_{n-i,j}(i).
\end{equation}
This way, the term $\theta_{n-i,j}(i)$ in $\theta^{(i)}_n\circ\epsilon_{n-i}=\sum_{j=1}^{\epsilon_{n-i}} \theta_{n-i,j}(i)$ in \eqref{proof_step_2_CTL_k0} may be interpreted as the number of descendants at time $n$ of the $j$th immigrants arrived at time $n-i$, $j\in\{1,...\epsilon_{n-i}\}$ and $i\in\nbN$. 
Note that the representation \eqref{op_stationary} in particular implies the two following important facts. First, the operators $\theta^{(i)}_n$, $i\in \nbN$, $n\in \nbZ$, are independent from the immigration sequence $\{ \epsilon_i,\ i\in\nbZ\}$. Second, for all $n$ and $r$ in $\nbZ$ and $i\ge 0$, $j\ge 0$, $\theta^{(i)}_n$ is independent from $\theta^{(j)}_r$ as soon as $ n-i\neq r-j$. This latter feature will be extensively used later on, particularly in the proofs of the central limit theorems (see upcoming Theorem \ref{theo_CLT_vector_k_0} and Proposition \ref{prop_main_theo2}).
\begin{prop}[Stationary distribution and existence of moments]\label{prop_stationarity}\normalfont
The stationary version of the model described by \eqref{model} exists and admits moments of order $2\beta$ under Assumption $\mathbf{ (A2)}$.
\end{prop}
The proof of Proposition \ref{prop_stationarity} is postponed to Section \ref{sec:prop_beta_moments}.
\begin{lemm}\label{lemma_v_k}\normalfont
Let us define $v_k:=\nbE(\epsilon_1 Y_{-k})=\nbE(\epsilon_{k+1} Y_{0})$ for all $k\in \nbN$. Its explicit expression is given by
\begin{equation}\label{expr_v_k}
v_k= u_k - \lambda_0 u_{k-1}=\frac{m_0^2}{1-\lambda_0}- \chi_k,\quad k\in \nbN 
\end{equation}
where $u_k$ and $\chi_k$ are given by \eqref{def_u_k} and \eqref{def_chi}. As a consequence, the second moment of $Y_0$ is given by
\begin{equation}\label{expr_moments}
C_2 = \frac{1}{1-\lambda_0^2}\left[ \frac{V_{0,\xi} m_0}{1-\lambda_0}+2\lambda_0\left(\frac{m_0^2}{1-\lambda_0} +f_\nu(\lambda_0)\right)+\nbE\epsilon_1^2 \right]
\end{equation}
where $V_{0,\xi}:=\mbox{Var}(\xi)$.
\end{lemm}
The proof of Lemma \ref{lemma_v_k} is postponed to Section \ref{sec:lemma_vk}.
%
\subsection{Correlations of the process $\{Y_n,\ n\in \nbZ \}$} 
We next study properties related to the asymptotic behaviour of $\mbox{Cov}(Y_0,Y_n)$. The following result, proved in Appendix \ref{sec:proof_prop_C_1_u_k}, shows that the correlation of the process $\{Y_n,\ n\in\nbZ \}$ converges to $0$ exponentially fast.
\begin{cor}\label{prop_C_1_u_k}\normalfont
One has the following expression
\begin{equation}\label{expr_C_1_u_k}
C_1^2 - u_k= \lambda_0^k \left(\sum_{j=0}^k \lambda_0^{-j} \chi_j + \lambda_0(C_1^2-C_2)\right).
\end{equation}
In particular, one has that $\cov(Y_0,Y_{k+1})= u_k-C_1^2=\mathrm{O}(\lambda_0^k)$ as $k\to\infty$.
\end{cor}

One deduces from \eqref{expr_C_1_u_k} the asymptotic equivalent $C_1^2 - u_k\sim \lambda_0^k \left(\sum_{j=0}^\infty \lambda_0^{-j} \chi_j + \lambda_0(C_1^2-C_2)\right)$ as $k\to\infty$, when the quantity $\sum_{j=0}^\infty \lambda_0^{-j} \chi_j + \lambda_0(C_1^2-C_2)$ is not zero. The latter quantity will be studied later on, hence it will be interesting to write this quantity in function of the parameters $m$, $\lambda_0$ and the function $f_\nu$ defined in \eqref{def_f_nu}. This is done as follows. Using Fubini and \eqref{def_chi}, we first obtain that
\begin{multline}\label{calcul_serie_chi}
\sum_{j=0}^\infty \lambda_0^{-j} \chi_j= - \sum_{j=0}^\infty \lambda_0^{-2j} \sum_{r=j}^\infty \lambda_0^r \nu_{r+1}=-\sum_{r=0}^\infty \lambda_0^r \left( \sum_{j=0}^r \lambda_0^{-2j}\right)\nu_{r+1}\\
=-\sum_{r=0}^\infty\lambda_0^2 \frac{\lambda_0^{r}-\lambda_0^{-(r+2)}}{\lambda_0^2-1}\nu_{r+1}= \frac{\lambda_0^2}{1-\lambda_0^2}\left[ f_\nu(\lambda_0)-\lambda_0^{-2}f_\nu(\lambda_0^{-1})\right].
\end{multline}
Moreover, the expressions \eqref{expr_first_moment} and \eqref{expr_moments} of $C_1$ and $C_2$ the first and second moments of $\{Y_n,\,n\in\mathbb{Z}\}$ entail that
\begin{multline*}
C_1^2-C_2=\frac{m_0^2}{(1-\lambda_0)^2}-\frac{1}{1-\lambda_0^2}\left[ \frac{V_{0,\xi} m_0+2\lambda_0 m_0^2}{1-\lambda_0}+2\lambda_0f_\nu(\lambda_0)+V_{0,\epsilon} +m_0^2 \right]\\
= -  \frac{V_{0,\xi} m_0}{(1-\lambda_0^2)(1+\lambda_0)}- \frac{2\lambda_0}{1-\lambda_0^2} f_\nu (\lambda_0) - \frac{ V_{0,\epsilon}}{1-\lambda_0^2},
\end{multline*}
by using $\nbE\epsilon_1^2=V_{0,\epsilon}+m_0^2$ with $V_{0,\epsilon}:=\mbox{Var}(\epsilon)$. 
Combining the above expression and \eqref{calcul_serie_chi}, yields
\begin{equation}
\sum_{j=0}^\infty \lambda_0^{-j} \chi_j + \lambda_0(C_1^2-C_2) =-  \frac{V_{0,\xi} m_0 \lambda_0}{(1-\lambda_0^2)(1+\lambda_0)} - \frac{ V_{0,\epsilon} \lambda_0}{1-\lambda_0^2} - \frac{\lambda_0^2}{1-\lambda_0^2} f_\nu (\lambda_0) - \frac{1}{1-\lambda_0^2} f_\nu(\lambda_0^{-1}) .\label{expr_f_nu}
\end{equation}
We also mention that, still when $\sum_{j=0}^\infty \lambda_0^{-j} \chi_j + \lambda_0(C_1^2-C_2)$ is not zero, \eqref{expr_C_1_u_k} yields the following interesting limit
$$
\lambda_0= \lim_{k\to\infty}\frac{C_1^2 - u_k}{C_1^2 - u_{k-1}},
$$
which is the starting point for exhibiting estimators for $\lambda_0$ in the following Section.
\section{Ultimately uncorrelated immigration}\label{sec:ultimate}
We suppose in this section that Assumptions {\bf (A1)}, {\bf (A2)} and that $\mathbf{ (A3)_1}$ defined thereafter hold.
\begin{itemize}
\item[\hspace*{1em} $\mathbf{ (A3)_1}$:]
\hspace*{1em}   $\exists k_0\in \nbN\setminus \{0\}$ such that $\nu_k=0$ for all $k\ge k_0$, where the instant $k_0$ is known.
\end{itemize}
In other words, the above assumption means that the immigration is no more correlated from an instant $k_0$. Also note that $\mathbf{ (A3)_1}$ is stronger than {\bf (A3)}.
\begin{example}\label{example1}\normalfont
Let $\{Z_n,\ n\in\nbZ \}$ be any sequence of iid discrete (for instance: Poisson, Bernoulli, Binomial,\dots). For fixed $k_0>0$, let us assume that those random variables have finite $2\beta$ moments and let 
\begin{equation}\label{example_corr}
\epsilon_n:=\prod_{i=0}^{k_0-1} Z_{n-i}.
\end{equation}
We may alternatively assume that they have $4\beta$ moments and let $\epsilon_n:=Z_{n}^2\prod_{i=1}^{k_0-1} Z_{n-i}$. Thus we have:
\begin{equation}\label{corr1}
\nu_h=\left\{\begin{array}{ccc}
\mbox{Cov}\left(\epsilon_{n},\epsilon_{n-h}\right)\neq0 & \mbox{ if }  & h\leq k_0-1,\\
\mbox{Cov}\left(\epsilon_{n},\epsilon_{n-h}\right)=0 & \mbox{ if } &h\geq k_0 .
\end{array}\right.
\end{equation}
 Then the $\epsilon_n$ form an $(k_0-1)$-dependent sequence; that is, the mixing coefficients of the $\epsilon_n$  sequence satisfies $\alpha_{\epsilon}(h)=0$ if $h> k_0-1$, and the assumptions {\bf (A1)},  {\bf (A2)} and  $\mathbf{ (A3)_1}$ are satisfied.
\end{example}
\begin{example}\label{example_immigration}\normalfont
The previous example may be generalized when $k_0=2$ as follows. We suppose that we are in a setting where $\{Y_n,\ n\in\nbZ\}$ described by \eqref{model} represents the number of persons infected by an illness. We consider an exogenous system that provides the contaminated immigrants $\{\epsilon_n,\ n\in\nbZ\}$ described as follows. This system consists of contaminated individuals pertaining to two classes, namely the class of contagious individuals  (Class 1)  and the class of non contagious individuals  (Class 2).
We are endowed with a family of $\nbN\setminus\{0\}$ valued, independent and identically distributed processes $\{(\zeta_{n,k})_{k\in\nbN\setminus\{ 0\}},\ n\in\nbZ \}$, with the assumption that $\sup_{k\in\nbN\setminus\{ 0\}}\|\zeta_{n,k} \|_{2\beta}<\infty$. Note that we do not assume necessarily that $(\zeta_{n,k})_{k\in\nbN\setminus\{ 0\}}$ is i.i.d. for fixed $n\in\nbZ$. We suppose that at time $n-1$, the $k$-th individual, from Class 1 contaminate $ \zeta_{n,1}$ individuals of Class 1 if $k=1$ and $ \zeta_{n,k}$ individuals of Class 2 if $k\in\{2,...,\zeta_{n-1,1} \}$.
Soon after contaminating those individuals, they emigrate to our system and we then set
\begin{equation}\label{epsilon_ex_contamination}
\epsilon_n:=\sum_{k=1}^{\zeta_{n-1,1}}\zeta_{n,k},\quad n\in\nbZ .
\end{equation}
An illustration of the evolution of individuals from different classes is provided in Figure \ref{fig:immigrants}: the Class $1$ individuals are represented by crosses, the Class $1$ individuals that contaminate Class $1$ individuals are represented by boxed crosses, while the Class 2 individuals  are represented by dots.
\begin{figure}[hbtp!]
  \centering
  \includegraphics[scale=0.9]{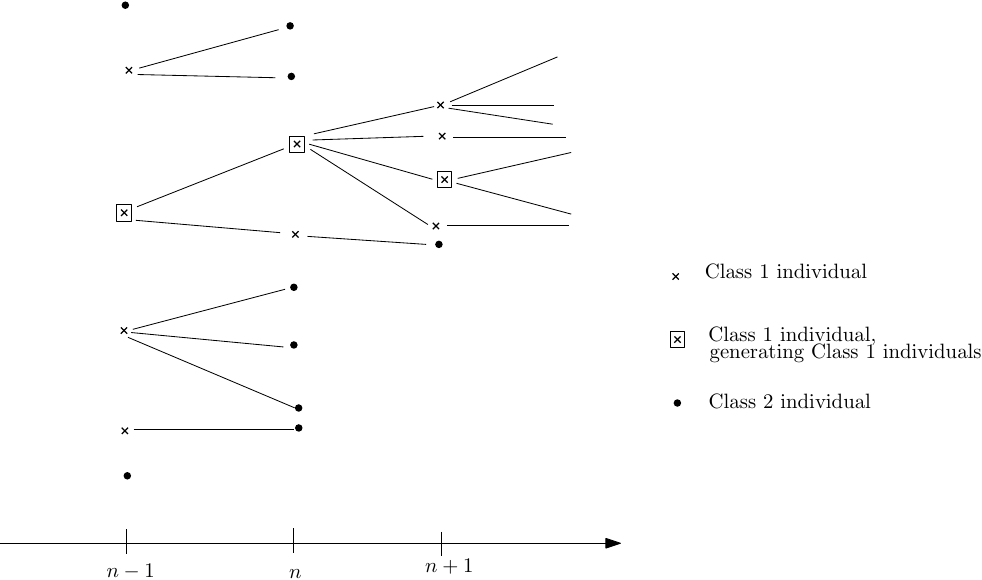}
\caption{Infected immigrants $\{\epsilon_n,\ n\in \nbZ\}$}
\label{fig:immigrants}
\end{figure}
Then, the sequence $\{\epsilon_n,\ n\in \nbZ\}$ given by \eqref{epsilon_ex_contamination} satisfies the $k_0$ dependence relation \eqref{corr1} with $k_0=2$. Also, we note that $\epsilon_n$ in Example \ref{example1} satisfies \eqref{epsilon_ex_contamination} with $ \zeta_{n,k}=Z_n$, which is independent from $k$.
\end{example}
Under the assumption $\mathbf{ (A3)_1}$, it is interesting to notice that $\chi_k$ defined in \eqref{def_chi} verifies $\chi_k=0$ for all $k\ge k_0-1$, so that we obtain from \eqref{expr_C_1_u_k} the following moment expression for the unknown reproduction rate:
\begin{equation}\label{lambda_0_k_0}
\lambda_0 = \frac{C_1^2- u_{k_0-1}}{C_1^2- u_{k_0-2}}.
\end{equation}
For all $k\in \nbN$, we introduce  the following notation:
\begin{multline}\label{def_estimators}
\bar{Y}_n:= \frac{1}{n} \sum_{j=1}^n Y_j,\,\, \bar{Y}_{k,n}:= \frac{1}{n} \sum_{j=1}^n Y_jY_{j+k},\,\, S_n(k):= \sum_{j=1}^n X_j(k),\,\, X_j(k):=(Y_j-C_1, Y_jY_{j+k+1}-u_{k})' \\
T_n(k):= \sum_{j=1}^n {\cal Y}_j(k),\quad {\cal Y}_j(k):=(Y_j-C_1,Y_jY_{j+k-1}-u_{k-2}, Y_jY_{j+k}-u_{k-1})' .
\end{multline}
 
Note that  \eqref{lambda_0_k_0} is the starting point for the estimation of the unknown parameter. 
The following consistance result holds.
\begin{prop}[Consistency]\label{prop_consistency}\normalfont
Assume that $\mathbf{ (A1)}$, $\mathbf{ (A2)}$ and $\mathbf{ (A3)_1}$ hold. The following estimators
\begin{equation}\label{estimator_k_0}
\hat{R}_{k_0,n}:= \frac{\left[ \bar{Y}_n \right]^2- \bar{Y}_{k_0,n}}{\left[ \bar{Y}_n \right]^2- \bar{Y}_{k_0-1,n}},\quad \hat{M}_{k_0,n}:= \bar{Y}_n (1-\hat{R}_{k_0,n})
\end{equation}
converge a.s. towards the unknown parameters $\lambda_0$ and $m_0$ as $n\to\infty$.
\end{prop}
The proof of Proposition \ref{prop_consistency} is postponed to Appendix \ref{sec:consistency_k_0}.
\begin{remi}\label{rem1}\normalfont
Note that, when $k_0=1$, the estimator $\hat{R}_{k_0,n}$ in Proposition \ref{prop_consistency} agrees with those in the case where the immigration process $\{\epsilon_n,\ n\in \nbZ\}$ is i.i.d., see \citet[Expression (5.1)]{KN78}, \citet[Expression (4.1)]{AOA87} and \citet[Section 1]{WW90}.
\end{remi}
The following theorem states the asymptotic distributions of the process ${n}^{-1/2}{T_n(k_0)}$, where we recall that $T_n(k_0)$ is defined in \eqref{def_estimators}.
\begin{theorem}\label{theo_CLT_vector_k_0}\normalfont
The following central limit theorem holds:
\begin{equation}\label{CLT_vector_k_0}
\frac{T_n(k_0)}{\sqrt{n}}=\sqrt{n}\left[\bar{Y}_n-C_1, \bar{Y}_{k_0-1,n}-u_{k_0-2}, \bar{Y}_{k_0,n}-u_{k_0-1}\right]' \xrightarrow[n\to\infty]{\mathcal{D}} {\cal N}(0,\mathfrak{S}_{k_0}),
\end{equation}
for some (non zero) semi-definite positive matrix $\mathfrak{S}_{k_0}\in \nbR^{3\times 3}$ given by
\begin{eqnarray}\label{expression_cov_series}
\mathfrak{S}_{k_0}&:=&\sum_{j=-\infty}^\infty \gamma_j(k_0),\quad \mbox{where}\label{def_frac_sigma}
\end{eqnarray}
\begin{align}
\gamma_j(k_0)&:=\nbE\left({\cal Y}_0(k_0){\cal Y}_j(k_0)'\right)\nonumber\\
&=
\left[
\begin{array}{ccc}
\mbox{Cov}(Y_0,Y_j) & \mbox{Cov}(Y_0, Y_jY_{j+k_0-1}) & \mbox{Cov}(Y_0, Y_jY_{j+k_0})\\
\mbox{Cov}(Y_0, Y_jY_{j+k_0-1}) & \mbox{Cov}(Y_0Y_{k_0-1}, Y_jY_{j+k_0-1}) & \mbox{Cov}(Y_0Y_{k_0-1}, Y_jY_{j+k_0})\\
\mbox{Cov}(Y_0, Y_jY_{j+k_0}) & \mbox{Cov}(Y_0Y_{k_0-1}, Y_jY_{j+k_0}) & \mbox{Cov}(Y_0Y_{k_0}, Y_jY_{j+k_0})
\end{array}
\right]
,\ j\in \nbN.\label{def_gamma}
\end{align}
\end{theorem}
The proof of Theorem \ref{theo_CLT_vector_k_0} is postponed to Appendix \ref{sec:proof_theo_k_0}.

Let us now define the open set ${\cal O}:=\{ (a,b,c) \in \nbR^3 |\ a^2 \neq b\}\subset \nbR^3$ as well as the following function
\begin{equation}\label{def_phi}
\varphi:(a,b,c)\in {\cal O}\mapsto \left( \displaystyle \frac{a^2-c}{a^2-b},\ a\left(1- \displaystyle \frac{a^2-c}{a^2-b}\right) \right)'\in \nbR^2 ,
\end{equation}
so that $(\lambda_0,m_0)'=\varphi(C_1, u_{k_0-2}, u_{k_0-1})$ thanks to Relations \eqref{expr_first_moment} and \eqref{lambda_0_k_0}. 

We are now able to state the following proposition, which shows the asymptotic normality of the estimators proposed in Proposition \ref{prop_consistency}.
The $\delta$-method as well as Theorem \ref{theo_CLT_vector_k_0} thus yield the main following result.
\begin{theorem}[Asymptotic normality]\label{theo_CLT_k_0}\normalfont
Assume that $\mathbf{ (A1)}$, $\mathbf{ (A2)}$ and $\mathbf{ (A3)_1}$ hold. The following central limit theorem holds:
\begin{equation}
\label{CLT_k_0}
\sqrt{n}[\hat{R}_{k_0,n}-\lambda_0, \hat{M}_{k_0,n}-m_0]' \xrightarrow[n\to\infty]{\mathcal{D}} {\cal N}(0,\Omega_{k_0}),
\end{equation}
where $\Omega_{k_0}=  \nabla \varphi(C_1, u_{k_0-2}, u_{k_0-1})' \mathfrak{S}_{k_0}\nabla\varphi(C_1, u_{k_0-2}, u_{k_0-1})$ with $\varphi$ and $\mathfrak{S}_{k_0}$ are defined in \eqref{def_phi} and  in Theorem \ref{theo_CLT_vector_k_0}. 
\end{theorem}
\begin{remi}\label{remNAstrong}\normalfont
In the independent immigration case, \citet[Theorem 2.2]{KN78} show that the asymptotic covariance matrix is  reduced as $$\Omega_{S}:=J^{-1}WJ^{-1},\quad \mbox{where}$$
\[J := \mathbb{E} \left[ \dfrac{\partial g(\theta_0,Y_{n-1})}{\partial \theta} \dfrac{\partial g(\theta_0,Y_{n-1})}{\partial \theta'}\right], \quad
W:= \mathbb{E} \left[\left\{Y_n-g(\theta_0,Y_{n-1}) \right\}^2 \dfrac{\partial g(\theta_0,Y_{n-1})}{\partial \theta} \dfrac{\partial g(\theta_0,Y_{n-1})}{\partial \theta'}\right]\]
with $g(\theta_0,Y_{n-1})=Y_n-\lambda_0Y_{n-1}$ and $\theta_0=(\lambda_0,m_0)'$. See also \citet[Section 4.2]{AOA87}. The true asymptotic covariance matrix $\Omega_{k_0}=  \nabla \varphi(C_1, u_{k_0-2}, u_{k_0-1})' \mathfrak{S}_{k_0}\nabla\varphi(C_1, u_{k_0-2}, u_{k_0-1})$ obtained in the correlated immigration framework can be very different from $\Omega_{S}$ and can yield inacurrate estimation results when used on a model with correlated immigration sequence, as illustrated later on in Section \ref{sec:ultimately_correlated_numerics}.
This is why it is interesting to find an estimator of $\Omega_{k_0}$ which is consistent for both independent  and correlated immigration cases.
\end{remi}
\section{Estimating the asymptotic variance matrix}\label{estimOmega}

For statistical inference problem, the asymptotic variance $\Omega_{k_0}$ has to be estimated. In particular
Theorem \ref{theo_CLT_k_0} can be used to obtain confidence intervals.

The estimation of the long-run variance (LRV) matrix expressed by \eqref{expression_cov_series} and given by
$$\mathfrak{S}_{k_0}=\sum_{j=-\infty}^\infty \mbox{Cov}\left({\cal Y}_0(k_0),{\cal Y}_{j}(k_0)\right)$$ is  complicated. In the literature, two types of estimators are generally employed: heteroskedasticity and autocorrelation consistent (HAC) estimators based on kernel methods (see \cite{newey} and \cite{A91} for general references, and \cite{FZ07} for an application to testing strong linearity in ARMA models with dependent errors) and the spectral density estimators (see e.g. \citet{B74} and \cite{haan} for a general references; see also \cite{BMF11} for an application to a  multivariate ARMA (VARMA) model with dependent errors). In the present paper, we focus on an estimator based on a spectral density form.

 Following the arguments developed in \cite{BMCF12}, the matrix $\mathfrak{S}_{k_0}$ can be estimated using Berk's approach (see \cite{B74}). More precisely, by interpreting $\mathfrak{S}_{k_0}/2\pi$ as the spectral density of the stationary process $\{\mathcal{Y}_n=(\Y_n^1,\Y_n^2,\Y_n^3)',\ n\in\mathbb{Z}\}:=\{\mathcal{Y}_n(k_0),\ n\in\mathbb{Z}\}$ evaluated at frequency $0$, we can use a parametric autoregressive estimate of the spectral density of $(\mathcal{Y}_n)_{n\in\mathbb{Z}}$ in order to estimate the matrix $\mathfrak{S}_{k_0}$. This is explained more precisely in the following.

First, the stationary process $\{\mathcal{Y}_n,\ n\in\mathbb{Z}\}$ is in $\mathbb{L}^2$ hence admits the following Wold decomposition 
\begin{equation}\label{Y_Wold}
\mathcal{Y}_n=\varepsilon_{n}+\sum_{k=1}^{\infty}\Psi_k\varepsilon_{n-k},\quad n\in \nbZ ,
\end{equation}
where $(\varepsilon_{n})_{n\in\mathbb{Z}}$ is a $3-$variate uncorrelated but not independent white noise with variance matrix denoted by $\Sigma_\varepsilon$. We recall e.g. from \citet[Theorem 5.7.1 p.187]{BD91} that $\varepsilon_{n}\in \nbR^{3\times 1}$ and $\Psi_n\in \nbR^{3\times 3}$ are expressed in function of $\{\mathcal{Y}_n,\ n\in\mathbb{Z}\}$ as follows
\begin{equation}\label{Wold_expressions}
\begin{array}{rcl}
\varepsilon_{n} &= & \Y_n - P_{n-1} \Y_n,\quad n\in \nbZ ,\\
\Psi_n & = & \EE(\Y_0 \varepsilon_{-n}')\Sigma_\varepsilon^{-1}=\EE(\Y_n \varepsilon_{0}')\Sigma_\varepsilon^{-1},\quad n\ge 0,
\end{array}
\end{equation}
where $P_{n-1}\Y_n$ is the $\mathbb{L}^2$ projection of the random vector $\Y_n$ onto the vector space spanned by $\Y_k$, $k\le n-1$. It may be verified that $\Sigma_\varepsilon$ is indeed non-singular in the case when the reproduction sequence $\{\xi_{n,k},\ n\in\nbZ,\ k\in \nbN \}$ is non deterministic. 
We next introduce
\begin{equation}\label{def_psi_op}
\Psi(z):= \sum_{k=0}^{\infty}\Psi_k z^k ,
\end{equation}
for $z\in \nbC$ where the series converges, and where $\Psi_0$ is the identity matrix. We suppose in this section that Assumptions $\mathbf{ (A1)}$, $\mathbf{ (A2)_1}$, $\mathbf{ (A3)}$ and $\mathbf{ (A4)}$ hold where 
\begin{itemize}
\item[\hspace*{1em} {$\mathbf{ (A2)_1}$:}]
\hspace*{1em}  The following moment condition holds
$$\| \epsilon \|_{8\vee (2\beta)} < \infty \quad \mbox{and}\quad \|\xi\|_{2\beta}<\infty .$$
\item[\hspace*{1em} {$\mathbf{ (A4)}$:}]
\hspace*{1em}   The sequence $\{ \Psi_k,\ k\ge 1\}$ verifies that $\sum_{k=1}^{\infty}\|\Psi_k \|<\infty$ and
$$\det\left(\Psi(z)\right)\neq 0$$
for all $z\in {\mathbb{C}}$ such that $\left|z\right| \le 1$.
\end{itemize}
Under Assumption $\mathbf{ (A4)}$, $\{\mathcal{Y}_n,\ n\in\mathbb{Z}\}$ admits a  multivariate $\mathrm{AR}(\infty)$ representation (see \citet[Section 2]{A57}) with dependent errors of the form
\begin{equation}\label{AR_infty}
\mathcal{Y}_n-\sum_{k=1}^{\infty}\Phi_k\mathcal{Y}_{n-k}=\varepsilon_{n} ,
\end{equation}
where $(\Phi_k)_{k\ge 1}$ are $ \nbR^{3\times 3}$ matrices such that $\sum_{k=1}^{\infty}\left\|\Phi_k\right\|<\infty$ and $\det\left( \Phi(z)\right) \neq 0$, $\left|z\right|\leq 1$, where 
\begin{equation}\label{def_Phi}
\Phi(z):=I_3 - \sum_{k=1}^{\infty}\Phi_k z^k=\Psi(z)^{-1}.
\end{equation}
In particular, \eqref{Wold_expressions} implies that the $\mathbb{L}^2$ projection $P_{n-1} \Y_n$ is in this case expressed as the series $\sum_{k=1}^{\infty}\Phi_k\mathcal{Y}_{n-k}$. One may argue that Assumption $\mathbf{ (A4)}$ is quite complicated to verify in practice, and it may be informative and useful to give some assumption on the base model \eqref{model} such that it holds. This is provided in the following proposition (proved in Appendix \ref{app:proof_prop_theo_psi}). 
\begin{prop}\label{cond_psi_invertible}\normalfont
Let us define $\underline{\epsilon}_n:=(\epsilon_n, \epsilon_n\epsilon_{n+k_0-1},\epsilon_n \epsilon_{n+k_0})'$, $n\in \nbZ$, and suppose that
\begin{itemize}
\item the spectral density $f_{\underline{\epsilon}}:x\mapsto \frac{1}{2\pi} \sum_{h=-\infty}^\infty \mbox{Cov}(\underline{\epsilon}_0,\underline{\epsilon}_h)e^{ihx}$ is such that $f_{\underline{\epsilon}}(x)$ is invertible for all $x\in [0,2\pi]$.
\item $\mathbf{ (A1)}$ is replaced by the stronger assumption 
\begin{equation}\label{cond_mixing_coef}
\alpha_\epsilon(n)=\mathrm{O}\left(\lambda_0^{\frac{n}{1-2/\beta}} \right).
\end{equation}
\end{itemize}
Then $\mathbf{ (A4)}$ is verified if $\lambda_0$ is small enough, and the sequence $\{ \Phi_k,\ k \ge 1\}$ in the expansion \eqref{AR_infty} verifies $\|\Phi_k\| =\mathrm{o}(1/k^2)$ as $k\to \infty$.
\end{prop}
Note that the proof of the above Proposition is technical and occupies most of Appendix \ref{app:proof_prop_theo_psi}. What is essentially proved is that the spectral density is close to $f_{\underline{\epsilon}}$ when $\lambda_0$ is small enough.

Thanks to the previous remarks, the estimation of $\mathfrak{S}_{k_0}$ in \eqref{def_frac_sigma} is therefore based on the following expression
$$\mathfrak{S}_{k_0}=\Phi^{-1}(1)\Sigma_\varepsilon\Phi^{-1 '}(1)=\left[ I_3 - \sum_{k=1}^{\infty}\Phi_k\right]^{-1}\Sigma_\varepsilon \left[ I_3 - \sum_{k=1}^{\infty}\Phi_k\right]^{-1'}.$$
For $r\in \nbN \setminus \{0\}$, consider the least square regression of $\mathcal{Y}_n$ on $\mathcal{Y}_{n-1},\dots,\mathcal{Y}_{n-r}$, for all $n\in \nbZ$, defined by
\begin{equation}\label{AR_tronquee}
\mathcal{Y}_n=\sum_{k=1}^{r}\Phi_{r,k}\mathcal{Y}_{n-k}+\varepsilon_{r,n},
\end{equation}
where $\varepsilon_{r,n}$ is uncorrelated with $\mathcal{Y}_{n-1},\dots,\mathcal{Y}_{n-r}$.
Since $\mathcal{Y}_n$ is not observable, we introduce $\hat{\mathcal{Y}}_n\in\mathbb{R}^{3}$ obtained by replacing $C_1$ by $\bar Y_n$, $u_{k_0-2}$ by $\bar{Y}_{k_0-1,n}$ and $u_{k_0-1}$ by $\bar{Y}_{k_0,n}$ in \eqref{def_estimators}:
\begin{align}\label{Ht_chapeau}
\hat{{\cal Y}}_n:=(Y_n-\bar Y_n,Y_nY_{n+k_0-1}-\bar{Y}_{k_0-1,n}, Y_nY_{n+k_0}-\bar{Y}_{k_0,n})' \ .
\end{align}
Let us define $\hat{\Phi}_r(z):=I_{3}-\sum_{k=1}^r\hat\Phi_{r,k}z^k$, where $\hat\Phi_{r,1},\dots,\hat\Phi_{r,r}$ denote the coefficients of the least squares regression
of $\hat{{\cal Y}}_n$ on $\hat{{\cal Y}}_{n-1},\dots,\hat{{\cal Y}}_{n-r}$. Let $\hat{\varepsilon}_{r,n}$ be the residuals of this regression so that $\hat{{\cal Y}}_n=\sum_{k=1}^r \hat{\Phi}_{r,k}\hat{{\cal Y}}_{n-k} + \hat{\varepsilon}_{r,n}$, and  let $\hat{\Sigma}_{\hat{\varepsilon}_r}=\frac{1}{n}\sum_{t=1}^n \hat{\varepsilon}_{r,t} \hat{\varepsilon}_{r,t}'$ be the empirical variance of $\hat{\varepsilon}_{r,1},\dots,\hat{\varepsilon}_{r,n}$.

In the case of linear processes with independent innovations, \cite{B74} has shown that the spectral density can be consistently estimated by fitting AR models of order $r=r(n)$, whenever $r$ tends to infinity and $r^3/n$ tends to $0$ as $n$ tends to infinity. Here, there are differences with \cite{B74}: $\{\mathcal{Y}_n,\ n\in\mathbb{Z}\}$ is multivariate, is not directly observed and is replaced by $\{\hat{\mathcal{Y}}_n,\ n\in\mathbb{Z}\}$. It is shown that this result remains valid for the multivariate linear process $\{\mathcal{Y}_n,\ n\in\mathbb{Z}\}$ with non-independent innovations (see for instance  \cite{BMCF12,BMF11}, for references in  (multivariate) ARMA models with dependent errors).

Thus, the asymptotic study of the estimator of $\mathfrak{S}_{k_0}$ using the spectral density method is given in the following theorem.
\begin{theorem}\label{convergence_Isp}\normalfont
Assume that $\mathbf{(A1)}$ , $\mathbf{(A2)_1}$ hold with $\beta=4+2\kappa$ for some $\kappa>0$ and $\mathbf{(A3)}$ hold, and that the process $\{\mathcal{Y}_n,\ n\in\mathbb{Z}\}$ defined in \eqref{def_estimators} verifies $\mathbf{(A4)}$ with $\Sigma_\varepsilon=\var(\varepsilon_{n})$ non-singular. Assume in addition that  the sequence $\{\Phi_k,\ k\ge 1\}$ defined in \eqref{AR_infty} verifies $\|\Phi_k\| =\mathrm{o}(1/k^2)$.
Then, the spectral estimator $\hat{\mathfrak{S}}_{k_0}^{\mathrm{SP}}$ of $\mathfrak{S}_{k_0}$ defined as follows verifies
\begin{equation}\label{convergence_spectral_estimator}
\hat{\mathfrak{S}}_{k_0}^{\mathrm{SP}}:=\hat{\Phi}_r^{-1}(1)\hat{\Sigma}_{\hat{\varepsilon}_r}\hat{\Phi}_r^{-1 '}(1)\xrightarrow[n\to\infty]{\mathbb{P}} \mathfrak{S}_{k_0}=\Phi^{-1}(1)\Sigma_\varepsilon\Phi^{-1'}(1)
\end{equation}
when $r=r(n)\to\infty$ and $r^{3}/n\to0$ as $n\to\infty$.
\end{theorem}
As explained previously, Proposition \ref{cond_psi_invertible} gives some conditions such that $\mathbf{(A4)}$ and $\|\Phi_k\| =\mathrm{o}(1/k^2)$ are  verified. Its proof occupies a large bulk of Appendix \ref{app:proof_prop_theo_psi}. The proof of Theorem \ref{convergence_Isp} is similar to that given by \citet[Section 3.3.1]{BMCF12} and is given in the same Appendix. 

Let $\nabla \varphi(\bar Y_n, \bar{Y}_{k_0-1,n}, \bar{Y}_{k_0,n})$ be a consistent estimator of $\nabla \varphi(C_1, u_{k_0-2}, u_{k_0-1})$, where $\varphi$ is the function defined by \eqref{def_phi}. 
Theorem~\ref{convergence_Isp} thus implies that
\begin{align}
\label{estThetaSP}
  \hat{\Omega}_{k_0}^{\mathrm{SP}}&:=\left(\nabla \varphi(\bar Y_n, \bar{Y}_{k_0-1,n}, \bar{Y}_{k_0,n})\right)' \hat{\mathfrak{S}}_{k_0}^{\mathrm{SP}}\left(\nabla \varphi(\bar Y_n, \bar{Y}_{k_0-1,n}, \bar{Y}_{k_0,n})\right)
\end{align}
is a weakly consistent estimator of ${\Omega}_{k_0}$. Some numerical illustrations are presented in Section \ref{sec:ultimately_correlated_numerics}.

\section{General correlated immigration}\label{sec:general}
We suppose throughout this section that $\mathbf{ (A1)}$, $\mathbf{ (A2)}$  and $\mathbf{ (A3)}$ hold, as well as the following additional Assumptions.
\begin{itemize}
\item[\hspace*{1em} $\mathbf{ (A5)}$:]
\hspace*{1em}  The unknown parameters $\lambda_0$, $m_0$, $V_{0,\xi}$ and $V_{0,\epsilon}$ (reminding that the two latter are the variances of $\xi$ and $\epsilon$) belong to $\Theta:= [\lambda_-,\lambda_+]\times \Theta_m\times \Theta_{V_{\xi}}\times \Theta_{V_{\epsilon}}$ for some known respective intervals $[\lambda_-,\lambda_+]$ and compact intervals $\Theta_m$, $\Theta_{V_{\xi}}$ and $\Theta_{V_{\epsilon}}$ included in $[0,\infty)$, with $0<\lambda_-<\lambda_+<1$, and the generating function $f_\nu:x\in [0,1]\mapsto \sum_{s=0}^\infty x^s \nu_{s+1}$ belongs to some known class of function ${\cal F}$.
\end{itemize}
\begin{itemize}
\item[\hspace*{1em} $\mathbf{ (A6)}$:]
\hspace*{1em} There exists a known quantity $K_m>0$ such that
\begin{equation}\label{def_Km}
\inf_{(\lambda,\mu,V_{\xi},V_{\epsilon})\in \Theta,\ f\in {\cal F}}\left| \Xi(\lambda,\mu, V_{\xi}, V_{\epsilon}, f) \right| \ge K_m
\end{equation}
where $\Xi$ is the function defined on $\Theta \times {\cal F}$ by
\begin{equation}\label{def_Xi}
\Xi(\lambda,\mu,V_{\xi},V_{\epsilon}, f):=
-  \frac{V_{\xi} m_0 \lambda}{(1-\lambda^2)(1+\lambda)} - \frac{ V_{\epsilon} \lambda}{1-\lambda^2} - \frac{\lambda^2}{1-\lambda^2} f (\lambda) - \frac{1}{1-\lambda^2} f(\lambda^{-1})
\end{equation}
where $(\lambda,\mu,V_{\xi},V_{\epsilon}, f)\in \Theta\times  {\cal F}$.
\end{itemize}
\begin{itemize}
\item[\hspace*{1em} $\mathbf{ (A7)}$:]
\hspace*{1em}  There exists a known constant $C_Y>0$ such that the $k$-th moments $\|Y_0\|_k$ of the stationary process, $k=1,2$ are less than $C_Y$.
\end{itemize}
We give a few comments on the two last assumptions. Assumption $\mathbf{ (A6)}$ may at first sight look artificial, but is actually natural for the following reason. First note that the function $\Xi$ in \eqref{def_Xi} verifies from \eqref{expr_f_nu} that $\Xi(\lambda_0,m_0,V_{0,\xi},V_{0,\epsilon},f_\nu)=\sum_{j=0}^\infty \lambda_0^{-j} \chi_j + \lambda_0(C_1^2-C_2)$ is non zero (a condition which will turn out to be necessary in the forthcoming Theorem \ref{main_theo_genera2}), and that it is still non zero in the neighborhood $\Theta\times {\cal F}$ of the unknown parameters $(\lambda_0,m_0,V_{0,\xi},V_{0,\epsilon},f_\nu)$. So that $\mathbf{ (A6)}$ is relevant because, unless in very specific circumstances, 
$\sum_{j=0}^\infty \lambda_0^{-j} \chi_j + \lambda_0(C_1^2-C_2)$ is not likely to be exactly zero. This is verified for example in the following situations:
\begin{itemize}
\item When $\cal F$ is the set of power series with non negative coefficents, meaning that the immigration sequence $\{\eps_n,\ n\in\nbZ \}$ is positively correlated, i.e. $\nu_n \ge 0$ for all $n\ge 1$. This is the case when one models the evolution of a disease with constant (increasing or decreasing) trend. In that case we have that $f(\lambda)$ and $f(\lambda^{-1})$ are non negative, so that we have from \eqref{def_Km} that
$$
\left| \Xi(\lambda,\mu, V_{\xi}, V_{\epsilon}, f) \right| \ge K_m:= \inf \Theta_{V_{\xi}} . \inf \Theta_m .\frac{\lambda_-}{(1-\lambda_-^2)(1+\lambda_+)} + \inf \Theta_{V_{\epsilon}}\frac{\lambda_-}{1-\lambda_-^2}.
$$
\item When $\cal F$ is included in the set of functions bounded by some constant $M_{\cal F}$, in which case we may set
$$
K_m:= \inf \Theta_{V_{\xi}} . \inf \Theta_m .\frac{\lambda_-}{(1-\lambda_-^2)(1+\lambda_+)} + \inf \Theta_{V_{\epsilon}}\frac{\lambda_-}{1-\lambda_-^2}- M_{\cal F}\left[ \frac{\lambda_+^2}{1-\lambda_+^2}+ \frac{1}{1-\lambda_+^2}\right],
$$
which is positive e.g. if the bound $M_{\cal F}$ is small enough.
\end{itemize}
As to $\mathbf{ (A7)}$, note that the explicit upper bound $C_Y$ in $\mathbf{(A7)}$ may be obtained from \citet[Section 3.1]{KW21} under the assumptions that $\|\epsilon\|_2$ and $\|\xi\|_2$ are finite, so that $\mathbf{(A7)}$ is not too restrictive. The expression of the bound obtained in that latter paper may not be simple, which is why one simple example where $C_Y$ has a nice expression may be illuminating. This is given in the following example:
\begin{example}\label{lem_bounds_moments}\normalfont
Assume that there exists known constants $M_\epsilon>0$ and $M_\xi \in (0,1)$ such that $\|\epsilon\|_k=[\nbE(|\epsilon|^k)]^{1/k} \le M_\epsilon$ and $\|\xi\|_k=[\nbE(|\xi|^k)]^{1/k} \le M_\xi$, $k=1,2$. Let us then write that
\begin{equation}\label{motivate_example}
\|Y_0\|_k = \left\| \sum_{j=1}^{Y_{-1}} \xi_{0,j}+\eps_0\right\|_k \le \left\| \sum_{j=1}^{Y_{-1}} \xi_{0,j}\right\|_k+ \left\| \eps_0\right\|_k .
\end{equation}
Thanks to the independence of $Y_{-1}$ from the $\xi_{0,j}$'s, $j\ge 1$:
\begin{eqnarray*}
\left\| \sum_{j=1}^{Y_{-1}} \xi_{0,j}\right\|_k^k &=&\nbE\left( \left[ \sum_{j=1}^{Y_{-1}} \xi_{0,j}\right]^k\right)= \sum_{n=0}^\infty \left\| \sum_{j=1}^{n} \xi_{0,j}\right\|_k^k \nbP(Y_{-1}=n)\\
&\le  &\sum_{n=0}^\infty n^k \|\xi\|_k^k \; \nbP(Y_{-1}=n) = \|\xi\|_k^k \;  \|Y_{-1}\|_k^k = \|\xi\|_k^k \;  \|Y_{0}\|_k^k
\end{eqnarray*}
which, plugged into \eqref{motivate_example} yields that the stationary process $\{Y_n,\ n\in\nbZ \}$ admits the following $k$-th moment upper bounds
$$
\|Y_0\|_k \le \frac{\|\epsilon\|_k}{1-\|\xi\|_k}\le \frac{M_\epsilon}{1-M_\xi},\quad k=1,2.
$$
Here, the fact that we impose for $M_\xi$ to be less than $1$ may look demanding, but is satisfied for instance if $\xi$ is zero with sufficiently high probability.
\end{example}
In order to state the main result of this section, we first introduce some auxiliary functions that will enable us to construct the estimator for the unknown parameters $\lambda_0$ and $m_0$. We first let $\varpi: \R \longrightarrow \R$ be a twice differentiable function such that 
\begin{eqnarray}
\varpi(x)& & \left\{
\begin{array}{cl}
=1, & x\in [1,\infty),\\
=0, & x\in (-\infty ,0],\\
\in [0,1], & x\in [0,1],
\end{array}
\right. \label{def_varpi}\\
\varpi''(x) &= & \mathrm{o}(x),\quad x\to 0, \label{equiv_varpi}
\end{eqnarray}
so that $\varpi'(x)=\mathrm{o}(x^2)$ and $\varpi(x)=\mathrm{o}(x^3)$ as $x\to 0$. 

We also need a twice differentiable function $G:\R\longrightarrow \R$ with finite support such that $G(x)=1$ for $x\in \left[0, \max \left( C_Y,C_Y^2\right) \right]$. We then let
\begin{equation}\label{def_chi_Km}
\varpi_k:x\in\R \mapsto \varpi \left( \frac{2}{K_m \lambda_-^k}x\right),\quad k\in \N ,
\end{equation}
$K_m$ being defined in \eqref{def_Km}, so that $\varpi_k(x)$ is equal to $1$ on $\left[ 2^{-1}{K_m \lambda_-^k},\infty\right)$, lies in $[0,1]$ on $\left[0, {2^{-1}K_m \lambda_-^k}\right]$ and is $0$ on $(-\infty,0]$. We finally define for all $k\in \N\setminus \{0\}$
\begin{eqnarray}
H_k: x\in \R & \mapsto & \frac{1}{k} \varpi_k(|x|) \ln |x| ,\label{def_H_k}\\
\psi_k: (a,b)\in \R^2 &\mapsto & G(a) G(b) H_k(a^2-b).\label{def_psi_k}
\end{eqnarray}
Now, to motivate the expression of the upcoming estimator we first make a few comments. Since $\sum_{j=0}^\infty \lambda_0^{-j} \chi_j$ is a convergent series (an easy consequence of the definition \eqref{def_chi} of $\chi_k$ and the convergence of the series in assumption ${\mathbf{(A2)}}$), and since we saw that $\sum_{j=0}^\infty \lambda_0^{-j} \chi_j + \lambda_0(C_1^2-C_2)$ is not zero thanks in particular to $\mathbf{ (A6)}$, the first crucial observation that we may make from \eqref{expr_C_1_u_k} is that
\begin{equation}\label{observation_S_k}
S_k:= \frac{1}{k}\ln |C_1^2-u_k| \longrightarrow \ln (\lambda_0),\quad k\to \infty .
\end{equation}
This limit brings us to two deductions. First, w.l.o.g. we are going to estimate the unknown parameter $\ln \lambda_0$ rather than $\lambda_0$. Second, since for fixed $k\in \N$, $C_1$ and $u_k$ are respectively estimated by
\begin{equation}\label{def_estim_gen1}
\bar{Y}_n:= \frac{1}{n} \sum_{i=1}^n Y_i,\quad \bar{Y}_{k+1,n}:=\frac{1}{n} \sum_{i=1}^n Y_i Y_{i+k+1},
\end{equation}
then it is natural to let $k=k_n$ in \eqref{observation_S_k} for an integer valued sequence $(k_n)_{n\in\N}$ such that $\lim_{n\to \infty} k_n=\infty$,  and introduce an estimator where $C_1$ and $u_{k_n}$ are respectively replaced by $\bar{Y}_n$ and $\bar{Y}_{k_n+1,n}$.

We now arrive at the definition of the estimator. As previously said, and in view of \eqref{observation_S_k}, a first idea that springs to mind is to consider the following estimator:
\begin{equation}\label{intuitive_estimator}
\hat{T}_n:=\frac{1}{k_n}\ln \left| [\bar{Y}_n]^2 -  \bar{Y}_{k_n+1,n}\right|
\end{equation}
for a sequence $(k_n)_{n\in\N}$ with $\lim_{n\to \infty} k_n=\infty$. Equation \eqref{intuitive_estimator} has however two drawbacks. First, by a very crude approximation we have that $[\bar{Y}_n]^2 -  \bar{Y}_{k_n+1,n}\approx C_1^2- u_{k_n}$ which tends to $0$ as $n\to \infty$: this observation foretells that $[\bar{Y}_n]^2 -  \bar{Y}_{k_n+1,n}$ can be arbitrarily too close to $0$ with possibly large probability as $n$ grows large, so that $\ln \left| [\bar{Y}_n]^2 -  \bar{Y}_{k_n+1,n}\right|$ is difficult to control and may e.g. tend to $ \infty$ faster or more slowly than $k_n$. Second, $[\bar{Y}_n]^2 -  \bar{Y}_{k_n+1,n}$ may also be arbitrarily large. Both those observations imply some potential "bad" behaviour of the estimator defined in \eqref{intuitive_estimator}, namely that it can grow large as $n\to\infty$, which will induce undesirable statistical properties. This justifies to rather consider the following regularized estimator
\begin{multline}\label{def_estimator_general}
\hat{S}_n:=\psi_{k_n}(\bar{Y}_n, \bar{Y}_{k_n+1,n})= G(\bar{Y}_n) G(\bar{Y}_{k_n+1,n}) H_{k_n} \left( \left| [\bar{Y}_n]^2 -  \bar{Y}_{k_n+1,n}\right| \right)\\
=G(\bar{Y}_n) G(\bar{Y}_{k_n+1,n})  \varpi_{k_n}\left( \left| [\bar{Y}_n]^2 -  \bar{Y}_{k_n+1,n}\right| \right) \frac{1}{k_n} \ln \left| [\bar{Y}_n]^2 -  \bar{Y}_{k_n+1,n}\right| ,\quad n\in \N ,
\end{multline}
for a well chosen sequence $(k_n)_{n\in\N}$ specified later, where $\psi_k$ is defined in \eqref{def_psi_k}. One sees that $\hat{S}_n$ resembles \eqref{intuitive_estimator}, but the behaviour of $[\bar{Y}_n]^2 -  \bar{Y}_{k_n+1,n}$ at $0$ and infinity is "smoothed out" thanks to the functions $\varpi_{k_n}(\cdot)$ and $G(\cdot)$, see \eqref{def_psi_k}. The role of the function $\psi_{k_n}$ in the definition \eqref{def_estimator_general} is motivated as follows. The terms $G(\bar{Y}_n)$ and  $G(\bar{Y}_{k_n+1,n})$ prevent $\hat{S}_n$ from being too large if $\bar{Y}_n$ or $\bar{Y}_{k_n+1,n}$ are large. The term $H_{k_n} \left( \left| [\bar{Y}_n]^2 -  \bar{Y}_{k_n+1,n}\right| \right)$ is such that it behaves like ${k^{-1}_n} \ln \left| [\bar{Y}_n]^2 -  \bar{Y}_{k_n+1,n}\right| $ when $\left| [\bar{Y}_n]^2 -  \bar{Y}_{k_n+1,n}\right| $ is far enough from $0$. When the latter is close to $0$, the definition \eqref{def_chi_Km} for $\varpi_{k_n}$ entails that $\varpi_{k_n}\left( \left| [\bar{Y}_n]^2 -  \bar{Y}_{k_n+1,n}\right| \right)$ tends to $\infty$ with the same order of magnitude as $k_n$, again preventing $\hat{S}_n$ from being too large.

The main results of this section show the convergence in $\mathbb{L}^2$ of the estimator towards the unknown parameter and gives an expansion for $\hat{S}_n$ for a suitable choice for $k_n$.
\begin{theorem}[Quadratic convergence of estimator]\label{main_theo_general1}\normalfont
We suppose that $\mathbf{ (A1)}$, $\mathbf{ (A2)}$, $\mathbf{ (A3)}$, $\mathbf{ (A5)}$, $\mathbf{ (A6)}$ and $\mathbf{ (A7)}$ hold. Let us set $k_n:=\lfloor c \ln n\rfloor$ where $c<-{1}/({2 \ln \lambda_-})$. The following convergence in quadratic mean holds:
\begin{equation}\label{conv_gen_estimator_quad}
\left| \left| \hat{S}_n - \ln \lambda_0\right| \right|_2=\mathrm{O}\left( \frac{1}{\ln n}\right) \longrightarrow 0,\quad n\to \infty 
\end{equation}
so that, in particular, $e^{\hat{S}_n}$ converges in probability towards $\lambda_0$ as $n\to\infty$. Furthermore, the estimator defined by $\hat{N}_n:=\bar{Y}_n \left(1-e^{\hat{S}_n}\right)$ converges in probability towards $m_0$ as $n\to \infty$.
\end{theorem}
The proof of Theorem \ref{main_theo_general1} is postponed to Appendix \ref{proof_gen1}.
\begin{remi}\normalfont
The sequence $(k_n)_{n\in \N}$ grows like $\ln n$ in the previous theorem. It may be (wrongly) guessed that such a growth is inconsequential and that any sequence satisfying $k_n=\mathrm{o}(n)$ could guarantee the convergence of the estimator $\hat{S}_n$. In fact, it turns out that this growth rate is not only significant but actually careful chosen, as it is illustrated numerically in Section \ref{sec:numerics} that the estimator performs very badly if $k_n$ grows for example like $\sqrt{n}$ instead.
\end{remi}
Last, a finer behaviour for the estimator can be described when the covariance of the immigration process decreases fast enough. This is explained in the following result.
\begin{theorem}[Expansion for $\hat{S}_n$]\label{main_theo_genera2}\normalfont
We suppose that $\mathbf{ (A1)}$, $\mathbf{ (A2)}$, $\mathbf{ (A3)}$, $\mathbf{ (A5)}$, $\mathbf{ (A6)}$ and $\mathbf{ (A7)}$ hold. Let us also assume here that the covariance of the immigration process satisfies $\nu_h= \mathrm{O}(\zeta^h)$ for some $\zeta <\lambda_-$ (ensuring that $\mathbf{(A2)}$ holds). Let us set $k_n:=\lfloor c \ln n\rfloor$ where $c\in \left( -{1}/{2 \ln \zeta}, -{1}/{2 \ln \lambda_-}\right)$. Then one has the two terms expansion
\begin{equation}\label{expansion_Sn}
\hat{S}_n- \ln \lambda_0=\frac{1}{k_n} \ln \left| \sum_{j=0}^\infty \lambda_0^{-j} \chi_j + \lambda_0(C_1^2-C_2)\right| + \frac{1}{\sqrt{n}k_n \lambda_0^{k_n}} Z_n
\end{equation}
where $Z_n$ satisfies $Z_n \stackrel{\cal D}{\longrightarrow} {\cal N}(0,\sigma)$, $n\to\infty$, with
\begin{eqnarray}
\sigma &:=& V
 \mathfrak{S}V',\label{def_sigma_gen}\\
 V&:=& \left( \frac{2C_1}{ \left|\sum_{j=0}^{\infty} \lambda_0^{-j} \chi_j + \lambda_0(C_1^2-C_2)\right|},
- \frac{1}{ \left|\sum_{j=0}^{\infty} \lambda_0^{-j} \chi_j + \lambda_0(C_1^2-C_2)\right|}
\right)',\label{def_sigma_gen_V}
\end{eqnarray}
and where $\mathfrak{S}$ is given by
\begin{equation}\label{expression_big_Sigma}
\mathfrak{S}= \left[
\begin{array}{cc}
2 \sum_{h=1}^\infty (C_1^2 - u_{h-1}) &   
\begin{array}{c}
-2 \frac{m_0}{(1-\lambda_0)^2}  \sum_{h=0}^\infty \chi_h \\
+  \frac{C_1}{1-\lambda_0}(C_2-C_1^2) \\
+3 C_1 \sum_{h=1}^{\infty } (u_{h-1}-C_1^2)
\end{array}
        \\
    &    \\
\begin{array}{c}
-2 \frac{m_0}{(1-\lambda_0)^2}  \sum_{h=0}^\infty \chi_h \\
+  \frac{C_1}{1-\lambda_0}(C_2-C_1^2) \\
+3 C_1 \sum_{h=1}^{\infty } (u_{h-1}-C_1^2)
\end{array}
 & 
\begin{array}{c}
C_2^2 - C_1^4 + 2 \sum_{h=1}^{\infty} (u_{h-1}^2 -C_1^4)\\
 +  2 C_1^2 \sum_{h=0}^\infty (u_{h-1}- C_1^2)\\
 - 2 \frac{C_1 m_0}{(1-\lambda_0)^2}\sum_{t=1}^{\infty} \chi_{t-1}
\end{array}
\end{array}
\right].
\end{equation}
\end{theorem}
The proof of Theorem \ref{main_theo_genera2} is postponed to Appendix \ref{proof_gen2}.

It seems that such expansions as \eqref{expansion_Sn} are quite rare in the literature and exist in the case of regression models, see convergence (3.6) in \citet[Theorem 2]{TP96} for a corrected central limit theorem for an estimator in a linear model. Also note that the normalization factor in the central limit theorem involved in \eqref{expansion_Sn} is $\sqrt{n}k_n \lambda_0^{k_n}$ i.e. in the form $ n^\delta\ln n$ with $\delta= 1/2 + c \ln \lambda_0$, which is different from the classical re-normalization in $\sqrt{n}$. 

Also, a byproduct of \eqref{expansion_Sn} is the following refinement of \eqref{conv_gen_estimator_quad} in Theorem \ref{main_theo_general1}: 
\begin{equation}\label{2_terms}
\hat{S}_n- \ln \lambda_0=\frac{1}{k_n} \ln \left| \sum_{j=0}^\infty \lambda_0^{-j} \chi_j + \lambda_0(C_1^2-C_2)\right| + \mathrm{O}_{\mathbb{L}^2}\left(\frac{1}{\sqrt{n}k_n \lambda_0^{k_n}}\right),
\end{equation}
as indeed it is shown in the proof of Theorem \ref{main_theo_genera2} that $||Z_n||_2$ converges as $n\to \infty$ to some finite quantity. This makes more precise the speed of convergence of $\hat{S}_n$ towards $ \ln \lambda_0$, as \eqref{2_terms} implies that 
\begin{equation}\label{slow_convergence}
\left| \left| \hat{S}_n - \ln \lambda_0\right| \right|_2\propto \frac{1}{\ln(n)},
\end{equation}
i.e. that the speed of convergence is {\it exactly} of the order ${1}/{\ln(n)}$ here, and not just at most of order ${1}/{\ln(n)}$ as shown in \eqref{conv_gen_estimator_quad}. Convergence  is thus very slow, which in particular explains why we need to pick very large $n$ in the numerical illustrations in the forthcoming Section \ref{sec:correlated}.

Theorems \ref{main_theo_general1} and \ref{main_theo_genera2} will heavily rely on the two following technical results, proved respectively in Appendices \ref{subsec:proof_gradient} and \ref{sec:proof_prop_main_theo2}.
\begin{prop}\label{prop_main_theo1}\normalfont
The following estimate holds for all $k$ and $n$ in $\N$:
\begin{equation}\label{bound_S_k_n}
||\psi_{k}(\bar{Y}_n, \bar{Y}_{k+1,n})- \ln \lambda_0||_2\le K_S \left( \frac{1}{ \lambda_-^k} \left|\left| \frac{S_n(k)}{n}\right|\right|_2 +  \frac{1}{k}\right)
\end{equation}
for some constant $K_S>0$, where we recall that $S_n(k)$ is defined in \eqref{def_estimators}.
\end{prop}

\begin{prop}\label{prop_main_theo2}\normalfont
Let the sequence $(k_n)_{n\in \N}$ be such that $k_n=\mathrm{o}(n)$ as $n\to\infty$.
Then the following convergence holds
\begin{equation}
\frac{1}{n} \nbE(S_n(k_n) S_n(k_n)')  \longrightarrow  \mathfrak{S} \label{CLT_general_Sigma}
\end{equation}
as $n\to \infty$, where we recall that $S_n(\cdot)$ is defined in \eqref{def_estimators}, and where $\mathfrak{S}$ is given by \eqref{expression_big_Sigma}. If furthermore the sequence $(k_n)_{n\in \N}$ verifies
\begin{equation}\label{Cond_CLT_k_n}
k_n=\mathrm{O}(n^\kappa),\ n\to \infty \mbox{ with }\kappa\in \left(0,1\wedge \frac{\beta-1}{2(2\beta-1)}\right),
\end{equation}
then the following central limit theorem holds
\begin{equation}
\frac{S_n(k_n)}{\sqrt{n}}\stackrel{\cal D}{\longrightarrow} {\cal N}(0, \mathfrak{S}),\quad n\to \infty . \label{CLT_general}
\end{equation}
\end{prop}
Finally, the proofs of Theorem \ref{main_theo_general1} and Theorem \ref{main_theo_genera2} are resppectively presented in Appendices \ref{proof_gen1} and \ref{proof_gen2}.

\section{Numerical illustrations}\label{sec:numerics}
%

\subsection{Ultimately correlated case}\label{sec:ultimately_correlated_numerics}
We  study numerically the behavior of  the moment estimator for independent and correlated immigration model \eqref{model} on $N=1000$ independent simulated trajectories to assess the performance of the  moment estimator in finite sample. 
 We generated from model \eqref{model} three trajectories of length $n=300$, $n=2000$  and $n=5000$, in which  the reproduction sequence $\{\xi_{n,k},\ n\in\nbZ,\ k\in \nbN\}$  is taken to be:
 \begin{itemize}
 \item[1)] random variables ${\cal P}(\lambda_0)$ distributed,
 \item[2)] Bernoulli random variables with expectation $\lambda_0$.
 \end{itemize}
 The immigration process $\{\epsilon_n,\ n\in\nbZ \}$ is taken to be:
 \begin{itemize}
 \item[1)] ${\cal P}(1)$ distributed (for the independent case),
 \item[2)] of the form $\epsilon_n=\prod_{i=0}^{k_0-1} Z_{n-i}$ where $\{Z_n,\ n\in \nbZ \}$ is i.i.d. with ${\cal P}(a_0)$ distribution (for the correlated case), for $k_0=2, 3$. In view of \eqref{corr1}, we have:
\begin{equation*}
\mbox{Cov}\left(\epsilon_{n},\epsilon_{n-h}\right)=\left\{\begin{array}{ccc}
a_0^{2h}(a_0+a_0^2)^{k_0-h}- a_0^{k_0}  & \mbox{if}  & h\leq k_0-1 .\\
0 & \mbox{if} &h\geq k_0
\end{array}\right.
\end{equation*}
 \end{itemize}
We take $a_0=1$ in our numerical illustrations.
Table \ref{tab} summarizes the distribution of the moment estimator $\hat{R}_{k_0,n}$ of $\lambda_0$ over these simulations experiments.
As expected, Table \ref{tab} shows that the bias and the RMSE decrease when the size of the sample increases. 
From Table \ref{tab}, the estimation of $\lambda_0$ becomes more difficult when $k_0$ increases.  We will now numerically evaluate the potential effects of underestimation or overestimation of the instant $k_0$ (see Assumption $(\textbf{A3})_1$). Table \ref{tab1}  summarizes the  distribution of the moment estimators $\hat{R}_{1,n}$ and $\hat{R}_{3,n}$ of $\lambda_0$ for the  model \eqref{model} when $\epsilon_n= Z_{n}Z_{n-1}$ i.e. $k_0=2$ with $Z_n \sim IID{\cal P}(1)$. As expected $\hat{R}_{1,n}$ is inconsistent and the bias of the estimation  increases with increasing $n$. In contrast, when $n$ increases, the distributions of $\hat{R}_{3,n}$ are close to those of Table \ref{tab} (see $k_0=2$) except for $\lambda_0=0.2$.
From this example we draw the conclusion that, in the case of an overestimation of the parameter $k_0$ the proposed estimator $\hat{R}_{k_0,n}$ of $\lambda_0$ remain asymptotically consistent. On the other hand, if the instant $k_0$ is underestimated, the proposed estimator $\hat{R}_{k_0,n}$ is inconsistent even as $n$ increases.
\\
 Figures~\ref{fig:param},~\ref{fig:error},~\ref{fig:qq},~\ref{fig:qq1},~\ref{fig:hist} and~\ref{fig:hist1}  summarize via box-plots and  compare the distribution of the proposed moment estimators and those proposed by \cite{KN78} (see also \cite{AOA87}) in the independent  and correlated cases.  For these simulations the distributions of the proposed moment estimators and those proposed by \cite{KN78}  are similar in the independent case (see the left panels of Figure~\ref{fig:param}), whereas the distributions of our proposed estimators $\hat{R}_{k_0,n}$ and $\hat{M}_{k_0,n}$ are more accurate in the correlated case than those proposed by \cite{KN78} (see the right panels of Figures~\ref{fig:param} and~\ref{fig:error}). Under non-independent immigration, it appears that the estimators proposed by  \cite{KN78} (see also \cite{AOA87})  are unreliable (see the right-bottom panel of Figures~\ref{fig:param} and~\ref{fig:error}). See also the right-bottom panel of Figures~\ref{fig:hist} and~\ref{fig:hist1} for which the  estimation of the centered Gaussian density with the same variance plotted in dotted line is very poorly approximated because the fact that the  estimators proposed by \cite{KN78}  is very biased ($\text{bias}(\lambda_0)=0.165$ and $\text{bias}(m_0)=-0.331$) in the correlated case.
 For these simulations we also observe that:  the precision around the estimated coefficients is better when the size of the sample increases and the distributions of $\hat{R}_{k_0,n}$ and $\hat{M}_{k_0,n}$ are more accurate in the independent case than in the correlated one. This is in accordance with the results of \cite{RT96} who showed that, with similar $(k_0-1)$-dependent noises, the asymptotic variance of the sample autocorrelations can be greater than 1 as well (1 is the asymptotic variance for independent white noises).

Figure~\ref{fig:var} compares the standard estimator $\hat{\Omega}_S$ of $\Omega_S$ in Remark \ref{remNAstrong} by \citet{KN78} with the proposed  estimator based on spectral density estimation $\hat{\Omega}_{k_0}^\mathrm{{SP}}$ of the asymptotic variance ${\Omega}_{k_0}$ in \eqref{estThetaSP}. We used the spectral estimator defined as \eqref{convergence_spectral_estimator} in Theorem~\ref{convergence_Isp}. The AR order $r$ in this theorem is taken equal to $\lfloor n^{(1/3)-\verb".Machine$double.eps"}\rfloor$ where \verb".Machine$double.eps" is the smallest positive floating-point number $x$ such that $n+x\neq n$. The order $r$ can also be automatically selected by AIC or BIC, using the function \verb"VARselect()" of the \textit{vars} R package).
In the case  of independent immigration we know that the two estimators are consistent. In view of the 
 left-top and the left-bottom panels of Figure~\ref{fig:var}, it seems that the standard estimator is most accurate than the proposed  estimators in the independent case. This is not surprising because the spectral estimator is more robust, in the sense that this estimator continues to be consistent in the correlated case, contrary to the standard estimator. It is clear that in the correlated case $n \mbox{Var}(\hat{R}_{k_0,n}-\lambda_0)^2$ and $n \mbox{Var}(\hat{M}_{k_0,n}-m_0)^2$ are better estimated by $\hat{\Omega}_{k_0}^\mathrm{{SP}}(1,1)$ and  $\hat{\Omega}_{k_0}^\mathrm{{SP}}(2,2)$  (see the box-plots $1$ and $2$ of the right-bottom panel of Figure~\ref{fig:var}) than by $\hat{\Omega}_{S}(1,1)$ and $\hat{\Omega}_{S}(2,2)$ (see the box-plots 1 and $2$ of the right-top panel). The failure of the standard estimator of ${\Omega}_{k_0}$ in the correlated immigration model setting may have important consequences in terms of statistical inference.

\begin{table}[H]
 \caption{\small{Sampling distribution of the moment estimator $\hat{R}_{k_0,n}$ of $\lambda_0$ for the 
  model \eqref{model} with $\epsilon_n=\prod_{i=0}^{k_0-1} Z_{n-i}$ where $Z_n \sim IID{\cal P}(1)$. }}
{\scriptsize
\begin{center}
\begin{tabular}{lll rrrr rrrr}
\hline\hline
&  &   & \multicolumn{4}{c}{$\xi \sim {\cal P}(\lambda_0)$}& \multicolumn{4}{c}{$\xi \sim {\cal B}(\lambda_0)$} \\
& &  & \multicolumn{4}{c}{$\lambda_0$}&  \multicolumn{4}{c}{$\lambda_0$}
\\
&&&$0.2$&$0.5$&$0.7$&$0.9$
$\quad$& $0.2$&$0.5$&$0.7$&$0.9$
\vspace*{0.1cm}\\
&& Biais &-0.006 &-0.033 &-0.029& -0.029 $\quad$&-0.007& -0.027& -0.027 &-0.034 \vspace*{0.1cm}\\
&& RMSE& 0.108 &0.103 &0.075& 0.049 $\quad$& 0.110 &0.090 &0.066 &0.051 
\vspace*{0.1cm}\\
 && $\hat{\Omega}_{k_0}^{\mathrm{SP}}(1,1)$&4.582 &1.938 &1.006 &0.293 $\quad$&4.527 &2.261 &0.833 &0.301\vspace*{0.1cm}
\\
&& Min&0.001 & 0.116 &0.406 &0.694 $\quad$& 0.002 &0.220 &0.463 &0.710  \vspace*{0.1cm}\\
 $k_0=2$&$n=300$ & $Q_1$&0.108 & 0.402 &0.627 &0.846 $\quad$& 0.109 &0.414 &0.637 &0.844  \vspace*{0.1cm}\\
 && $Q_2$&0.183 & 0.473 &0.678 &0.876 $\quad$& 0.185 &0.475 &0.677 &0.873 \vspace*{0.1cm}\\
 && Mean &0.194 & 0.467 &0.671 &0.871 $\quad$& 0.193 &0.473 &0.673 &0.866  \vspace*{0.1cm}\\
 && $Q_3$&0.268 & 0.536 &0.720 &0.899 $\quad$& 0.266 &0.533 &0.716 &0.894  \vspace*{0.1cm}\\
 && Max &0.513 & 0.753 &0.834 &0.956 $\quad$& 0.591 &0.716 &0.821 &0.954  \vspace*{0.1cm} 

 \\                                                   
 \\               
&& Biais &-0.007 &-0.007 &-0.005& -0.004 $\quad$&-0.009& -0.007& -0.004 &-0.004 \vspace*{0.1cm}\\
&& RMSE& 0.057 &0.037 &0.026& 0.014 $\quad$& 0.057 &0.035 &0.023 &0.012 
\vspace*{0.1cm}\\
 && $\hat{\Omega}_{k_0}^{\mathrm{SP}}(1,1)$&5.935 &2.417 &1.202 &0.337 $\quad$&5.690 &2.081 &0.938 &0.240 \vspace*{0.1cm}
\\
&&  Min&0.002 & 0.374 &0.599 &0.857 $\quad$& 0.031 &0.372 &0.598 &0.856  \vspace*{0.1cm}\\
$k_0=2$&$n=2,000$&  $Q_1$&0.158 & 0.470 &0.678 &0.887 $\quad$& 0.156 &0.470 &0.683 &0.889  \vspace*{0.1cm}\\
 && $Q_2$&0.193 & 0.493 &0.696 &0.897 $\quad$& 0.190 &0.494 &0.697 &0.896 \vspace*{0.1cm}\\
 && Mean &0.193 & 0.493 &0.695 &0.896 $\quad$& 0.191 &0.493 &0.696 &0.896 \vspace*{0.1cm}\\
 && $Q_3$&0.229 & 0.518 &0.713 &0.905 $\quad$& 0.236 &0.515 &0.712 &0.904  \vspace*{0.1cm}\\
 && Max &0.379 & 0.634 &0.782 &0.934 $\quad$& 0.372 &0.595 &0.765 &  0.929
 \vspace*{0.1cm} 
\\
\hline
&& Biais &0.051 &-0.057 &-0.040& -0.031 $\quad$&0.060& -0.056& -0.038 &-0.030 \vspace*{0.1cm}\\
&& RMSE& 0.164 &0.153 &0.096& 0.052 $\quad$& 0.171 &0.142 &0.083 &0.049 
\vspace*{0.1cm}\\
 && $\hat{\Omega}_{k_0}^{\mathrm{SP}}(1,1)$&14.888 &4.184 &1.409 &1.098 $\quad$&14.043 &3.533 &1.428 &0.371\vspace*{0.1cm}
\\
&&  Min&0.000 & 0.006 &0.291 &0.691 $\quad$& 0.000 &0.003 &0.394 &0.644  \vspace*{0.1cm}\\
$k_0=3$&$n=300$&  $Q_1$&0.133 & 0.356 &0.607 &0.843 $\quad$& 0.129 &0.363 &0.617 &0.849  \vspace*{0.1cm}\\
 && $Q_2$&0.233 & 0.454 &0.669 &0.874 $\quad$& 0.249 &0.445 &0.666 &0.875 \vspace*{0.1cm}\\
 && Mean &0.251 & 0.443 &0.660 &0.869 $\quad$& 0.260 &0.444 &0.662 &0.870  \vspace*{0.1cm}\\
 && $Q_3$&0.354 & 0.547 &0.721 &0.898 $\quad$& 0.364 &0.539 &0.714 &0.896  \vspace*{0.1cm}\\
 && Max &0.780 & 0.790 &0.860 &0.959 $\quad$& 0.864                      &0.825 &0.838 &0.957  \vspace*{0.1cm} 

 \\                                                                 
 \\                              
&& Biais &-0.008 &-0.009 &-0.007& -0.005 $\quad$&0.001& -0.009& -0.006 &-0.005 \vspace*{0.1cm}\\
&& RMSE& 0.101 &0.056 &0.033& 0.016 $\quad$& 0.104 &0.053 &0.029 &0.013 
\vspace*{0.1cm}\\
 && $\hat{\Omega}_{k_0}^{\mathrm{SP}}(1,1)$&24.418 &5.098 &1.742 &0.390 $\quad$&24.525 &4.525 &1.353 &0.250 \vspace*{0.1cm}
\\
&& Min&0.000 & 0.268 &0.580 &0.830 $\quad$& 0.000 &0.328 &0.616 &0.852  \vspace*{0.1cm}\\
 $k_0=3$&$n=2,000$&  $Q_1$&0.113 & 0.455 &0.670 &0.886 $\quad$& 0.122 &0.458 &0.675 &0.888  \vspace*{0.1cm}\\
 && $Q_2$&0.190 & 0.494 &0.694 &0.896 $\quad$& 0.195 &0.491 &0.694 &0.896 \vspace*{0.1cm}\\
 && Mean &0.192 & 0.491 &0.693 &0.895 $\quad$& 0.201 &0.491 &0.694 &0.895 \vspace*{0.1cm}\\
 && $Q_3$&0.259 & 0.526 &0.715 &0.905 $\quad$& 0.273 &0.524 &0.714 &0.904  \vspace*{0.1cm}\\
 && Max &0.525 & 0.688 &0.782 &0.947 $\quad$& 0.581                        &0.670 &0.797 &0.933                                       

  \vspace*{0.1cm} 
                
 \\                                                             
\hline\hline 
\end{tabular}
\end{center}
}
\label{tab}
\end{table}

\begin{table}[H]

 \caption{\small{Sampling distribution of the moment estimator $\hat{R}_{k_0,n}$ of $\lambda_0$ for the 
  model \eqref{model} with $\epsilon_n= Z_{n}Z_{n-1}$ i.e. $k_0=2$ where $Z_n \sim IID{\cal P}(1)$. }}
{\scriptsize
\begin{center}
\begin{tabular}{lll rrrr rrrr}
\hline\hline
&  &   & \multicolumn{4}{c}{ $\xi \sim {\cal P}(\lambda_0)$}& \multicolumn{4}{c}{$\xi \sim {\cal B}(\lambda_0)$} \\
& &  & \multicolumn{4}{c}{$\lambda_0$}&  \multicolumn{4}{c}{$\lambda_0$}
\\
&&&$0.2$&$0.5$&$0.7$&$0.9$
$\quad$& $0.2$&$0.5$&$0.7$&$0.9$
\vspace*{0.1cm}\\
&& Biais &0.162&0.060  &0.014 &-0.018 $\quad$   &0.165 &0.062 &0.021&-0.012 \vspace*{0.1cm}\\
&& RMSE &0.208 &0.106 &0.064 &0.039  $\quad$ & 0.214 &0.102 &0.059 &0.035
\vspace*{0.1cm}\\
&& Min &0.020 &0.270 &0.379 &0.696  $\quad$ &0.004 &0.299 &0.534 &0.721 \vspace*{0.1cm}\\
 $k_0=1$&$n=300$ & $Q_1$ &0.276 &0.501 &0.672 &0.863  $\quad$&0.270 & 0.513 &0.683 &0.871 \vspace*{0.1cm}\\
 && $Q_2$ &0.360 &0.562 &0.717 &0.886 $\quad$ &0.367 & 0.564 &0.725 &0.892 \vspace*{0.1cm}\\
 && Mean &0.362 &0.560 &0.714 &0.882  $\quad$ &0.365 & 0.562 &0.721 &0.888 \vspace*{0.1cm}\\
 && $Q_3$ &0.451 &0.617 &0.757 &0.906 $\quad$ &0.460 & 0.618 &0.762 &0.910 \vspace*{0.1cm}\\
 && Max  &0.732 &0.858 &0.865 &0.957  $\quad$ &0.847 & 0.797 &0.863 &0.961 \vspace*{0.1cm}                            
 \\                                                                                                                                          
 \\               
&& Biais &0.192 &0.085 &0.037 &0.004 $\quad$  &0.198 &0.086 &0.042 & 0.009 \vspace*{0.1cm}\\
&& RMSE &0.201 &0.092 &0.045 &0.014 $\quad$   & 0.206 &0.094 &0.048& 0.014
\vspace*{0.1cm}\\
&&  Min &0.188 &0.455 &0.657 &0.857 $\quad$ &0.194 & 0.497 &0.675 & 0.873 \vspace*{0.1cm}\\
$k_0=1$&$n=2,000$&  $Q_1$ &0.352 &0.561 &0.721 &0.896  $\quad$&0.359 & 0.560 &0.727 &0.902 \vspace*{0.1cm}\\
 && $Q_2$ &0.393 &0.586 &0.737 &0.905 $\quad$ &0.398 & 0.586 &0.742 &0.909 \vspace*{0.1cm}\\
 && Mean &0.392 &0.585 &0.737 &0.904 $\quad$  &0.398 & 0.586 &0.742 &0.909 \vspace*{0.1cm}\\
 && $Q_3$ &0.431 &0.608 &0.754 &0.914  $\quad$ &0.436 & 0.609 &0.758 &0.916 \vspace*{0.1cm}\\
 && Max &0.645 &0.726 &0.824 &0.947 $\quad$ &0.587 & 0.715 &0.829 &0.939
 \vspace*{0.1cm}                                                                                           
\\                                                      
\hline
&& Biais &0.303& -0.030& -0.060 &-0.035 $\quad$  &0.305 &-0.041 &-0.059&-0.034  \vspace*{0.1cm}\\
&& RMSE &0.409 &0.222 &0.148 &0.060  $\quad$ & 0.403 &0.214 &0.132& 0.058
\vspace*{0.1cm}\\
&&  Min & 0.001 &0.000 &0.047 &0.645 $\quad$ &0.003 &0.000 &0.015 &0.587 \vspace*{0.1cm}\\
$k_0=3$&$n=300$&  $Q_1$ &0.274 &0.324 &0.563 &0.835 $\quad$ &0.294 &0.300 &0.572 &0.844 \vspace*{0.1cm}\\
 && $Q_2$ &0.509 &0.474 &0.658 &0.871 $\quad$ &0.498 &0.471 &0.656 &0.873 \vspace*{0.1cm}\\
 && Mean &0.503 &0.474 &0.640 &0.865 $\quad$  &0.505 &0.459 &0.641 &0.866 \vspace*{0.1cm}\\
 && $Q_3$ &0.737 &0.633 &0.739 &0.900 $\quad$ &0.705 &0.609 &0.722 &0.896 \vspace*{0.1cm}\\
 && Max &0.996 &0.989 &0.938 &0.964 $\quad$   &0.998 &0.975 &0.961 &0.961 \vspace*{0.1cm}  
 \\                                                                                                                            
 \\                              
&& Biais &0.201 &-0.019 &-0.009 &-0.006  $\quad$ &0.202 &-0.016 &-0.006 &-0.005 \vspace*{0.1cm}\\
&& RMSE &0.317 &0.112 &0.045 &0.018  $\quad$ &0.317 &0.103 &0.040 &0.014 \vspace*{0.1cm}\\
&& Min &0.001 &0.056 &0.532 &0.830  $\quad$ &0.000 &0.046 &0.573 &0.848 \vspace*{0.1cm}\\
 $k_0=3$ &$n=2,000$ &$Q_1$ &0.209 &0.415 &0.663 &0.884 $\quad$  & 0.202 &0.421 &0.668 &0.887 \vspace*{0.1cm}\\
 && $Q_2$ &0.371 &0.489 &0.693 &0.895 $\quad$ &0.390 &0.492 &0.694 &0.896 \vspace*{0.1cm}\\
 && Mean &0.401 &0.481 &0.691 &0.894 $\quad$ &0.402 &0.484 &0.694 &0.895 \vspace*{0.1cm}\\
 && $Q_3$ &0.573 &0.549 &0.721 &0.906 $\quad$ &0.577 &0.552 &0.718 &0.905 \vspace*{0.1cm}\\
 && Max &0.998 &0.800 &0.820 &0.955  $\quad$ &0.994 &0.770 &0.821 & 0.934 \vspace*{0.1cm}                                                         
 \\                                                                                                                            
 \\                              
&& Biais &0.128 &-0.010 &-0.004 &-0.002  $\quad$ &0.134 &-0.008 &-0.003 &-0.001 \vspace*{0.1cm}\\
&& RMSE &0.238 &0.069 &0.028 &0.011  $\quad$ &0.242 &0.065 &0.026 &0.008 \vspace*{0.1cm}\\
&& Min &0.000 &0.275 &0.603 &0.864  $\quad$ &0.002 &0.280 &0.609 &0.870 \vspace*{0.1cm}\\
 $k_0=3$ &$n=5,000$ &$Q_1$ &0.169 &0.444 &0.677 &0.891 $\quad$  & 0.173 &0.451 &0.681 &0.893 \vspace*{0.1cm}\\
 && $Q_2$ &0.310 &0.494 &0.697 &0.899 $\quad$ &0.314 &0.496 &0.698 &0.899 \vspace*{0.1cm}\\
 && Mean &0.328 &0.490 &0.696 &0.898 $\quad$ &0.334 &0.492 &0.697 &0.899 \vspace*{0.1cm}\\
 && $Q_3$ &0.465 &0.535 &0.714 &0.905 $\quad$ &0.468 &0.537 &0.714 &0.904 \vspace*{0.1cm}\\
 && Max &0.952 &0.688 &0.774 &0.938  $\quad$ &0.999 &0.689 &0.785 & 0.929 \vspace*{0.1cm}                                                                                                                                         
 \\                                                                                                                                                                                     
\hline\hline 
\end{tabular}
\end{center}
}
\label{tab1}
\end{table}

\begin{figure}[H]
  \centering
  \includegraphics[width=0.5\linewidth]{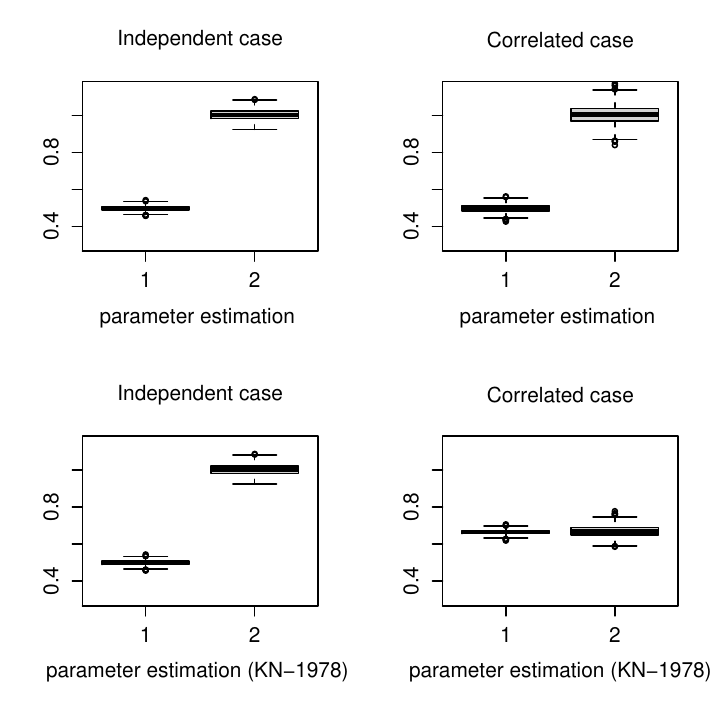}
  \caption{The moment estimators and the estimators proposed by \cite{KN78} of $N=1,000$ independent simulations of model \eqref{model} of size $n=5,000$ with unknown parameter $\lambda_0=0.5$ and $m_0=1$, when the immigration is independent (left panels, with $\epsilon_n \sim IID{\cal P}(1)$) and when the immigration is correlated (right panels, with $\epsilon_n=Z_n Z_{n-1}$ where $Z_n \sim IID{\cal P}(1)$). The reproduction sequence $\xi\sim\mathcal{B}(\lambda_0)$.
The panels display the distribution of the estimators $\hat{R}_{k_0,n}$ and  $\hat{M}_{k_0,n}$. The two top panels (resp. two bottom panels) correspond to our proposed estimators (resp. to the estimators proposed by \cite{KN78} and \cite{AOA87}).}
\label{fig:param}
\end{figure}

\begin{figure}[H]
  \centering
  \includegraphics[width=0.5\linewidth]{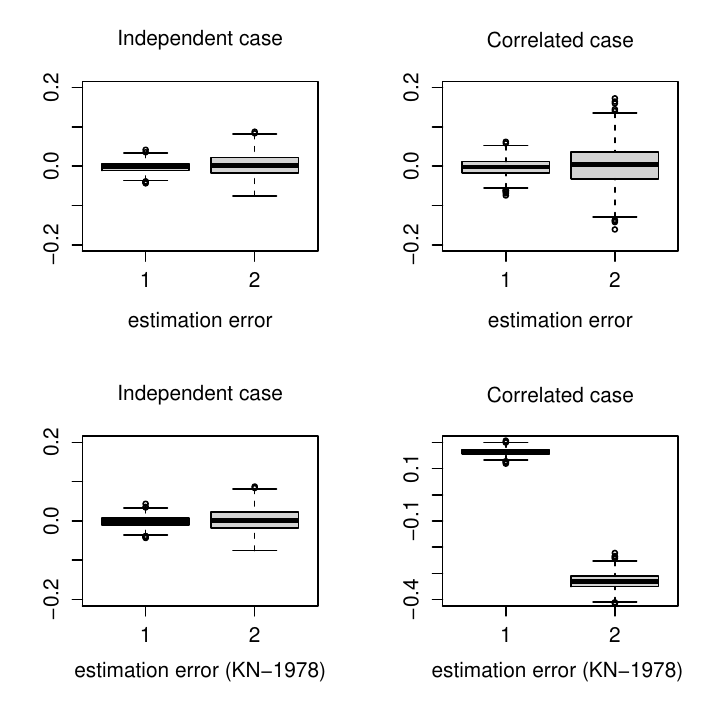}
  \caption{The panels at the top present the distribution of the estimation error of the estimates $\lambda_0$ and $m_0$ given in Figure \ref{fig:param}. The left (resp. right) panel corresponds to
the case of independent (correlated) immigration.
 Points 1 and 2, in the box-plots, display the distribution of the estimation error  $\hat{R}_{k_0,n}-\lambda_0$ and  $\hat{M}_{k_0,n}-m_0$. The two top panels (resp. two bottom panels) correspond to our proposed estimators (resp. to the estimators proposed by \cite{KN78} and \cite{AOA87}).}
 \label{fig:error}
\end{figure}

\begin{figure}[H]
  \centering
  \includegraphics[width=0.55\linewidth]{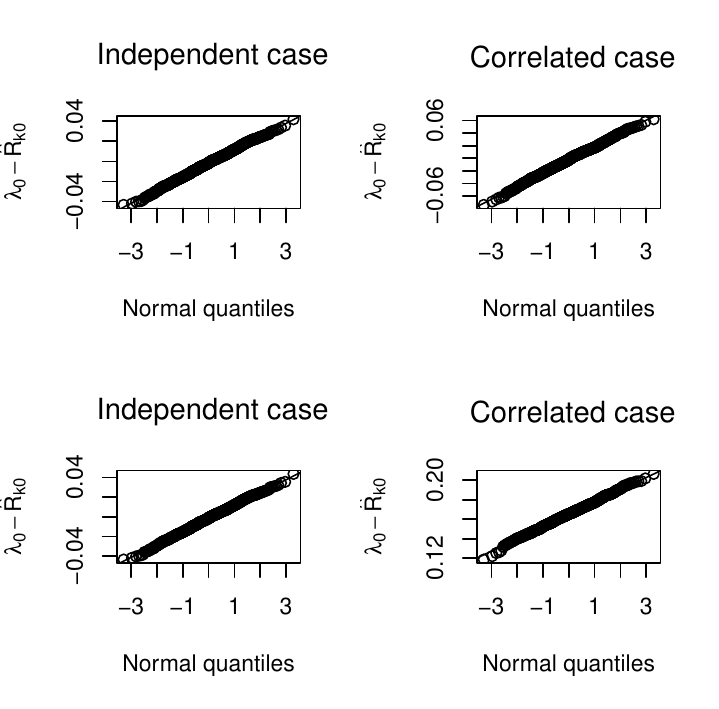}
  \caption{
The panels at the top present the Q–Q plot of the estimates $\lambda_0$ given in Figure \ref{fig:param}. The left (resp. right) panel corresponds to
the case of independent (correlated) immigration.
The two top panels (resp. two bottom panels) correspond to our proposed estimators (resp. to the estimators proposed by \cite{KN78} and \cite{AOA87}).} 
  \label{fig:qq}
\end{figure}

\begin{figure}[H]
  \centering
  \includegraphics[width=0.55\linewidth]{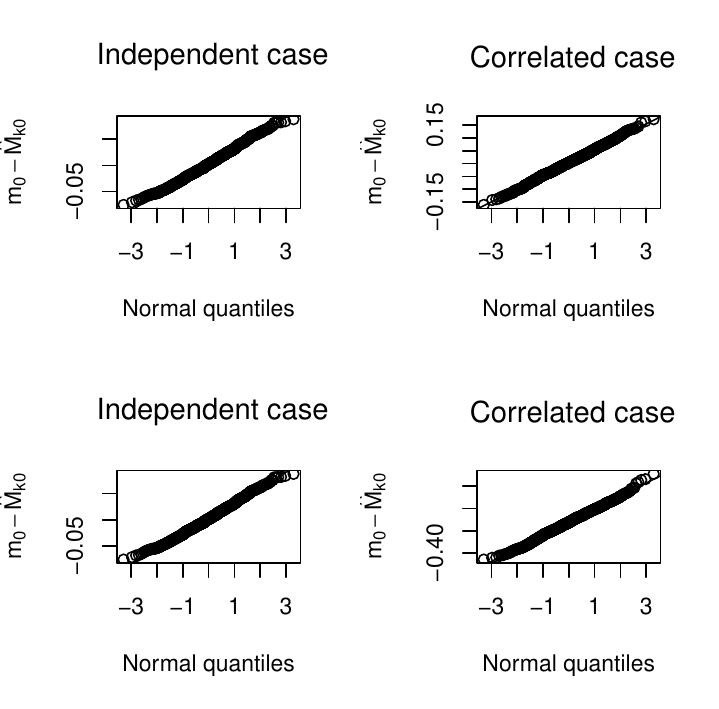}
  \caption{
The panels at the top present the Q–Q plot of the estimates $m_0$ given in Figure \ref{fig:param}. The left (resp. right) panel corresponds to
the case of independent (correlated) immigration.
The two top panels (resp. two bottom panels) correspond to our proposed estimators (resp. to the estimators proposed by \cite{KN78} and \cite{AOA87}).} 
  \label{fig:qq1}
\end{figure}

\begin{figure}[H]
  \centering
  \includegraphics[width=0.55\linewidth]{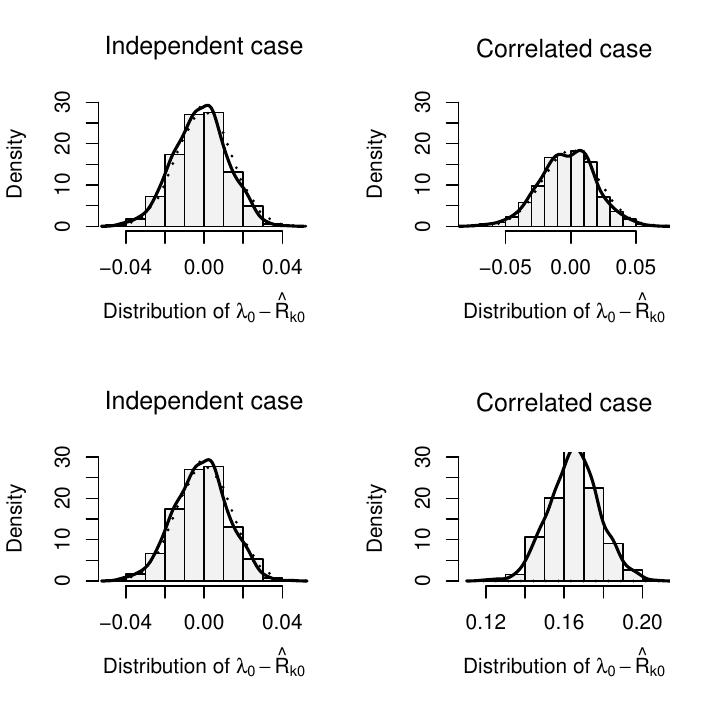}
  \caption{
The panels at the top display the distribution of the estimates $\lambda_0$  given in Figure \ref{fig:param}. The kernel density estimate is displayed in full line, and the centered Gaussian density with the same variance is plotted in dotted line. The left (resp. right) panel corresponds to the case of independent (correlated) immigration. The two top panels (resp. two bottom panels) correspond to our proposed estimators (resp. to the estimators proposed by \cite{KN78} and \cite{AOA87}).}
\label{fig:hist}
\end{figure}

\begin{figure}[H]
  \centering
  \includegraphics[width=0.55\linewidth]{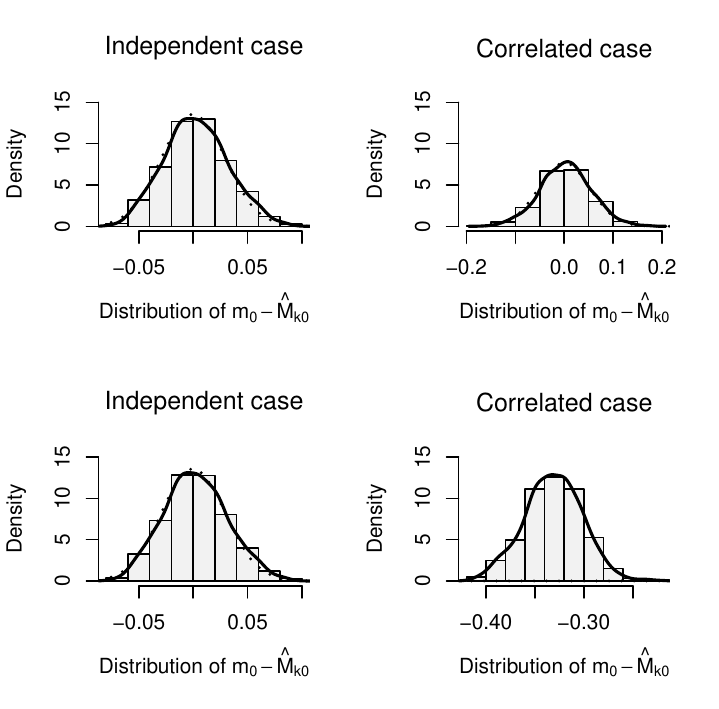}
  \caption{
The panels at the top display the distribution of the estimates $m_0$ given in Figure \ref{fig:param}. The kernel density estimate is displayed in full line, and the centered Gaussian density with the same variance is plotted in dotted line. The left (resp. right) panel corresponds to the case of independent (correlated) immigration. The two top panels (resp. two bottom panels) correspond to our proposed estimators (resp. to the estimators proposed by \cite{KN78} and \cite{AOA87}).}
\label{fig:hist1}
\end{figure}

\begin{figure}[H]
  \centering
  \includegraphics[width=0.55\linewidth]{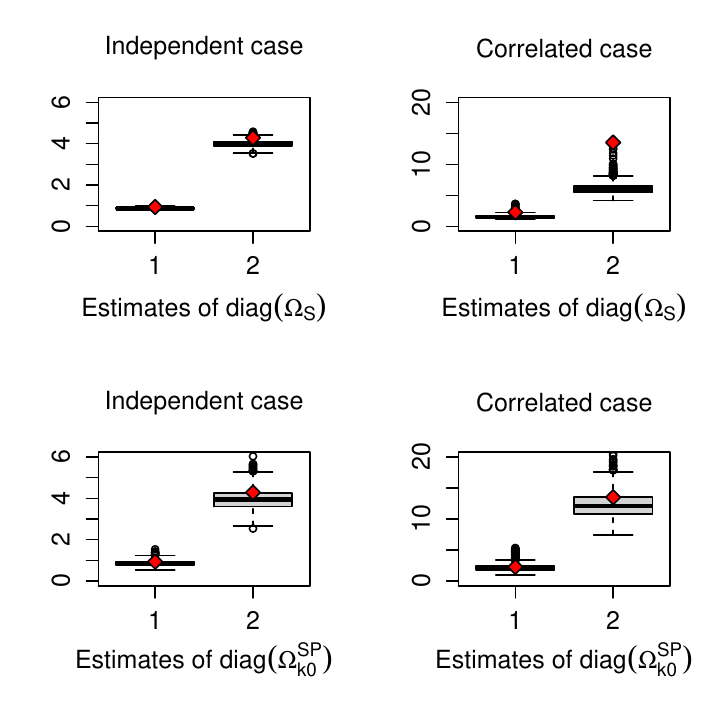}
\caption{Comparison of standard and modified estimates of the asymptotic variance of the moment estimator of the model parameters, on the simulated models presented in~Figure~\ref{fig:param}. The diamond symbols represent the mean, over $N=1,000$ replications, of the standardized squared errors $n\left(\hat{R}_{k_0,n}-0.5\right)^2$ for 1 (0.934 in the independent immigration case and 2.260 in the correlated immigration case)  and $n\left(\hat{M}_{k_0,n}-1\right)^2$ for 2 (4.283 in the independent immigration case and 13.519 in the correlated immigration case) .}
\label{fig:var}
\end{figure}

\subsection{Correlated case}\label{sec:correlated}
We now illustrate the performance of the estimator $\hat{S}_n$ in \eqref{def_estimator_general} of which theorical convergence result was presented in Theorems \ref{main_theo_general1} and \ref{main_theo_genera2}. 

The process considered here is the stationary version of \eqref{model}, where $\xi \sim {\cal P}(\lambda_0)$, and where $\{\eps_n,\ n\in\nbZ \}$ is the stationary Markov chain with state space $\{0,1\}$ and transition matrix $P=\left( \begin{array}{cc}
1/2 & 1/2\\
1 & 0
\end{array}\right)$. It is standard that $ \alpha_\epsilon$ is exponentially decreasing and that $\mathbf{ (A1)}$ is satisfied if $\lambda_0$ is not close to $0$, and all simulations below are performed such that $\Xi(\lambda_0,m_0,V_0, f_\nu)\neq 0$ (with the function $\Xi$ defined in \eqref{def_Xi}) with $V_0=\mbox{Var}(\xi)=\lambda_0$ and (after some easy computation) $m_0=\nbE(\epsilon)=1/3$.
\begin{table}[hbtp!]
\centering
\begin{tabular}[hbtp!]{|c|c|c|c|c|c|c|}
\hline
$c$ & $\lambda_0$ &$\ln (\lambda_0)$ & $\nbE(\hat{S}_n)$ & $\mbox{Var}(\hat{S}_n)$ & $||\hat{S}_n- \ln (\lambda_0) ||_2$& $\nbE(\exp(\hat{S}_n))$\\
\hline 
 $4.7$ & $0.9$ & $-0.1053605$ & $-0.0665497$  & $0.0001058$ & $  0.0401508$ & $0.9356657$ \\
\hline
 $1.4$ & $0.7$ & $-0.3566749$ &  $-0.3192667$ &  $0.0066237$  &  $0.0895714$ & $0.7292888$ \\%
\hline
 $0.7$ & $0.5$ & $-0.6931472$ & $-0.6326057$ & $0.0369518$ & $ 0.2015367$ & $0.5423869$ \\
\hline
 $0.3$ & $0.2$ & $-1.6094379$ &  $-1.490265$ &  $0.0000447$ & $0.1193602$ & $0.225318$\\ 
\hline
\end{tabular}
\caption{Illustration of estimator $\hat{S}_n$ with $k_n=\lfloor c \ln n\rfloor$ (with $n=5.10^6$)}\label{MC_T_hat}
\end{table}
Inspired by Theorems \ref{main_theo_general1} and \ref{main_theo_genera2}, we proceed to estimate $\nbE(\hat{S}_n)$ and $\nbE(\exp(\hat{S}_n))$ by Monte Carlo simulation by simulating $20$ times $\hat{S}_n$ with $n=5.10^6$ and $k_n=\lfloor c \ln n\rfloor$. The factor $c$ is chosen purposefully close to the value $ -{1}/({2 \ln \lambda_-})$ while assuming that $\lambda_-$ is close to $\lambda_0$ in Theorem \ref{main_theo_general1}, which theoretically guarantees the convergence for $\hat{S}_n$. The results are displayed in Table \ref{MC_T_hat}. We first remember that convergence is extremely slow, as was argued in \eqref{slow_convergence}, hence the choice of the very lengthy trajectory $n=5.10^6$ (as opposed to the shorter lengths considered in Section \ref{sec:ultimately_correlated_numerics}). Table \ref{MC_T_hat} shows that the estimator $\hat{S}_n$ performs well: in particular, $\nbE(\exp(\hat{S}_n))$ is close to $\lambda_0$ for different values of the reproduction parameter.

We next discuss the choice of the lag sequence $(k_n)_{n\in \nbN}$. To show that that the choice $k_n \propto \ln n$ (which, again, is motivated by Theorem \ref{main_theo_general1}) is efficient, we provide numerics for the case where the sequence tends to $\infty$ at a different speed, namely faster than $\ln n$, while still being an $\mathrm{o}(n)$. Thus, and for comparison, Table \ref{MC_T_hat_nul} shows estimations when $k_n=\lfloor  \sqrt{n}\rfloor$ (which still satisfies $k_n\ll n$): it is clear that the corresponding estimator performs very badly, with values of $\nbE(\hat{S}_n)$ being very close to $0$, and $\nbE(\exp(\hat{S}_n))$ around $0.99$, independently of the value of $\lambda_0$.

\begin{table}[hbtp!]
\centering
\begin{tabular}[hbtp!]{|c|c|c|c|c|c|}
\hline
$\lambda_0$ &$\ln (\lambda_0)$ & $\nbE(\hat{S}_n)$ & $\mbox{Var}(\hat{S}_n)$ & $\nbE(\exp(\hat{S}_n))$\\
\hline
 $0.9$ & $-0.1053605$ & $ -0.0001556$  & $8.019 .10^{-10}$ &  $0.9998444$\\
\hline
  $0.7$ & $-0.3566749$ & $ -0.0008472$   & $7.539.10^{-09}$  &  $0.9991532  $ \\
\hline
  $0.5$ & $-0.6931472 $ & $-0.001969$  & $2.651.10^{-08}$ & $0.998033$ \\
\hline
  $0.2$ & $-1.6094379$ & $-0.0055293$  & $0.0000007 $  &  $0.9944862$ \\ 
\hline
\end{tabular}
\caption{Illustration of estimator $\hat{S}_n$ with $k_n=\lfloor  \sqrt{n}\rfloor$ (with $n=5.10^6$)}\label{MC_T_hat_nul}
\end{table}

\appendix
\section{Appendix : Proofs of the main results}\label{sec:app}
\subsection{Proof of Proposition \ref{prop_stationarity}}\label{sec:prop_beta_moments}
%
In order to prove that $Y_0$ admits moments of order $2\beta$, it suffices from \eqref{rep_stationary} to prove that $\sum_{i=0}^\infty ||\theta^{(i)}\circ\epsilon_{-i}||_{2\beta}$ converges. This comes from the fact that
\begin{equation}\label{bound_KW}
||\theta^{(i)}\circ\epsilon_{-i}||_{2\beta}=\mathrm{O}(\upsilon^i),\quad ||\theta^{(i)}_n\circ\epsilon_{n-i}||_{2\beta}=\mathrm{O}(\upsilon^i),\quad \forall n\in \nbZ
\end{equation}
for some $\upsilon<1$, see the proof in \citet[Section 3.1]{KW21}. Note that the authors proved this latter exponential decrease in a context where the immigration process is i.i.d., but a close look reveals that the only ingredient in the proof is the independence of $\theta^{(i)}$ from $ \epsilon_{-i}$ for all $i\in \nbN$, which is the case here since $\{\xi_{n,k},\ n\in\nbZ,\ k\in \nbN \}$ and $\{\eps_n,\ n\in\nbZ \}$ are independent sequences.
\zak
\subsection{Proof of Lemma \ref{lemma_v_k}} \label{sec:lemma_vk}
We introduce the quantity $\eta_k:= \nbE(\epsilon_1 \epsilon_{-k})$, $k\in\N$, and observe, using \eqref{model} as well as the independence of $\epsilon_1$ and $Y_{-k-1}$ from $\{ \xi_{-k,j},\ j\in \nbN\}$, that
\begin{eqnarray}
\eta_k&=& \nbE \left( \epsilon_1 \left[ Y_{-k}- \sum_{j=1}^{Y_{-k-1}} \xi_{-k,j}\right]\right)=\nbE(\epsilon_1 Y_{-k}) - \lambda_0 \nbE(\epsilon_1 Y_{-k-1})\nonumber\\
&=&v_k-\lambda_0 v_{k+1},\quad k\in \nbN .\label{lem_eta1}
\end{eqnarray}
Observing now that $\eta_k=m_0^2+\nu_{k+1}$, $k\in \nbN$, yields the following recursive first order relation from \eqref{lem_eta1}
$$
v_{k+1}= \frac{1}{\lambda_0} v_k - \frac{1}{\lambda_0}(m_0^2+\nu_{k+1}),\quad k\in \nbN ,
$$
leading to the expression
\begin{eqnarray}
v_k&=& - \sum_{j=1}^k \frac{1}{\lambda_0^j}(m_0^2+\nu_{k+1-j}) + \frac{1}{\lambda_0^k} v_0 \nonumber\\
&=& -m_0^2 \frac{1}{\lambda_0}\frac{1-1/\lambda_0^k}{1-1/\lambda_0}- \sum_{j=1}^k \frac{\nu_{k+1-j}}{\lambda_0^j}+ \frac{1}{\lambda_0^k} v_0  = \frac{m_0^2}{1-\lambda_0} + \frac{1}{\lambda_0^k}\left[ \frac{m_0^2}{\lambda_0-1} - \sum_{j=0}^{k-1} \lambda_0 ^j \nu_{j+1} + v_0\right]\nonumber\\
&=& \frac{m_0^2}{1-\lambda_0} + \frac{1}{\lambda_0^k}\left[ \frac{m_0^2}{\lambda_0-1} - \sum_{j=0}^{\infty} \lambda_0 ^j \nu_{j+1} + v_0\right] - \chi_k.\label{lem_eta2}
\end{eqnarray}
We note now that $0\le  v_k\le [\nbE (\epsilon_1^2)]^{1/2} [\nbE(Y_{-k}^2)]^{1/2}= [\nbE (\epsilon_1^2)]^{1/2} C_2^{1/2}$ thanks to the Cauchy Schwarz inequality, so that the sequence $\{v_k,\ k\in \nbN \}$ is uniformly bounded. Also, one has that $\left|\chi_k\right|\le \sum_{j=k}^\infty |\nu_{j+1}|$, which tends to $0$ as $k\to \infty$ thanks to the (stronger) assumption $\mathbf{(A3)}$. The subcriticality condition \eqref{stab} for $\lambda_0$ thus entails from \eqref{lem_eta2} the explicit expression
\begin{equation}\label{expr_v_0}
v_0= \frac{m_0^2}{1-\lambda_0} + \sum_{j=0}^{\infty} \lambda_0 ^j \nu_{j+1}= \frac{m_0^2}{1-\lambda_0} - \chi_0,
\end{equation}
as well as the more general expression \eqref{expr_v_k} for $v_k$, $k\in \nbN$. 

We now come back to \eqref{expr_moments}. Using \eqref{model}, the independence of the sequence $\{ \xi_{1,j},\ j\in \nbN\}$ as well as its independence from $Y_0$, yields
\begin{eqnarray*}
C_2=\nbE(Y_1^2)&=&\nbE \left( \left[\sum_{j=1}^{Y_{0}} \xi_{1,j} + \epsilon_1\right]^2\right)\\
&=& \nbE\left( \sum_{j=1}^{Y_{0}} \xi_{1,j}^2 + \sum_{1\le r\neq r'\le Y_0}\xi_{1,r} \xi_{1,r'} + 2 \epsilon_1 \sum_{j=1}^{Y_{0}} \xi_{1,j} +\epsilon_1^2 \right)\\
&=& \nbE(\xi^2) \nbE(Y_0) + [\nbE(\xi)]^2 \nbE(Y_0(Y_0-1)) + 2 \nbE(\xi) \nbE(\epsilon_1 Y_0)+\nbE\epsilon_1^2\\
&=& V_{0,\xi} C_1 + \lambda_0^2 C_2 + 2 \lambda_0 v_0 +\nbE\epsilon_1^2,
\end{eqnarray*}
which, thanks to \eqref{expr_v_0} and the fact that $\chi_0=-f_\nu(\lambda_0)$, yields \eqref{expr_moments}.
\zak
\subsection{Proof of Corollary \ref{prop_C_1_u_k}}\label{sec:proof_prop_C_1_u_k}
\eqref{expr_first_moment} coupled with the relation \eqref{expr_v_k} obtained in Lemma \ref{lemma_v_k} yields $u_k - \lambda_0 u_{k-1}=C_1^2(1-\lambda_0)- \chi_k$, from which in turn we deduce the following recursive relation
$$
C_1^2-u_k= \lambda_0(C_1^2- u_{k-1})-\chi_k,\quad k\in \nbN .
$$
Since $u_{-1}=C_2$, we arrive then at $C_1^2-u_k=\sum_{j=0}^k \lambda_0^j \chi_{k-j} + \lambda_0^{k+1}(C_1^2-C_2)$, hence \eqref{expr_C_1_u_k}.
\zak
\subsection{Proof of Proposition \ref{prop_consistency}}\label{sec:consistency_k_0}
Iterating Equation \eqref{model} yields that $Y_n$ is a function of $\epsilon_k$, $k\le n$ and $\xi_{k,i}$, $k\le n$, $i\in \nbN$. $\{Y_n,\ n\in \nbZ\}$ and $\{Y_nY_{n+k_0},\ n\in \nbZ\}$ being stationary and integrable, and $\{(\epsilon_n,\xi_{n,i}),\ n\in \nbZ,\ i\in \nbN\setminus \{0,\} \}$ being ergodic, yields the a.s. convergence of $\hat{Y}_n$ and $\hat{Y}_{k_0,n}$ (defined in \eqref{def_estimators}) respectively towards $C_1$ and $u_{k_0}$ by \citet[Theorem A.2 p.344]{FZ19}. Hence $\hat{R}_{k_0,n}$ and, in turn, $\hat{M}_{k_0,n}$, converge a.s. to $\lambda_0$ and $m_0$ as $n\to\infty$ based on \eqref{lambda_0_k_0} and \eqref{expr_first_moment}.
\zak

\subsection{Proof of Theorem \ref{theo_CLT_vector_k_0} }\label{sec:proof_theo_k_0}
We start by stating two technical lemmas.
\begin{lemm}\label{lemma_cov}\normalfont
Let $k$, $j$ and $m$ in $\nbN$ with $m>j> k$ and $\varphi_k^-$ (resp. $\varphi_{m}^+$) be a $\sigma(\epsilon_n,\ n\le k)\otimes \sigma(\xi_{n,i}, n\le k, i\in \N)$ measurable (resp. $\sigma(\epsilon_n,\ n\ge m)$ measurable) random variables belonging to $\mathbb{L}^{(1-2/\beta)\vee 4}$. The following decompositions hold:
\begin{eqnarray}
\mbox{Cov}(\varphi_k^-,Y_j Y_{m})&=& \sum_{s=0}^{m-j-1}\lambda_0^s \left[ \sum_{t=0}^{j-k-1} \lambda_0^t \mbox{Cov}(\varphi_k^-, \epsilon_{j-t} \epsilon_{m-s}) + \lambda_0^{j-k}
\mbox{Cov}(\varphi_k^-, Y_k \epsilon_{m-s}) \right]\nonumber\\
&&\hspace{1.5cm} + \lambda_0^{m-j} \mbox{Cov}(\varphi_k^-, Y_j^2),\label{decompo1}\\
\mbox{Cov}(\varphi_k^-,Y_j^2 \varphi_{m}^+)&=& \sum_{s=0}^{j-k-1}\lambda_0^{2s} f_{j-s}(k,m) +\lambda_0^{2(j-k)} \mbox{Cov}(\varphi_k^-,Y_k^2 \varphi_{m}^+),\quad\mbox{where} \label{decompo2}\\
f_j(k,m) &=&V_{0,\xi} \mbox{Cov}(\varphi_k^-,Y_{j-1}\varphi_{m}^+) + \mbox{Cov}(\varphi_k^-, \epsilon_j^2 \varphi_{m}^+)  
+ 2\lambda_0 \mbox{Cov}(\varphi_k^-,Y_{j-1} \epsilon_j \varphi_{m}^+),\label{decompo3}
\end{eqnarray}
with $V_{0,\xi}=\mbox{Var}(\xi)$.
\end{lemm}
\begin{proof}
In view of \eqref{model}, the independence of $\xi_{m,p}$ from $\varphi_k^-$, $Y_{j}$ yields:
$$
\mbox{Cov}(\varphi_k^-,Y_j Y_{m})= \mbox{Cov}\left(\varphi_k^-,Y_j \left[ \sum_{p=1}^{Y_{m-1}} \xi_{m,p}+\eps_{m} \right]\right)= \lambda_0 \mbox{Cov}(\varphi_k^-, Y_{j} Y_{m-1})+ \mbox{Cov}(\varphi_k^-, Y_{j}\eps_{m} )
$$
which, by direct induction, gives
$$
\mbox{Cov}(\varphi_k^-,Y_j Y_{m})= \sum_{s=0}^{m-j-1}\lambda_0^s \mbox{Cov}(\varphi_k^-, Y_{j}\eps_{m-s} ) + \lambda_0^{m-j} \mbox{Cov}(\varphi_k^-, Y_{j}^2).
$$
A similar argument (involving this time the independence of $\varphi_k^-$ and $\eps_{m-s}$ from $\{ \xi_{r,p},\ p\ge 1\}$, $r=k+1,...,j$), yields that $\mbox{Cov}(\varphi_k^-, Y_{j}\eps_{m-s} )=\sum_{t=0}^{j-k-1} \lambda_0^t \mbox{Cov}(\varphi_k^-, \epsilon_{j-t} \epsilon_{m-s}) + \lambda_0^{j-k}
\mbox{Cov}(\varphi_k^-, Y_k \epsilon_{m-s})$ which, plugged in the above equality, yields \eqref{decompo1}.\\
In order to prove \eqref{decompo2}, we first expand $Y_j^2=\left[ \sum_{p=1}^{Y_{j-1}} \xi_{j,p}+\eps_{j}\right]^2=\sum_{p=1}^{Y_{j-1}} \xi_{j,p}^2 + \sum_{1\le p\neq p'\le Y_{j-1}} \xi_{j,p}\xi_{j,p'}  + \eps_{j}^2 + 2 \eps_{j} \sum_{p=1}^{Y_{j-1}} \xi_{j,p}$, which, still by independence arguments and since $\nbE(\xi^2)=V_{0,\xi}+\lambda_0^2$, yields
\begin{align*}
\mbox{Cov}(\varphi_k^-,Y_j^2 \varphi_{m}^+)&=(V_{0,\xi}+\lambda_0^2) \mbox{Cov}(\varphi_k^-,Y_{j-1} \varphi_{m}^+)+\lambda_0^2 \mbox{Cov}(\varphi_k^-, Y_{j-1}(Y_{j-1}-1) \varphi_{m}^+)\\&\hspace*{2cm}+ \mbox{Cov}(\varphi_k^-,\eps_{j}^2 \varphi_{m}^+)
+ 2\lambda_0 \mbox{Cov}(\varphi_k^-,\eps_{j} Y_{j-1}\varphi_{m}^+)
\\&= \lambda_0^2 \mbox{Cov}(\varphi_k^-,Y_{j-1}^2 \varphi_{m}^+) + f_j(k,m)
\end{align*}
where $f_j(k,m)$ is defined by \eqref{decompo3}. Direct induction in the above equality yields \eqref{decompo2}.
\end{proof}
\begin{lemm}\label{lemma_cov2}\normalfont
Let $k$, $j$ in $\nbN$ with $j> k$ and $\varphi_t^-$ (resp. $\varphi_{t}^+$) be a $\sigma(\epsilon_n,\ n\le t,\ \xi_{n,i}, n\le t, i\in \N)$ measurable (resp. $\sigma(\epsilon_n,\ n\ge t)$ measurable) random variables belonging to $\mathbb{L}^{\beta\vee 4}$, uniformly in $t=k,...,j$. The following upper bound holds
\begin{equation}\label{bound_sum_Davydov0}
\left| \sum_{s=0}^{j-k-1} \lambda_0^s \mbox{Cov}(\varphi_k^-,\varphi_{j-s}^+)\right|\le K\left[ \alpha_\epsilon\left( \left\lfloor \frac{j-k-1}{2}\right\rfloor\right)^{1-2/\beta} + \lambda_0^{\left\lfloor \frac{j-k-1}{2}\right\rfloor}\right],
\end{equation}
where $\beta>2$ is defined in $\mathbf{(A1)}$. 
 As a consequence, the following similar bound holds for all $k<j<m$ in $\nbN$:
\begin{equation}\label{bound_sum_Davydov1}
\left| \mbox{Cov}(\varphi_k^-,Y_j^\ell \varphi_m^+)\right|\le K\left[ \alpha_\epsilon\left( \left\lfloor \frac{j-k-1}{2}\right\rfloor\right)^{1-2/\beta} + \lambda_0^{\left\lfloor \frac{j-k-1}{2}\right\rfloor}\right]
\end{equation}
for $\ell =1,2$. 
\end{lemm}
\begin{proof}
Since $\{\xi_{m,i}, m\in \N, i\in \N\}$ is independent from $\{\epsilon_n,\ n\in \nbZ\}$, Davydov's inequality (see \cite{D68}) as well as the $\mathbb{L}^{\beta}$ assumption on $\varphi_t^-$ and $\varphi_{t}^+$, $t=k,...,j$ (see Assumption \textbf{(A2)}), yields that $|\lambda_0^s\mbox{Cov}(\varphi_k^-,\varphi_{j-s}^+)|$ is less than $K\;  \alpha_\epsilon(j-s-k)^{\beta}$ or $K\; \lambda_0^{s}$ for $s=0,...,j-k-1$ (thanks to the Cauchy Scwharz inequality thanks to the $\mathbb{L}^{4}$, hence $\mathbb{L}^{2}$, assumption), for some constant $K>0$. We split the sum on the lefthanside of \eqref{bound_sum_Davydov0} in $s=0,..., \lfloor({j-k-1})/{2}\rfloor$ and $s=\lfloor({j-k-1})/{2}\rfloor,...,j-k-1$. When $s=0,..., \lfloor({j-k-1})/{2}\rfloor$ we get that $j-s\ge \lfloor({j-k-1})/{2}\rfloor$, so that Davydov's inequality  reads
\begin{equation}\label{Davydov1_lemmaA2}
\lambda_0^s \left| \mbox{Cov}(\varphi_k^-,\varphi_{j-s}^+)\right| \le C \lambda_0^s \alpha_\eps \left( \left\lfloor \frac{j-k-1}{2}\right\rfloor\right)^{1-2/\beta}  \left\| \varphi_k^-\right\|_2  .\left\| \varphi_{j-s}^+\right\|_2 \le K \lambda_0^s \alpha_\eps \left( \left\lfloor \frac{j-k-1}{2}\right\rfloor\right)^{1-2/\beta}
\end{equation}
for some universal positive constants $C$ and $K$. So that summing \eqref{Davydov1_lemmaA2} yields 
\begin{equation}\label{Davydov2_lemmaA2}
\left| \sum_{s=0}^{  \lfloor({j-k-1})/{2}\rfloor}\lambda_0^s  \mbox{Cov}(\varphi_k^-,\varphi_{j-s}^+)\right| \le K  \alpha_\eps \left( \left\lfloor \frac{j-k-1}{2}\right\rfloor\right)^{1-2/\beta}
\end{equation}
for some (different) constant $K>0$. When $s=\lfloor({j-k-1})/{2}\rfloor +1,... j-k-1$, one easily gets
\begin{equation}\label{Davydov3_lemmaA2}
\left| \sum_{s=\lfloor({j-k-1})/{2}\rfloor +1}^{ j-k-1}\lambda_0^s  \mbox{Cov}(\varphi_k^-,\varphi_{j-s}^+)\right| \le K \lambda_0^{\left\lfloor \frac{j-k-1}{2}\right\rfloor},
\end{equation}
so that combining \eqref{Davydov2_lemmaA2} and \eqref{Davydov3_lemmaA2}
yields \eqref{bound_sum_Davydov0}.\\ 
We next prove \eqref{bound_sum_Davydov1} for $s=2$ ($s=1$ being proved similarly). 

Observing each term on the righthandside of \eqref{decompo3}, we have (still by standard arguments) that $$|f_j(k,m)|\le K \left[ \alpha_\epsilon\left( \left\lfloor \frac{j-1}{2}\right\rfloor\right)^{1-2/\beta} + \lambda_0^{\left\lfloor \frac{j-1}{2}\right\rfloor}\right].$$ Hence, splitting the sum on the righthandside of \eqref{decompo2} in $s=0,..., \left\lfloor ({j-k-1})/{2}\right\rfloor$ and $s=\left\lfloor ({j-k-1})/{2}\right\rfloor+1,..., j-k-1$ and using the fact that $\alpha_\epsilon (\cdot)$ is decreasing, yields
\begin{align*}
\left|\sum_{s=0}^{j-k-1}\lambda_0^{2s} f_{j-s}(k,m)\right|&\le \sum_{s=0}^{\left\lfloor \frac{j-k-1}{2}\right\rfloor} |\lambda_0^{2s} f_{j-s}(k,m)| +  \sum_{s=\left\lfloor \frac{j-k-1}{2}\right\rfloor+1}^{j-k-1} |\lambda_0^{2s} f_{j-s}(k,m)|\\&
\le K \left[ \alpha_{{\epsilon}}\left( \left\lfloor \frac{j-k-1}{2}\right\rfloor\right)^{1-2/\beta} + \lambda_0^{\left\lfloor \frac{j-k-1}{2}\right\rfloor}\right] 
+ K \lambda_0^{2 \left\lfloor \frac{j-k-1}{2}\right\rfloor}\\&\le K \left[ \alpha_{{\epsilon}}\left( \left\lfloor \frac{j-k-1}{2}\right\rfloor\right)^{1-2/\beta} + \lambda_0^{\left\lfloor \frac{j-k-1}{2}\right\rfloor}\right] ,
\end{align*}
where the constant $K>0$ changes at each inequality. Combined with the fact that: $$\left|\lambda_0^{2(j-k)} \mbox{Cov}(\varphi_k^-,Y_k^2 \varphi_{m}^+)\right|\le K \lambda_0^{2(j-k)},$$ (a consequence of the Cauchy-Schwarz inequality, thanks to the $\mathbb{L}^{\beta\vee 4}$ integrability assumption of $\varphi_k^-$, $\varphi_{m}^+$ and the $\mathbb{L}^4$ property of $Y_j$), we thus obtain \eqref{bound_sum_Davydov1} from \eqref{decompo2}.
\end{proof}

{\bf Proof of Theorem \ref{theo_CLT_vector_k_0}.} The proof follows the line of proof in \cite{FZ98}, with some noticeable technical differences. More precisely, we intend to apply \cite{H84} to a truncation of $Y_n$ expressed as the series \eqref{proof_step_2_CTL_k0}.
\paragraph{$\diamond$  Step 1:} we prove that
$\mathfrak{S}_{k_0}$ defined in  \eqref{def_frac_sigma} is a convergent series.
To prove this, and in view of \eqref{def_gamma}, we prove here only that $\sum_{j=1}^\infty \mbox{Cov}(Y_0Y_{k_0}, Y_jY_{j+k_0})$ converges, the series involving the other terms in \eqref{def_gamma} being proved similarly. Using \eqref{decompo1} with $k:=k_0$, $m:=j+k_0$,  $\varphi_k^-:=Y_0 Y_{k_0}$,  yields for $j>k_0$
\begin{multline}\label{proof_step_1_CTL_k0}
\mbox{Cov}(Y_0Y_{k_0}, Y_jY_{j+k_0})=\sum_{s=0}^{k_0-1}\lambda_0^s \left[ \sum_{t=0}^{j-k_0-1} \lambda_0^t \mbox{Cov}(Y_0Y_{k_0}, \epsilon_{j-t} \epsilon_{j+k_0-s}) + \lambda_0^{j-k_0}
\mbox{Cov}(Y_0Y_{k_0}, Y_{k_0} \epsilon_{j+k_0-s}) \right]\\
 + \lambda_0^{k_0} \mbox{Cov}(Y_0Y_{k_0}, Y_j^2),
\end{multline}
of which different terms are appropriately upper bounded. Applying inequalities \eqref{bound_sum_Davydov0}, the Cauchy Schwarz inequality (under the assumptions $\mathbf{ (A2)}$ and thanks to Proposition \ref{prop_stationarity}) and \eqref{bound_sum_Davydov1} give respectively
\begin{eqnarray*}
\sum_{t=0}^{j-k_0-1} \lambda_0^t \left|\mbox{Cov}(Y_0Y_{k_0}, \epsilon_{j-t} \epsilon_{j+k_0-s})\right| & \le & K\left[ \alpha_\epsilon\left( \left\lfloor \frac{j-k_0-1}{2}\right\rfloor\right)^{1-2/\beta} + \lambda_0^{\left\lfloor \frac{j-k_0-1}{2}\right\rfloor}\right],\\
\left|\mbox{Cov}(Y_0Y_{k_0}, Y_{k_0} \epsilon_{j+k_0-s}) \right| &\le & K,\\
\left|\mbox{Cov}(Y_0Y_{k_0}, Y_j^2)\right| & \le &  K\left[ \alpha_\epsilon\left( \left\lfloor \frac{j-k_0-1}{2}\right\rfloor\right)^{1-2/\beta} + \lambda_0^{\left\lfloor \frac{j-k_0-1}{2}\right\rfloor}\right],
\end{eqnarray*}
for some constant $K>0$, which, plugged into \eqref{proof_step_1_CTL_k0}, gives
\begin{equation}\label{ineg_utile}
\left|\mbox{Cov}(Y_0Y_{k_0}, Y_jY_{j+k_0})\right| \le K \left[ \alpha_\epsilon\left( \left\lfloor \frac{j-k_0-1}{2}\right\rfloor\right)^{1-2/\beta} + \lambda_0^{\left\lfloor \frac{j-k_0-1}{2}\right\rfloor} + \lambda_0^{j-k_0}\right]
\end{equation}
for $j>k_0$, yielding that $\sum_{j=1}^\infty |\mbox{Cov}(Y_0Y_{k_0}, Y_jY_{j+k_0})|$ is indeed finite thanks to Condition \eqref{stab} and Assumption $\mathbf{ (A1)}$.
\paragraph{$\diamond$ Step 2:} Recall that $Y_n$ may be expanded as \eqref{proof_step_2_CTL_k0}.
We then further decompose \eqref{proof_step_2_CTL_k0} for all $m\in \nbN$ into
\begin{equation}\label{proof_step_2_CTL_k0_2_0}
Y_n=Y_{n,m} + Z_{n,m},\quad Y_{n,m}:=\sum_{i=0}^m \theta^{(i)}_n\circ\epsilon_{n-i},\quad Z_{n,m}:=\sum_{i=m+1}^\infty \theta^{(i)}_n\circ\epsilon_{n-i},
\end{equation}
and ${\cal Y}_n(k_0)=\left[Y_n-C_1, Y_nY_{n+k_0-1}-u_{k_0-1}, Y_nY_{n+k_0}-u_{k_0}\right]'$ into
\begin{eqnarray}
{\cal Y}_n(k_0)&=&{\cal Y}_{n,m}(k_0) + {\cal Z}_{n,m}(k_0),\label{proof_step_2_CTL_k0_2}\\
{\cal Y}_{n,m}(k_0)&:= &\left[Y_{n,m}-\nbE(Y_{n,m}), Y_{n,m}Y_{n+k_0-1,m}-\nbE(Y_{n,m}Y_{n+k_0-1,m}), \right. \nonumber \\
&& \left. Y_{n,m}Y_{m+k_0,m}-\nbE(Y_{n,m}Y_{n+k_0,m})\right]',\nonumber\\
 {\cal Z}_{n,m}(k_0)&:= & \left[Z_{n,m}-\nbE(Z_{n,m}), Y_{n,m}Z_{n+k_0-1,m}+ Z_{n,m}Y_{n+k_0-1,m}+ Z_{n,m}Z_{n+k_0-1,m}\right. \nonumber \\
  && -\nbE(Y_{n,m}Z_{n+k_0-1,m}+ Z_{n,m}Y_{n+k_0-1,m}+ Z_{n,m}Z_{n+k_0-1,m}), \nonumber \\
 && Y_{n,m}Z_{n+k_0,m}+ Z_{n,m}Y_{n+k_0,m}+ Z_{n,m}Z_{n+k_0,m} \nonumber \\
 && \left. -\nbE(Y_{n,m}Z_{n+k_0,m}+ Z_{n,m}Y_{n+k_0,m}+ Z_{n,m}Z_{n+k_0,m}),\right]'\nonumber
\end{eqnarray}
for all $m\in \nbN$. We also recall from the proof of Proposition \ref{prop_stationarity} that $\|\theta^{(i)}_n\circ\epsilon_{n-i}\|_{2\beta}=\|\theta^{(i)}\circ\epsilon_{-i}\|_{2\beta}$ is an $\mathrm{O}(\upsilon^i)$ for some $\upsilon<1$, so that $\|Y_n\|_{2\beta}$, $\|Y_{n,m}\|_{2\beta}$, $\|Z_{n,m}\|_{2\beta}$ are finite for all $n\in \nbZ$ and $m\in \nbN$. We prove in this Step that 
\begin{equation}\label{conv_distribution_m}
\frac{1}{\sqrt{n}}\sum_{i=1}^n {\cal Y}_{i,m}(k_0)\xrightarrow[n\to\infty]{\mathcal{D}} {\cal N}(0,\mathfrak{S}_{k_0}^m),
\end{equation}
for some semi-definite positive matrix $\mathfrak{S}_{k_0}^m\in \nbR^{3\times 3}$, for all $m\in \nbN$ given by
\begin{equation}\label{proof_step_2_CTL_k0_def_sigma}
\mathfrak{S}_{k_0}^m:=\mbox{Var}({\cal Y}_{0,m}(k_0))+\sum_{h=1}^\infty [\mbox{Cov}({\cal Y}_{0,m}(k_0),{\cal Y}_{h,m}(k_0))+\mbox{Cov}({\cal Y}_{0,m}(k_0),{\cal Y}_{h,m}(k_0))']
\end{equation}
For this we verify Assumption (1.2) in \cite{H84} which says that the following limit exists
\begin{multline}\label{proof_step_2_CTL_k0_3}
\frac{1}{n}\nbE\left( \left( \sum_{i=1}^n {\cal Y}_{i,m}(k_0)\right)\left( \sum_{i=1}^n {\cal Y}_{i,m}(k_0)\right)'\right)=\frac{1}{n}\nbE\left( \sum_{1\le i,j \le n}{\cal Y}_{i,m}(k_0)'{\cal Y}_{j,m}(k_0)\right)\\
=\frac{1}{n}\sum_{h=-n}^n (n-|h|) \mbox{Cov} ({\cal Y}_{0,m}(k_0),{\cal Y}_{h,m}(k_0))
\xrightarrow[n\to\infty]{} \mathfrak{S}_{k_0}^m,
\end{multline}
as indeed, if we prove that $\mathfrak{S}_{k_0}^m$ in \eqref{proof_step_2_CTL_k0_def_sigma} converges absolutely then the limit \eqref{proof_step_2_CTL_k0_3} exists by the dominated convergence theorem.
Then, under Assumption {\bf (A1)} and  $m\in\nbN$ fixed, we
prove that the process $(\mathcal{Y}_{n,m}(k_0))_{n\in\mathbb{Z}}$ is strongly mixing (see \citet[ Theorem 14.1]{D94}), i.e. that the following condition holds:
\begin{equation}\label{proof_step_2_CTL_k0_4}
\sum_{h=1}^\infty \alpha_{m,{\cal Y}}(h)^{1-2/\beta}<\infty
\end{equation}
where $\alpha_{m,{\cal Y}}(h):= \sup\{ |\PP(A\bigcap B)- \PP(A)\PP(B)|,\ A\in \sigma({\cal Y}_{n,m}(k_0),\ 1\le n\le r ),\ B\in \sigma({\cal Y}_{n,m}(k_0),\ n\ge r+h)\}$ is the mixing coefficient associated to the sequence $\{{\cal Y}_{i,m}(k_0),\ i\in \nbN \}$.

Now, in view of the definition \eqref{proof_step_2_CTL_k0_2} of ${\cal Y}_{n,m}(k_0)$,  we first prove the following limit
\begin{equation}\label{proof_step_2_CTL_k0_6}
\sum_{h=0}^\infty \left| \mbox{Cov}(Y_{0,m}Y_{k_0-t_1,m}^{t_2}, Y_{h,m}Y_{h+k_0-t_1,m}^{t_2})-\mbox{Cov}(Y_{0}Y_{k_0-t_1}^{t_2}, Y_{h}Y_{h+k_0-t_1}^{t_2}) \right| \xrightarrow[m\to\infty]{} 0, 
\end{equation}
for all $t_1$, $t_2$ in $\{0,1\}$. Indeed, one deduces easily from this both that $\mathfrak{S}_{k_0}^m$ in \eqref{proof_step_2_CTL_k0_def_sigma} is finite for all $m$, and that it converges to $\mathfrak{S}_{k_0}$ defined in \eqref{def_frac_sigma}. We prove \eqref{proof_step_2_CTL_k0_6} only for $t_1=0$ and $t_2=1$, the other cases being similar. For these $t_1$ and $t_2$, we decompose the summands in the above series as
\begin{multline*}
\mbox{Cov}(Y_{0,m}Y_{k_0,m}, Y_{h,m}Y_{h+k_0,m})-\mbox{Cov}(Y_{0}Y_{k_0}, Y_{h}Y_{h+k_0})
= \mbox{Cov}([Y_{0,m}-Y_0]Y_{k_0,m}, Y_{h,m}Y_{h+k_0,m})\\
+ \mbox{Cov}(Y_{0}[Y_{k_0,m}-Y_{k_0}], Y_{h,m}Y_{h+k_0,m}) + \mbox{Cov}(Y_{0}Y_{k_0}, [Y_{h,m}-Y_h]Y_{h+k_0,m})+\mbox{Cov}(Y_{0}Y_{k_0}, Y_{h}[Y_{h+k_0,m}-Y_{h+k_0}])
\end{multline*}
and we focus on the first term in the righhandside, the other terms being tackled with similarly, so that we prove that
\begin{equation}\label{proof_step_2_CTL_k0_7}
\sum_{h=0}^\infty \left|\mbox{Cov}([Y_{0,m}-Y_0]Y_{k_0,m}, Y_{h,m}Y_{h+k_0,m})\right| \xrightarrow[m\to\infty]{} 0. 
\end{equation}
For this, we observe that, letting
\begin{equation}\label{def_phi_minus_k_0}
\varphi^-_{k_0,m}:= Z_{0,m}Y_{k_0,m},
\end{equation}
and in view of \eqref{proof_step_2_CTL_k0_2_0} and for $h/2>k_0$:
\begin{eqnarray}
-\mbox{Cov}([Y_{0,m}-Y_0]Y_{k_0,m}, Y_{h,m}Y_{h+k_0,m})&=& \mbox{Cov}(Z_{0,m}Y_{k_0,m}, Y_{h,m}Y_{h+k_0,m}) = \mbox{Cov}(\varphi^-_{k_0,m}, Y_{h,m}Y_{h+k_0,m})\nonumber\\
&=& \sum_{i_1=0}^m \sum_{i_2=0}^m \mbox{Cov} \left(\varphi^-_{k_0,m}, [\theta^{(i_1)}_h\circ\epsilon_{h-i_1}][\theta^{(i_2)}_{h+k_0}\circ\epsilon_{h+k_0-i_2}]\right),\label{proof_step_2_CTL_k0_8}\\\nonumber
&=& \sum_{j=1}^4 I_j(h,m),
\end{eqnarray}
where
\begin{eqnarray*}
I_1(h,m)&=& \sum_{\substack{0\le i_1 \le \min (\lfloor h/2 \rfloor - k_0,m)\\0 \le i_2 \le \min(\lfloor h/2 \rfloor,m)}} \mbox{Cov} \left(\varphi^-_{k_0,m}, [\theta^{(i_1)}_h\circ\epsilon_{h-i_1}][\theta^{(i_2)}_{h+k_0}\circ\epsilon_{h+k_0-i_2}]\right), \nonumber\\
I_2(h,m)&=& \sum_{\substack{\lfloor h/2 \rfloor - k_0<i_1 \le m\\0 \le i_2 \le \lfloor h/2 \rfloor}} \mbox{Cov} \left(\varphi^-_{k_0,m}, [\theta^{(i_1)}_h\circ\epsilon_{h-i_1}][\theta^{(i_2)}_{h+k_0}\circ\epsilon_{h+k_0-i_2}]\right), \nonumber\\
I_3(h,m)&=& \sum_{\substack{0\le i_1 \le \lfloor h/2 \rfloor - k_0\\  \lfloor h/2 \rfloor < i_2 \le m}} \mbox{Cov} \left(\varphi^-_{k_0,m}, [\theta^{(i_1)}_h\circ\epsilon_{h-i_1}][\theta^{(i_2)}_{h+k_0}\circ\epsilon_{h+k_0-i_2}]\right), \nonumber\\
I_4(h,m)&=& \sum_{\substack{\lfloor h/2 \rfloor - k_0<i_1 \le m\\  \lfloor h/2 \rfloor < i_2 \le m}} \mbox{Cov} \left(\varphi^-_{k_0,m}, [\theta^{(i_1)}_h\circ\epsilon_{h-i_1}][\theta^{(i_2)}_{h+k_0}\circ\epsilon_{h+k_0-i_2}]\right), \nonumber
\end{eqnarray*}
where a sum over an empty set is equal to zero by convention.
We notice that $\varphi^-_{k_0,m}$ depends on $(\theta^{(i)}_0)_{i\in \nbN}$, $(\theta^{(i)}_{k_0})_{i\in \nbN}$ and $\epsilon_j$, $j\le k_0$. Also, we recall from \eqref{op_stationary} that $\theta^{(i)}_0$ and $\theta^{(i)}_{k_0}$ are independent from $\theta^{(i_1)}_h$ and $\theta^{(i_2)}_{h+k_0}$ for all $i\in \nbN$ as soon as $[\{-i,k_0 -i\}\cap \{ h-i_1,h+k_0-i_2\} =\emptyset,\ \forall i\in \nbN]\iff [k_0<h-i_1,\ k_0<h+k_0-i_2]\iff[i_1<h-k_0,\ i_2<h]$. The mixing coefficient
$\alpha_{\theta}(h):= \sup\{ |\PP(A\bigcap B)- \PP(A)\PP(B)|,\ A\in \sigma((\theta^{(i)}_n)_{i\in \nbN}, 1 \le n\le r),\ B\in \sigma( \theta^{(i)}_n, n-i+1\ge r +h)\}$ verifies $\alpha_\theta(h)=0$ as soon as $h\ge 1$ for similar independence reasons. And, we remember that $\{ (\theta^{(i)}_n)_{i\in \nbN}, \ n\in \nbZ\}$ and $\{\epsilon_n,\ n\in \nbZ \}$ are independent. Thus, the following upper bound holds, using Davydov's inequality (see \cite{D68}):
\begin{multline}
\left|I_1(h,m)\right| \le \sum_{\substack{0\le i_1 \le \lfloor h/2 \rfloor - k_0\\0 \le i_2 \le \lfloor h/2 \rfloor}} K \left\| \varphi^-_{k_0,m} - \nbE (\varphi^-_{k_0,m})\right\|_\beta \\ \left\| [\theta^{(i_1)}_h\circ\epsilon_{h-i_1}][\theta^{(i_2)}_{h+k_0}\circ\epsilon_{h+k_0-i_2}] - \nbE \left([\theta^{(i_1)}_h\circ\epsilon_{h-i_1}][\theta^{(i_2)}_{h+k_0}\circ\epsilon_{h+k_0-i_2}]\right) \right\|_\beta \;
\alpha_\epsilon(\lfloor h/2 \rfloor)^{1-2/\beta}\ , \label{proof_step_2_CTL_k0_9}
\end{multline}
for some constant $K$. The terms involved on the right-hand side of the above inequality are bounded as follows. First, thanks to Minkoswki's inequality followed by Cauchy Schwarz as well as \eqref{bound_KW}, we get
\begin{eqnarray}
\left\| \varphi^-_{k_0,m} - \nbE (\varphi^-_{k_0,m})\right\|_\beta &\le &\left\|Z_{0,m} Y_{k_0,m} \right\|_{\beta} + \left|\nbE(Z_{0,m} Y_{k_0,m})\right|\\ &\le&
\left\|Z_{0,m} \right\|_{2\beta} \left\|Y_{k_0,m} \right\|_{2\beta} + \left\|Z_{0,m} \right\|_{2} \left\|Y_{k_0,m} \right\|_{2}  \nonumber\\
&\le & \left( \sum_{i=m+1}^\infty  \left\|\theta^{(i)}\circ\epsilon_{-i}\right\|_{2\beta} \right) \left( \sum_{i=0}^m \left\|\theta^{(i)}\circ\epsilon_{-i}\right\|_{2\beta} \right)  \nonumber\\
&& + \left( \sum_{i=m+1}^\infty  \left\|\theta^{(i)}\circ\epsilon_{-i}\right\|_{2} \right) \left( \sum_{i=0}^m \left\|\theta^{(i)}\circ\epsilon_{-i}\right\|_{2} \right) =\mathrm{O}(\upsilon^m).\label{proof_step_2_CTL_k0_10}
\end{eqnarray}
Similarly,
\begin{eqnarray}
\left\| \left[\theta^{(i_1)}_h\circ\epsilon_{h-i_1}\right]\left[\theta^{(i_2)}_{h+k_0}\circ\epsilon_{h+k_0-i_2}\right] - \nbE \left(\left[\theta^{(i_1)}_h\circ\epsilon_{h-i_1}\right]\left[\theta^{(i_2)}_{h+k_0}\circ\epsilon_{h+k_0-i_2}\right]\right) \right\|_\beta &\le & \mathrm{O}(\upsilon^{i_1+i_2}). \label{proof_step_2_CTL_k0_11}
\end{eqnarray}
Plugging \eqref{proof_step_2_CTL_k0_10} and \eqref{proof_step_2_CTL_k0_11} into \eqref{proof_step_2_CTL_k0_9} yields their upper bound
\begin{equation}\label{proof_step_2_CTL_k0_12}
\left|I_1(h,m)\right| \le K \upsilon^m \alpha_\epsilon\left(\lfloor h/2 \rfloor\right)^{1-2/\beta}, 
\end{equation}
for some constant $K$. Similarly, using again \eqref{proof_step_2_CTL_k0_10} and \eqref{proof_step_2_CTL_k0_11} yields
\begin{eqnarray}
\left|I_2(h,m)\right| &\le & K \upsilon^{\lfloor h/2 \rfloor} \upsilon^m , \label{proof_step_2_CTL_k0_13}\\
\left|I_3(h,m)\right| &\le & K \upsilon^{\lfloor h/2 \rfloor} \upsilon^m , \label{proof_step_2_CTL_k0_14}\\
\left|I_4(h,m)\right| &\le & K \upsilon^{2\lfloor h/2 \rfloor} \upsilon^m , \label{proof_step_2_CTL_k0_15}
\end{eqnarray}
for some constant $K$. Using \eqref{proof_step_2_CTL_k0_8} and plugging the bounds, \eqref{proof_step_2_CTL_k0_12}, \eqref{proof_step_2_CTL_k0_13}, \eqref{proof_step_2_CTL_k0_14} and \eqref{proof_step_2_CTL_k0_15} we thus obtain
\begin{align}\nonumber
\sum_{h/2>k_0}& \left|\mbox{Cov}([Y_{0,m}-Y_0]Y_{k_0,m}, Y_{h,m}Y_{h+k_0,m})\right|\\ \label{proof_step_2_CTL_k0_16}&\le K \upsilon^m  \sum_{h/2>k_0}  \left[ \alpha_\epsilon\left(\lfloor h/2 \rfloor\right)^{1-2/\beta}  + \upsilon^{\lfloor h/2 \rfloor} + \upsilon^{2\lfloor h/2 \rfloor}\right] \xrightarrow[m\to\infty]{} 0, 
\end{align}
for some constant $K$. Since it is easy to check that $$\left|\mbox{Cov}\left([Y_{0,m}-Y_0]Y_{k_0,m}, Y_{h,m}Y_{h+k_0,m}\right)\right|\xrightarrow[m\to\infty]{} 0,$$  for the finite number of terms $h$ such that $h/2 \le k_0$, one then deduces that \eqref{proof_step_2_CTL_k0_7} holds. And all in all, we thus proved \eqref{proof_step_2_CTL_k0_6}, ensuring that $\mathfrak{S}_{k_0}^m\longrightarrow \mathfrak{S}_{k_0}$ as $m\to\infty$.\\
In order to prove \eqref{proof_step_2_CTL_k0_4}, we notice from \eqref{proof_step_2_CTL_k0_2} that ${\cal Y}_{n,m}(k_0)$ depends on $\epsilon_i$, $i=n-m,...,n+k_0$, and $\theta_n^{(i)}$, $\theta_{n+k_0-1}^{(i)}$, $\theta_{n+k_0}^{(i)}$, $i=0,...,m$. We already observed that the operators $\theta^{(i)}_q$, $i\in\nbN$, $q\in \nbZ$ and $\theta^{(j)}_p$, $j\in\nbN$, $p\in \nbZ$ are independent as soon as $q-i\neq p-j$. Remembering that $\{ (\theta^{(i)}_n)_{i\in \nbN}, \ n\in \nbZ\}$ and $\{\epsilon_n,\ n\in \nbZ \}$ are independent, we hence obtain that
\begin{equation}\label{mixing_calY}
\alpha_{m,{\cal Y}}(h) \le \alpha_\epsilon (h-k_0-m)
\end{equation}
as soon as $k_0+m<h$, so that \eqref{proof_step_2_CTL_k0_4} holds thanks to Assumption $\mathbf{ (A1)}$. 
All in all, the finiteness of $\mathfrak{S}_{k_0}^m$ and \eqref{proof_step_2_CTL_k0_4} imply the convergence \eqref{conv_distribution_m} thanks to \citet[Corollary 1]{H84}. 

\paragraph{$\diamond$ Step 3:} We then prove that
\begin{equation}\label{proof_step_2_CTL_k0_17}
\lim_{m\to \infty}\limsup_{n\to \infty}\nbP\left( \left\|  \frac{1}{\sqrt{n}} \sum_{i=1}^n {\cal Z}_{i,m}(k_0) \right\|> \epsilon\right)=0,\quad \forall \epsilon >0,
\end{equation}
where the norm $\|\cdot\|$ is defined here on $\nbR^3$ by $\|u\|^2=u'u$, $u\in\nbR^3$ a column vector. The result \eqref{proof_step_2_CTL_k0_17} follows from a
straightforward adaptation of \citet[Theorem 7.7.1 and Corollary 7.7.1, pages 425-426]{anderson}. Indeed, Markov's inequality as well as the stationarity of $\{ {\cal Z}_{i,m}(k_0),\ i\in \nbZ\}$ yields
\begin{multline}\label{proof_step_2_CTL_k0_18}
\nbP\left( \left\| \frac{1}{\sqrt{n}} \sum_{i=1}^n {\cal Z}_{i,m}(k_0) \right\|> \epsilon\right) \le \frac{1}{n\epsilon^2}\nbE\left( \left\| \sum_{i=1}^n {\cal Z}_{i,m}(k_0) \right\|^2\right)
= \frac{1}{n\epsilon^2}\nbE\left( \sum_{1\le i,j \le n}{\cal Z}_{i,m}(k_0)'{\cal Z}_{j,m}(k_0)\right)\\
\xrightarrow[n\to\infty]{} \frac{1}{\epsilon^2}\sum_{h=-\infty}^\infty \nbE\left( {\cal Z}_{0,m}(k_0)'{\cal Z}_{h,m}(k_0)\right),
\end{multline}
where the latter limit is obtained similarly as in \eqref{proof_step_2_CTL_k0_3}, provided that $\sum_{h=1}^\infty \nbE\left( {\cal Z}_{0,m}(k_0)'{\cal Z}_{h,m}(k_0)\right)$ is a convergent series that tends to $0$ as $m\to \infty$. In view of the definition \eqref{proof_step_2_CTL_k0_2_0}, this is the case if we prove that the following series:

\begin{multline*}
\sum_{h=1}^\infty \mbox{Cov}(Z_{0,m}Y_{k_0-t_2,m}^{t_1}, \Upsilon_{h,m}^{t_3}),\quad \sum_{h=1}^\infty\mbox{Cov}(Y_{0,m}Z_{k_0-t_2,m}, \Upsilon_{h,m}^{t_3}), \quad \sum_{h=1}^\infty\mbox{Cov}(Z_{0,m}Z_{k_0-t_2,m}, \Upsilon_{h,m}^{t_3}),\\
\mbox{where }\Upsilon_{h,m}^{t_3} \in \{ Z_{h,m}Y_{h+k_0-t_3,m}, Y_{h,m}Z_{h+k_0-t_3,m}, Z_{h,m}Y_{h+k_0-t_3,m}\},
\end{multline*}
converge to $0$ as $m\to \infty$ for $t_1$,$t_2$ and $t_3$ in $\{0,1\}$. We proceed to do prove this convergence for the series with general term $\mbox{Cov}(Z_{0,m}Z_{k_0,m}, Z_{h,m}Z_{h+k_0,m})$, the other cases being similar. As previously while studying \eqref{proof_step_2_CTL_k0_8}, we write that
$$
\mbox{Cov}(Z_{0,m}Z_{k_0,m}, Z_{h,m}Z_{h+k_0,m})=  \sum_{i_1=m+1}^\infty \sum_{i_2=m+1}^\infty \mbox{Cov} \left(Z_{0,m}Z_{k_0,m}, [\theta^{(i_1)}_h\circ\epsilon_{h-i_1}][\theta^{(i_2)}_{h+k_0}\circ\epsilon_{h+k_0-i_2}]\right).
$$
Splitting the above double sum into $i_1$ and $i_2$ less or larger than $h/2$, the following equivalent of \eqref{proof_step_2_CTL_k0_16} may be obtained:
\begin{equation}\label{proof_step_2_CTL_k0_19}
\sum_{h/2>k_0} \left|\mbox{Cov}(Z_{0,m}Z_{k_0,m}, Z_{h,m}Z_{h+k_0,m})\right|\le K \upsilon^m  \sum_{h/2>k_0}  \left[ \alpha_\epsilon(\lfloor h/2 \rfloor)^{1-2/\beta}  + \upsilon^{\lfloor h/2 \rfloor} + \upsilon^{2\lfloor h/2 \rfloor}\right] =\mathrm{O}(\upsilon^m)
\end{equation}
for some constant $K>0$. It can be also proved that for $h/2 \le k_0$ we have $\left|\mbox{Cov}(Z_{0,m}Z_{k_0,m}, Z_{h,m}Z_{h+k_0,m})\right|=\mathrm{O}(\upsilon^m)$, which, combined with \eqref{proof_step_2_CTL_k0_19}, yields
$\sum_{h=1}^\infty \left|\mbox{Cov}(Z_{0,m}Z_{k_0,m}, Z_{h,m}Z_{h+k_0,m})\right|= \mathrm{O}(\upsilon^m),$
from which we in turn get
\begin{equation}\label{Step3_boundZ}
\sum_{h=-\infty}^\infty \nbE\left( {\cal Z}_{0,m}(k_0)'{\cal Z}_{h,m}(k_0)\right)= \mathrm{O}(\upsilon^m).
\end{equation}
Combining the above with \eqref{proof_step_2_CTL_k0_18} thus yields \eqref{proof_step_2_CTL_k0_17}.
\paragraph{$\diamond$ Step 4:} 
Combining \eqref{conv_distribution_m} and \eqref{proof_step_2_CTL_k0_17} yields by \citet[Proposition 6.3.9 p.208]{BD91} the final result \eqref{CLT_vector_k_0}. \zak
As a side result of the previous proof, let us prove that the following convergence holds:
\begin{equation}\label{convergence_Herrndorf}
\frac{1}{n}\nbE(T_n(k_0)T_n(k_0)')\xrightarrow[n\to\infty]{} \mathfrak{S}_{k_0}
\end{equation}
which will be useful later on. For this, we write, remembering the notation \eqref{def_estimators} and \eqref{proof_step_2_CTL_k0_2},
\begin{multline}\label{conv_Herrnorf1}
\left\| \EE \left( \frac{T_n(k_0)T_n(k_0)'}{n}\right) - \mathfrak{S}_{k_0}\right\|  \le
\left\| \EE \left( \frac{T_n(k_0)T_n(k_0)'}{n}\right) - \mathfrak{S}_{k_0}^m\right\| 
+\left\| \mathfrak{S}_{k_0}^m- \mathfrak{S}_{k_0}\right\|    \\
\le \left\| \EE \left(\left[   \frac{1}{\sqrt{n}} \sum_{i=1}^n {\cal Y}_{i,m}(k_0)+    
\frac{1}{\sqrt{n}} \sum_{i=1}^n {\cal Z}_{i,m}(k_0)
\right] \left[  \frac{1}{\sqrt{n}} \sum_{i=1}^n {\cal Y}_{i,m}(k_0)+    
\frac{1}{\sqrt{n}} \sum_{i=1}^n {\cal Z}_{i,m}(k_0)
\right]'\right)  - \mathfrak{S}_{k_0}^m\right\|\\
+ \left\| \mathfrak{S}_{k_0}^m- \mathfrak{S}_{k_0}\right\|\\
\le \left\| \EE \left(    \left[  \frac{1}{\sqrt{n}} \sum_{i=1}^n {\cal Y}_{i,m}(k_0)
\right] \left[  \frac{1}{\sqrt{n}} \sum_{i=1}^n {\cal Y}_{i,m}(k_0)
\right]'\right) - \mathfrak{S}_{k_0}^m\right\|\\
+ 2 \left\| \EE \left(\left[  \frac{1}{\sqrt{n}} \sum_{i=1}^n {\cal Y}_{i,m}(k_0)
\right] \left[  
\frac{1}{\sqrt{n}} \sum_{i=1}^n {\cal Z}_{i,m}(k_0)
\right]'\right)\right\|+  \left\| \mathfrak{S}_{k_0}^m- \mathfrak{S}_{k_0}\right\| .
\end{multline}
We now observe that, thanks to \eqref{proof_step_2_CTL_k0_3},
\begin{equation}\label{conv_Herrnorf2}
\left\| \EE \left(    \left[  \frac{1}{\sqrt{n}} \sum_{i=1}^n {\cal Y}_{i,m}(k_0)
\right] \left[  \frac{1}{\sqrt{n}} \sum_{i=1}^n {\cal Y}_{i,m}(k_0)
\right]'\right) - \mathfrak{S}_{k_0}^m\right\| \xrightarrow[n\to\infty]{} 0,\quad \forall m\in \nbN .
\end{equation}
Next, we have that $ \left\| \frac{1}{\sqrt{n}}\sum_{i=1}^n {\cal Y}_{i,m}(k_0)\right\|_2 $ is upper bounded by a constant $C$ independent from $n$ and $m$ thanks to \eqref{proof_step_2_CTL_k0_3} and the fact that $\mathfrak{S}_{k_0}^m$ converges to the finite limit $\mathfrak{S}_{k_0}$ as $m\to\infty$. Also a byproduct of \eqref{proof_step_2_CTL_k0_18} and \eqref{Step3_boundZ} is that $\lim_{n\to\infty} \left\|\frac{1}{\sqrt{n}} \sum_{i=1}^n {\cal Z}_{i,m}(k_0)\right\|_2= \sum_{h=-\infty}^\infty \nbE\left( {\cal Z}_{0,m}(k_0)'{\cal Z}_{h,m}(k_0)\right)= \mathrm{O}(\upsilon^m)$. Thus, Cauchy Schwarz's inequality yields that
\begin{multline}\label{conv_Herrnorf3}
\left\| \EE \left(\left[  \frac{1}{\sqrt{n}} \sum_{i=1}^n {\cal Y}_{i,m}(k_0)
\right] \left[  
\frac{1}{\sqrt{n}} \sum_{i=1}^n {\cal Z}_{i,m}(k_0)
\right]'\right)\right\|\le 
 \EE \left(\left\|  \frac{1}{\sqrt{n}} \sum_{i=1}^n {\cal Y}_{i,m}(k_0)
\right\| \; \left\|
\frac{1}{\sqrt{n}} \sum_{i=1}^n {\cal Z}_{i,m}(k_0)
\right\|\right)\\
\le \left\| \frac{1}{\sqrt{n}}\sum_{i=1}^n {\cal Y}_{i,m}(k_0)\right\|_2 \left\|\frac{1}{\sqrt{n}} \sum_{i=1}^n {\cal Z}_{i,m}(k_0)\right\|_2\le C \left\|\frac{1}{\sqrt{n}} \sum_{i=1}^n {\cal Z}_{i,m}(k_0)\right\|_2\\
\xrightarrow[n\to\infty]{} C \sum_{h=-\infty}^\infty \nbE\left( {\cal Z}_{0,m}(k_0)'{\cal Z}_{h,m}(k_0)\right)= \mathrm{O}(\upsilon^m).
\end{multline}
Thus, \eqref{conv_Herrnorf2}, \eqref{conv_Herrnorf3} along with $\lim_{m\to\infty}\mathfrak{S}_{k_0}^m=\mathfrak{S}_{k_0}$ yields from \eqref{conv_Herrnorf1} that 
$$
\limsup_{n\to\infty} \left\| \EE \left( \frac{T_n(k_0)T_n(k_0)'}{n}\right) - \mathfrak{S}_{k_0}\right\|  \le \mathrm{O}(\upsilon^m)+ \left\| \mathfrak{S}_{k_0}^m- \mathfrak{S}_{k_0}\right\| \xrightarrow[m\to\infty]{} 0,
$$
which implies \eqref{convergence_Herrndorf}.
\subsection{Proof of Proposition \ref{cond_psi_invertible} and Theorem \ref{convergence_Isp}}\label{app:proof_prop_theo_psi}
{\bf Proof of Proposition \ref{cond_psi_invertible}. }
Remembering the expression \eqref{def_estimators}, we recall that ${\cal Y}_n={\cal Y}_n(k_0)=(Y_n-C_1,Y_nY_{n+k_0-1}-u_{k_0-2}, Y_nY_{n+k_0}-u_{k_0-1})' $. Writing the $\nbR^3$ vector $\Sigma_\varepsilon^{-1}\varepsilon_{0}$ in the form $\Sigma_\varepsilon^{-1}\varepsilon_{0}= \left(\varepsilon_{0}^{(1)}, \varepsilon_{0}^{(2)}, \varepsilon_{0}^{(3)}\right)'$, we thus have from \eqref{Wold_expressions} that
$$
\Psi_n=\mbox{Cov}({\cal Y}_n, \Sigma_\varepsilon^{-1}\varepsilon_{0}),
$$
which is a $\nbR^{3\times 3}$ matrix. We first prove that
\begin{equation}\label{convergence_series_psi}
\sum_{n=1}^\infty \| \Psi_n \| < \infty .
\end{equation}
The first row, third line of $\| \Psi_n \|$ is equal to $\mbox{Cov}(\varepsilon_{0}^{(1)}, Y_nY_{n+k_0})$. 
Since $\varepsilon_{0}^{(1)}$ depends on $Y_j$, $j\le k_0$, an argument similar to \eqref{ineg_utile} yields that
$$\left|\mbox{Cov}(\varepsilon_{0}^{(1)}, Y_nY_{n+k_0})\right| \le K \left[ \alpha_\epsilon\left( \left\lfloor \frac{n-k_0-1}{2}\right\rfloor\right)^{1-2/\beta} + \lambda_0^{\left\lfloor \frac{n-k_0-1}{2}\right\rfloor} + \lambda_0^{n}\right]$$
And, more generally one obtains a similar bound for all coefficients of $\Psi_n$, which thus implies \eqref{convergence_series_psi}.\\
We next introduce the spectral density associated to the process $\{\mathcal{Y}_n,\ n\in\mathbb{Z}\}:=\{\mathcal{Y}_n(k_0),\ n\in\mathbb{Z}\}$ defined as
\begin{equation}\label{def_spectral_density}
f_\Y(x):=\frac{1}{2\pi}\sum_{h=-\infty}^\infty C(h)e^{ihx}\quad \mbox{where } C(h):=\nbE(\Y_0 \Y_h ' )=\mbox{Cov}(\Y_0,\Y_h).
\end{equation}
We note that the above series is indeed convergent thanks to arguments similar to the ones leading to the convergence of \eqref{convergence_series_psi}. Also note that, still thanks to that latter convergence, we have the relation
\begin{equation}\label{relation_spectral_psi}
f_\Y(x)= \overline{\Psi(e^{ix})} \Sigma_\varepsilon \Psi(e^{ix}) ' ,\quad \forall x\in \nbR ,
\end{equation}
where we recall that $\Psi(\cdot)$ is defined by \eqref{def_psi_op}. We first prove that 
\begin{equation}\label{A4_light}
\det\left(\Psi(z)\right)\neq 0,\quad \forall z\in {\mathbb{C}} \mbox{ such that }\left|z\right| =1.
\end{equation}
In view of \eqref{relation_spectral_psi}, the above is satisfied if and only if we show that $f_\Y(x)$ is invertible for all $x\in [0, 2\pi]$, which we proceed to do if $\lambda_0$ is small enough. For this we will momentarily write the dependence of the spectral density of $\{\mathcal{Y}_n,\ n\in\mathbb{Z}\}$ and write $f_\Y(\lambda_0,x)$ instead of $f_\Y(x)$, and we are going to prove that 
\begin{equation}\label{limitt_spectral_density}
\sup_{x\in [0, 2\pi]} \|f_\Y(\lambda_0,x) - f_{\underline{\epsilon}}(x)\|\longrightarrow 0,\quad \lambda_0 \to 0. 
\end{equation}
In view of the definitions of those spectral densities, it is not difficult to check that this amounts to show that the quantities 
\begin{equation}\label{series_that_tend_0}
\begin{array}{c}
\sum_{h=-\infty}^\infty |\mbox{Cov}(Y_0,Y_h) - \mbox{Cov}(\epsilon_0,\epsilon_h)|,\\
\sum_{h=-\infty}^\infty |\mbox{Cov}(Y_0,Y_hY_{h+k_0-1}) - \mbox{Cov}(\epsilon_0,\epsilon_h \epsilon_{h+k_0-1})|,\\
\sum_{h=-\infty}^\infty |\mbox{Cov}(Y_0,Y_hY_{h+k_0}) - \mbox{Cov}(\epsilon_0,\epsilon_h \epsilon_{h+k_0})|,\\
\sum_{h=-\infty}^\infty |\mbox{Cov}(Y_0Y_{k_0-1},Y_hY_{h+k_0-1}) - \mbox{Cov}(\epsilon_0 \epsilon_{k_0-1},\epsilon_h \epsilon_{h+k_0-1})|,\\
\sum_{h=-\infty}^\infty |\mbox{Cov}(Y_0Y_{k_0-1},Y_hY_{h+k_0}) - \mbox{Cov}(\epsilon_0 \epsilon_{k_0-1},\epsilon_h \epsilon_{h+k_0})|,\\
\sum_{h=-\infty}^\infty |\mbox{Cov}(Y_0 Y_{k_0},Y_hY_{h+k_0}) - \mbox{Cov}(\epsilon_0 \epsilon_{k_0},\epsilon_h \epsilon_{h+k_0})|
\end{array}
\end{equation}
tend to $0$ as $\lambda_0$ tends to $0$: we prove this only for $\sum_{h=-\infty}^\infty |\mbox{Cov}(Y_0 Y_{k_0},Y_hY_{h+k_0}) - \mbox{Cov}(\epsilon_0 \epsilon_{k_0},\epsilon_h \epsilon_{h+k_0})|$, the other cases being dealt with similarly. We first note that
\begin{multline}\label{A4_light_proof1}
\mbox{Cov}(Y_0 Y_{k_0},Y_hY_{h+k_0})=\lambda_0 [ \mbox{Cov}(Y_0 Y_{k_0},Y_hY_{h+k_0-1}) + \mbox{Cov}(Y_0 Y_{k_0},Y_{h-1}\epsilon_{h+k_0})\\
+ \mbox{Cov}(Y_0 Y_{k_0-1},\epsilon_{h}\epsilon_{h+k_0}) + \mbox{Cov}(Y_{-1} \epsilon_{k_0},\epsilon_{h}\epsilon_{h+k_0})] + \mbox{Cov}(\epsilon_{k_0} \epsilon_{k_0},\epsilon_{h}\epsilon_{h+k_0}),\quad h> k_0.
\end{multline}
Now, \eqref{ineg_utile} implies that $\left|\mbox{Cov}(Y_0Y_{k_0}, Y_hY_{h+k_0-1})\right| \le K \left[ \alpha_\epsilon\left( \left\lfloor \frac{h-k_0-1}{2}\right\rfloor\right)^{1-2/\beta} + \lambda_0^{\left\lfloor \frac{h-k_0-1}{2}\right\rfloor}\right]$ where $K$ can be verified to be uniformly bounded in $\lambda_0\in [\lambda_-,\lambda_+]$. This inequality along with the assumption \eqref{cond_mixing_coef} entails that $\left|\mbox{Cov}(Y_0Y_{k_0}, Y_hY_{h+k_0-1})\right| \le K \lambda_+^{\left\lfloor \frac{h-k_0-1}{2}\right\rfloor}$ for $h> k_0$, so that $\sum_{h=k_0+1}^\infty \left|\mbox{Cov}(Y_0Y_{k_0}, Y_hY_{h+k_0-1})\right|$ is finite and bounded in $\lambda_0\in [\lambda_-,\lambda_+]$. A similar argument and bound holds for $\sum_{h=k_0+1}^\infty \left|\mbox{Cov}(Y_0 Y_{k_0},Y_{h-1}\epsilon_{h+k_0})\right|$, $\sum_{h=k_0}^\infty \left|\mbox{Cov}(Y_0 Y_{k_0-1},\epsilon_{h}\epsilon_{h+k_0}) \right|$ and $\sum_{h=k_0+1}^\infty \left| \mbox{Cov}(Y_{-1} \epsilon_{k_0},\epsilon_{h}\epsilon_{h+k_0})\right|$. All in all, one thus deduce from \eqref{A4_light_proof1} that $$\sum_{h=k_0+1}^\infty |\mbox{Cov}(Y_0 Y_{k_0},Y_hY_{h+k_0}) - \mbox{Cov}(\epsilon_0 \epsilon_{k_0},\epsilon_h \epsilon_{h+k_0})|=\mathrm{O}(\lambda_0).$$
Besides, one can verify that $|\mbox{Cov}(Y_0 Y_{k_0},Y_hY_{h+k_0}) - \mbox{Cov}(\epsilon_0 \epsilon_{k_0},\epsilon_h \epsilon_{h+k_0})|=\mathrm{O}(\lambda_0)$ for $h=0,...,k_0$ so that one finally obtains that
$$
\sum_{h=0}^\infty |\mbox{Cov}(Y_0 Y_{k_0},Y_hY_{h+k_0}) - \mbox{Cov}(\epsilon_0 \epsilon_{k_0},\epsilon_h \epsilon_{h+k_0})|=\mathrm{O}(\lambda_0).
$$
A similar argument holds for $\sum_{h=-\infty}^0 |\mbox{Cov}(Y_0 Y_{k_0},Y_hY_{h+k_0}) - \mbox{Cov}(\epsilon_0 \epsilon_{k_0},\epsilon_h \epsilon_{h+k_0})|$, so that we eventually have that $\sum_{h=-\infty}^\infty |\mbox{Cov}(Y_0 Y_{k_0},Y_hY_{h+k_0}) - \mbox{Cov}(\epsilon_0 \epsilon_{k_0},\epsilon_h \epsilon_{h+k_0})|=\mathrm{O}(\lambda_0)$ as $\lambda_0$ tends to $0$. As explained earlier, all series in \eqref{series_that_tend_0} can be shown to be also $\mathrm{O}(\lambda_0)$, so that \eqref{limitt_spectral_density} is proved. This in turn implies that $\lim_{\lambda_0 \to 0}\sup_{x\in [0, 2\pi]} |\det f_\Y(\lambda_0,x) - \det f_{\underline{\epsilon}}(x)|=0$. Hence we have that, for $\lambda_0$ small enough:
$$\inf_{x\in [0, 2\pi]} |\det f_\Y(\lambda_0,x)| \ge \inf_{x\in [0, 2\pi]} |\det f_{\underline{\epsilon}} (x)| - \sup_{x\in [0, 2\pi]} |\det f_\Y(\lambda_0,x) - \det f_{\underline{\epsilon}}(x)| >0 ,$$
so that \eqref{A4_light} is proved. We now prove that $\mathbf{ (A4)}$ holds, i.e. that $\det\left(\Psi(z)\right)\neq 0$ not just on $z$ such that $|z|=1$ but on the broader set $\{z\in \nbC |\ |z| \le 1 \}$. This is the most delicate part of the proof, and the main argument is an analysis related to the measurability with respect to the immigration sequences. We first observe that \eqref{A4_light} implies that the expansion \eqref{Y_Wold} can be inverted and there exists a sequence of $\nbR^{3\times 3}$ matrices $(\phi_n)_{n\in \nbZ}$ such that
\begin{equation}\label{epsilon_2_sides}
\varepsilon_n=\sum_{j\in \nbZ}\phi_j \Y_{n-j},\quad n\in \nbZ ,
\end{equation}
and we are going to prove that $\phi_{j}=0$ for all $j\le -1$, so that the expansion \eqref{AR_infty} holds with $\phi_0=\mathbf{I_3}$ and $\Phi_j=-\phi_j$, $j\ge 1$, which in turn is equivalent to Condition $\mathbf{ (A4)}$. We first recall that, for all $n\in \nbZ$, $\Y_n$ depends on the random variables $\epsilon_j$ and $\{\xi_{j,i}\ i\in \nbN \}$, $j\le n+k_0$ (see \eqref{model} as well as the expression \eqref{def_estimators} for $\Y_n$). Next, we write, using the definition \eqref{Wold_expressions} for $\varepsilon_n$ and the expansion \eqref{epsilon_2_sides}:
\begin{equation}\label{equality_to0}
\Y_n - P_{n-1} \Y_n -  \sum_{j=0}^{\infty}\phi_j \Y_{n-j}=\sum_{j=-\infty}^{-1}\phi_j \Y_{n-j}.
\end{equation}
We now prove that the righthanside of \eqref{equality_to0} is $0$ by using arguments related to the indepence of the sequences $\{ \xi_{n,k},\ k\in\N \}$ when $n\in\N$. To avoid tedious notation we will denote by $\xi_{n,.}$ the latter sequence in the rest of the proof. We define the following $\sigma -$ algebras
\begin{equation}\label{def_Gn}
{\cal G}_n:= \sigma (Y_k, k\le n,\ \xi_{j,.}, j\ge n+1),\quad n\in \nbZ,
\end{equation}
and we proceed to express the conditional expectation of $\Y_{n-j}$ given ${\cal G}_n$ in function of $j\in \nbZ$ in order to take the conditional expectation on both sides of \eqref{equality_to0}. We observe that
\begin{equation}\label{cond_expectation_1}
\nbE(Y_{n-j}| \G_n )= \left\{
\begin{array}{cl}
Y_{n-j} & \mbox{if } j\ge 0,\\
G_{-j}^{(1)} (\xi_{k,.},k=n+1,...,n-j) & \mbox{if } j\le -1 .
\end{array}
\right.
\end{equation}
Here $G_{-j}^{(1)} (\xi_{k,.},k=n+1,...,n-j)$ is a random variable that depends on the sequences $\xi_{k,.}$,$ k=n+1,...,n-j$ as well as on the random variables $Y_k$, $k\le n$; however, only the dependence on $\xi_{k,.}$ is mentioned in order to avoid tedious notation. Similarly, we have
\begin{eqnarray}
\nbE(Y_{n-j} Y_{n-j+k_0-1}| \G_n ) &= & \left\{
\begin{array}{cl}
Y_{n-j} Y_{n-j+k_0-1} & \mbox{if } j\ge k_0-1,\\
G_{-j}^{(2)} (\xi_{k,.},k=n+1,...,n-j+k_0-1) & \mbox{if } j\le -1,\\
Y_{n-j} \tilde G_{-j}^{(2)} (\xi_{k,.},k=n+1,...,n-j+k_0-1) & \mbox{if } j=0,...,k_0-2 ,
\end{array}
\right. \label{cond_expectation_2}\\
\nbE(Y_{n-j} Y_{n-j+k_0}| \G_n ) &= & \left\{
\begin{array}{cl}
Y_{n-j} Y_{n-j+k_0} & \mbox{if } j\ge k_0,\\
G_{-j}^{(3)} (\xi_{k,.},k=n+1,...,n-j+k_0) & \mbox{if } j\le -1,\\
Y_{n-j} \tilde G_{-j}^{(3)} (\xi_{k,.},k=n+1,...,n-j+k_0) & \mbox{if } j=0,...,k_0-1 ,
\end{array}
\right. \label{cond_expectation_3}
\end{eqnarray}
where $\tilde G_{-j}^{(2)} (\xi_{k,.},k=n+1,...,n-j+k_0-1)= \nbE(Y_{n-j+k_0-1}| \G_n )$ and $ \tilde G_{-j}^{(3)} (\xi_{k,.},k=n+1,...,n-j+k_0)= \nbE(Y_{n-j+k_0}| \G_n )$. We now note that, by definition of the projection operator $P_{n-1}$, there exists a sequence of combinations $\left(\sum_{j=1}^{N_m} \alpha_j^{(m)} \Y_{n-j} \right)_{m\in \nbN}$ such that, $\alpha_j^{(m)}$ is a $\nbR^{3\times 3}$ matrix and $\sum_{j=1}^{N_m} \alpha_j^{(m)} \Y_{n-j}\stackrel{\mathbb{L}^2}{\longrightarrow} P_{n-1}  \Y_n $ as $m\to \infty$. Thus, the lefthanside of \eqref{equality_to0} is the limit in $\mathbb{L}^2$ of an (potentially infinite) combination $\sum_{j=1}^{\infty} \phi_j^{(m)} \Y_{n-j}$ as $m\to\infty$, where $(\phi_j^{(m)})_{j\in \nbN} $ are $\nbR^{3\times 3}$ matrices given by $\phi_0^{(m)}= I_3- \alpha_0^{(m)}$, $\phi_j^{(m)}=-\alpha_j^{(m)} - \phi_j$, $j=1,...,N_m$, and $\phi_j^{(m)}=- \phi_j$, $j> N_m$. Then, taking the conditional expectation in \eqref{equality_to0} reads
\begin{equation}\label{equality_to01}
\lim_{m\to \infty }\sum_{j=0}^{\infty} \phi_j^{(m)} \nbE (\Y_{n-j}| {\cal G}_n)= \sum_{j=-\infty}^{-1}\phi_j \nbE (\Y_{n-j}| {\cal G}_n).
\end{equation}
Let us now write the matrices $\phi_j$, $j\le -1$, and $\phi_j^{(m)}$, $j\ge 0$, in the form of  $\phi_j=[C_j^{(1)}, C_j^{(2)}, C_j^{(3)}]$ and $\phi_j^{(m)}=[C_j^{(1)(m)}, C_j^{(2)(m)}, C_j^{(3)(m)}]$ where $C_j^{(i)}$, $C_j^{(i)(m)}$, $i=1,2,3$, are $3\times 1$ column vectors. We next identify in both sides of \eqref{equality_to01} the terms that feature the sequence $\xi_{n+1,.}$, and only this sequence, among $\xi_{k,.}$, $k\ge n+1$. From the expressions \eqref{cond_expectation_1}, \eqref{cond_expectation_2} and \eqref{cond_expectation_3}, we may see that the only term on the righthandside of \eqref{equality_to01} with that property is $C_{-1}^{(1)} G_{-1}^{(1)} (\xi_{n+1,.})=C_{-1}^{(1)}\nbE(Y_{n+1}| \G_n )=C_{-1}^{(1)}\left[ \sum_{r=1}^{Y_n} \xi_{n+1,r} + \nbE(\epsilon_{n+1} | \G_n  )\right]$, with $\nbE(\epsilon_{n+1} | \G_n  )$ depending only on $Y_j$, $j\le n$ because of the independence of $\epsilon_{n+1}$ with $\xi_{j,.}$, $j\ge n+1$. On the lefthandside, and in view of the expressions \eqref{cond_expectation_2} and \eqref{cond_expectation_3}, the only case for $\nbE(Y_{n-j} Y_{n-j+k_0-1}| \G_n )$ and $\nbE(Y_{n-j} Y_{n-j+k_0}| \G_n )$ to depend on $\xi_{n+1,.}$ is when $j$ is respectively equal to $k_0-2$ and $k_0-1$: in that case $\nbE(Y_{n-(k_0-2)} Y_{n+1}| \G_n )= Y_{n-(k_0-2)} \tilde G_{-(k_0-2)}^{(2)} (\xi_{n+1,.})= Y_{n-(k_0-2)} \nbE(Y_{n+1}| \G_n )=Y_{n-(k_0-2)}\left[ \sum_{r=1}^{Y_n} \xi_{n+1,r} + \nbE(\epsilon_{n+1} | \G_n  )\right]$ and $\nbE(Y_{n-(k_0-1)} Y_{n+1}| \G_n )=Y_{n-(k_0-1)} \tilde G_{-(k_0-2)}^{(3)} (\xi_{n+1,.})=Y_{n-(k_0-1)}\left[ \sum_{r=1}^{Y_n} \xi_{n+1,r} + \nbE(\epsilon_{n+1} | \G_n  )\right]$. Isolating these particular terms, one thus checks that \eqref{equality_to01} can be equivalently written as
\begin{multline}\label{equality_to02}
{\cal E}_m:= - C_{-1}^{(1)} G_{-1}^{(1)} (\xi_{n+1,.}) + C_{-(k_0-2)}^{(2)(m)} Y_{n-(k_0-2)} \tilde G_{-(k_0-2)}^{(2)} (\xi_{n+1,.}) \\
+ C_{-(k_0-1)}^{(3)(m)} Y_{n-(k_0-1)} \tilde G_{-(k_0-1)}^{(3)} (\xi_{n+1,.}) + X^{(m)}\\
=\left[ - C_{-1}^{(1)} + C_{-(k_0-2)}^{(2)(m)} Y_{n-(k_0-2)}  + C_{-(k_0-1)}^{(3)(m)} Y_{n-(k_0-1)} 
 \right]\left[ \sum_{r=1}^{Y_n} \xi_{n+1,r} + \nbE(\epsilon_{n+1} | \G_n  )\right]\\
 +X^{(m)} \stackrel{\mathbb{L}^2}{\longrightarrow} 0,\quad m\to \infty ,
\end{multline}
where $X^{(m)}$ is a random vector which is the sum of 
\begin{itemize}
\item a random vector that involves the terms $\nbE (Y_{n-j}| {\cal G}_n)$, $j\ge 0$, $\nbE (Y_{n-j}Y_{n-j+k_0-1}| {\cal G}_n)$, $j\ge k_0-1$, $\nbE (Y_{n-j}Y_{n-j+k_0}| {\cal G}_n)$, $j\ge k_0$, that are respectively equal to $Y_{n-j}$, $Y_{n-j}Y_{n-j+k_0-1}$ and $Y_{n-j}Y_{n-j+k_0}$, and that thus depend on $Y_k$, $k\le n$  only,
\item  and of a random vector that figures terms of the form $\nbE(Y_{n+j}| \G_n )$, $\nbE( Y_{n+j-(k_0-1)} Y_{n+j}| \G_n )$, $\nbE( Y_{n+j-(k_0-2)} Y_{n+j}| \G_n )$, $j\ge 2$, that thus depend on $Y_k$, $k\le n$ and $\xi_{k,.}$, $k\ge n+1$.
\end{itemize}
Let us now suppose that
\begin{equation}\label{assumption_to0}
\limsup_{m\to \infty }\|  - C_{-1}^{(1)} + C_{-(k_0-2)}^{(2)(m)} Y_{n-(k_0-2)}  + C_{-(k_0-1)}^{(3)(m)} Y_{n-(k_0-1)}\| > 0 
\end{equation}
on a non negligible set. Then, from \eqref{equality_to02}, the $3\times 1$ vectors $D_m:= - C_{-1}^{(1)} + C_{-(k_0-2)}^{(2)(m)} Y_{n-(k_0-2)}  + C_{-(k_0-1)}^{(3)(m)} Y_{n-(k_0-1)}$, ${\cal E}_m$ and $X^{(m)}$ are such that there exists $i_0\in \{1,2,3\}$ and $\limsup_{m\to \infty} |D_{m,i_0}|>0$, where $D_{m,i_0}$ is the $i_0$th entries . With the same notation for ${\cal E}_{m,i_0}$ and $X^{(m)}_{i_0}$, we obtain by taking the limsup in \eqref{equality_to02} (observing that $\sum_{r=1}^{Y_n} \xi_{n+1,r} + \nbE(\epsilon_{n+1} | \G_n  )$ is non negative) that
\begin{equation}\label{assumption_to1}
\left[ \limsup_{m\to \infty} |D_{m,i_0} |\right] .\left[\sum_{r=1}^{Y_n} \xi_{n+1,r} + \nbE(\epsilon_{n+1} | \G_n  ) \right] = \limsup_{m\to \infty} \left|{\cal E}_{m,i_0}- X^{(m)}_{i_0} \right| .
\end{equation}
Up to a subsequence, we have from \eqref{equality_to02} that ${\cal E}_m\longrightarrow 0$ a.s. as $m\to \infty$; Thus we have that ${\cal E}_{m,i_0}\longrightarrow 0$.
So that,  on a non negligible set, \eqref{assumption_to1} yields
\begin{equation}\label{assumption_to2}
\sum_{r=1}^{Y_n} \xi_{n+1,r} + \nbE(\epsilon_{n+1} | \G_n  ) = \frac{\limsup_{m\to \infty } \left|X^{(m)}_{i_0} \right|}{\limsup_{m\to \infty } \left|D_{m,i_0}\right|}.
\end{equation}
Let us now define the operator ${\cal F}(\xi_{j,.},\epsilon_j): y\in \nbN\mapsto \sum_{r=1}^y \xi_{j,r} + \epsilon_j=\theta_j \circ y+ \epsilon_j$ for all $j\in \nbZ$. We may thus write that, for all $k\ge 2$,
\begin{equation}\label{Y_with_operator}
Y_{n+k}= {\cal F}(\xi_{n+k,.},\epsilon_{n+k})\circ \ldots \circ {\cal F}(\xi_{n+1,.},\epsilon_{n+1}) (Y_n),
\end{equation}
so that, obviously, $Y_{n+k}$, $Y_{n+k-(k_0-2)}$, $Y_{n+k-(k_0-1)}$, and thus also their conditional expectation given $\G_n$, depend on the sequences $\xi_{n+1,.}$, $\xi_{n+2,.}$,..., $\xi_{n+k,.}$ for all $k\ge 2$. More importantly  a crucial property is that they depend $\xi_{n+1,.}$ if and only if they depend on say $\xi_{n+2,.}$, in view of the expression of the operators ${\cal F}(\xi_{j,.},\epsilon_j)$ and the expression \eqref{Y_with_operator}. Thus, the righthandside of \eqref{assumption_to2} in particular depends on $\xi_{n+1,.}$ if and only if it depends on $\xi_{n+2,.}$, which is a contradiction with the expression on the lefthandside which only depends on $\xi_{n+1,.}$ (in addition to the $Y_k$, $k\le n$, which are independent from the the $\xi_{k,.}$, $k\ge n+1$). Thus, \eqref{assumption_to0} is violated and we have
\begin{equation}\label{assumption_to3}
\lim_{m\to \infty }  - C_{-1}^{(1)} + C_{-(k_0-2)}^{(2)(m)} Y_{n-(k_0-2)}  + C_{-(k_0-1)}^{(3)(m)} Y_{n-(k_0-1)}=0 \quad \mbox{a.s.}
\end{equation}
Let us now notice that the sequences $(C_{-(k_0-2)}^{(2)(m)})_{m\in \nbN}$ and $(C_{-(k_0-1)}^{(3)(m)})_{m\in \nbN}$ are then necessarily bounded. Indeed, if say $(C_{-(k_0-2)}^{(2)(m)})_{m\in \nbN}$ was not bounded, then we may suppose that $\|C_{-(k_0-2)}^{(2)(m)} \|\longrightarrow \infty$ up to a subsequence. Then \eqref{assumption_to3} implies that $\|C_{-(k_0-1)}^{(3)(m)} \|\longrightarrow \infty$, and, dividing on both sides of \eqref{assumption_to3} by $\|C_{-(k_0-1)}^{(3)(m)}\|$ and letting $m\to \infty$ yields that the (deterministic) limit $L:= \lim_{m\to \infty} C_{-(k_0-2)}^{(2)(m)}/ \|C_{-(k_0-1)}^{(3)(m)}\|$ exists and that $ Y_{n-(k_0-1)}=L  Y_{n-(k_0-2)} $, which is not possible by an independence argument. Hence the boundedness of the two sequences, from which one gets that $(C_{-(k_0-2)}^{(2)(m)})_{m\in \nbN}$ and $(C_{-(k_0-1)}^{(3)(m)})_{m\in \nbN}$ converge (up to a subsequence) towards some limits $L_2$ and $L_3$, which in turn yields from \eqref{assumption_to3} that $ - C_{-1}^{(1)} + L_2 Y_{n-(k_0-2)}  + L_3 Y_{n-(k_0-1)}=0$ a.s. An indenpendence argument again yields that $L_2$, $L_3$ and, more importantly, that $C_{-1}^{(1)}$ are equal to zero. Similar arguments yield that $C_{-1}^{(2)}=C_{-1}^{(3)}=0$ i.e. $\phi_{-1}=0$. Similarly, one proves that $\phi_{-j}=0$ for all $j\ge 2$.
\\

We now prove that $\|\Phi_k\| =\mathrm{o}(1/k^2)$ as $k\to\infty$. We define, as in \eqref{def_spectral_density}, the function
$$g_\Y(z):=\frac{1}{2\pi}\sum_{h=-\infty}^\infty C(h)z^{h},\quad C(h):=\mbox{Cov}(\Y_0,\Y_h),$$
which is convergent when $\lambda_0^{1/2}<|z|<1/\lambda_0^{1/2}$ thanks to arguments similar to the ones leading to the convergence of \eqref{convergence_series_psi}. As in \eqref{relation_spectral_psi}, one has the relation
\begin{equation}\label{relation_gY_psi}
g_\Y(z)= \Psi(1/z) \Sigma_\varepsilon \Psi(z) ' ,\quad \lambda_0^{1/2}<|z|<1/\lambda_0^{1/2}.
\end{equation} 
Note now that the assumption that $f_{\underline{\epsilon}}$ is invertible on $[0,2\pi]$ implies that there exists $\delta>0$ small enough such that ${\cal C}_\delta:=\{z\in \nbC |\ 1-\delta \le |z| \le 1+\delta\}\subset \left\{z\in \nbC |\ \lambda_0^{1/2}<|z|<1/\lambda_0^{1/2}\right\}$ and the function
$$
g_{\underline{\epsilon}}: z\mapsto \frac{1}{2\pi} \sum_{h=-\infty}^\infty \mbox{Cov}(\underline{\epsilon}_0,\underline{\epsilon}_h)z^h
$$
is invertible on ${\cal C}_\delta$. As in \eqref{limitt_spectral_density}, one may specify the dependence of $g_\Y$ w.r.t. $\lambda_0$ and prove that $\sup_{z\in {\cal C}_\delta} \|g_\Y(\lambda_0,z) - g_{\underline{\epsilon}}(z)\|\longrightarrow 0$ as $\lambda_0 \to 0$ so that, if $\lambda_0$ is small enough, $g_\Y$ is invertible on ${\cal C}_\delta$. Thus, for such a $\lambda_0$, one may invert \eqref{relation_gY_psi} and deduce in particular that $\Phi(z)$ in \eqref{def_Phi} is a convergent series on ${\cal C}_\delta$. This in particular implies that $\sum_{k=1}^\infty \| \Phi_k \| (1+\delta)^k <\infty$, which in turn implies $\|\Phi_k\| =\mathrm{o}(1/k^2)$.
\zak

{\bf Proof of Theorem \ref{convergence_Isp}.} As announced after the statement of the theorem, the scheme of the proof is quite standard and is similar to \citet[Section 3.3.1]{BMCF12}. Therefore, the following series of Lemmas will be proved only when their proofs are significantly different from the corresponding ones in the latter paper. We start by introducing some notation. If $r\ge 1$, we let $\Phi_r(z)=I_{3}-\sum_{k=1}^r\Phi_{r,k}z^k$, where $\Phi_{r,1},\dots,\Phi_{r,r}$ denote the coefficients of the least squares regression of ${\cal Y}_n$ on ${\cal Y}_{n-1},\dots,{\cal Y}_{n-r}$ as defined in \eqref{AR_tronquee}, and we recall that $\varepsilon_{r,n}$ is the corresponding residual. We let $\Y_{r,t}:= (\Y_{t-1}',\ldots, \Y_{t-r}')'\in \nbR^{3r}$, and introduce the following (real or empirical) covariance matrices
\begin{eqnarray}
&& \begin{array}{rcl}
\Sigma_{\Y} &:= & \nbE(\Y_{0}\Y_{0}')\in \nbR^{3\times 3},\\
\Sigma_{\Y_r}&:=& \nbE(\Y_{r,0}\Y_{r,0}')\in \nbR^{3r\times 3r},\\
\Sigma_{\Y,\Y_r}&:=& \nbE(\Y_{0}\Y_{r,0}')\in \nbR^{3\times 3r},
\end{array}\label{SigmaYr}\\
&& \begin{array}{rcl}
\hat\Sigma_{\Y}&:=& \frac{1}{n}\sum_{t=1}^n\Y_{t}\Y_{t}'\in \nbR^{3\times 3},\\
\hat\Sigma_{\Y_r}&:=& \frac{1}{n}\sum_{t=1}^n\Y_{r,t}\Y_{r,t}'\in \nbR^{3r\times 3r},\\
\hat\Sigma_{\Y,\Y_r}&:=& \frac{1}{n}\sum_{t=1}^n \Y_{t}\Y_{r,t}'\in \nbR^{3\times 3r},
\end{array}\label{hatSigmaYr}\\
&& \begin{array}{rcl}
\hat\Sigma_{\hat\Y}&:=& \frac{1}{n}\sum_{t=1}^n\hat\Y_{t}\hat\Y_{t}'\in \nbR^{3\times 3},\\
\hat\Sigma_{\hat\Y_r}&:=& \frac{1}{n}\sum_{t=1}^n\hat\Y_{r,t}\hat\Y_{r,t}'\in \nbR^{3r\times 3r},\\
\hat\Sigma_{\hat\Y,\hat\Y_r}&:=& \frac{1}{n}\sum_{t=1}^n \hat\Y_{t}\hat\Y_{r,t}'\in \nbR^{3\times 3r},
\end{array}\label{hatSigmahatYr}
\end{eqnarray}
Then we have the matrix expressions for the least square projections:
\begin{equation}\label{Phi_barre}
\underline{\Phi_r} := (\Phi_{r,1},\ldots,\Phi_{r,r})'=\Sigma_{\Y,\Y_r}\Sigma_{\Y_r}^{-1},\quad \underline{\hat\Phi_r} := (\hat\Phi_{r,1},\ldots,\hat\Phi_{r,r})'=\hat\Sigma_{\hat\Y,\hat\Y_r}\hat\Sigma_{\hat\Y_r}^{-1}.
\end{equation}
We use here the matrix norm $\| A\|:= \sup_{\|x\| \ \le 1}\| Ax\|$ for all matrix $A=(a_{ij})_{i=1,...,d_1,j=1,...,d_2}\in \nbR^{d_1\times d_2}$, associated to the usual euclidean norm, and recall that it is submultiplicative and verifies
\begin{equation}\label{subm_norm}
\| A\|^2 \le \sum_{i,j} a_{ij}^2.
\end{equation}
We now state the following lemmas that lead to the proof of the theorem.
\begin{lemm}\normalfont\label{lemmaA6}
One has that $\sup_{r\ge 1} \| \Sigma_{\Y_r}\|$, $\sup_{r\ge 1} \| \Sigma_{\Y,\Y_r}\|$ and $\sup_{r\ge 1} \| \Sigma_{\Y_r}^{-1}\|$ are finite.
\end{lemm}
\begin{proof}
We follow the proof of \citet[Lemma 1]{BMCF12}. We first observe that for all column vector $x\in \nbR^{3r}$ we have the inequalities $\left\| \Sigma_{\Y_{r+1}}x \right\|\le  \left\| \Sigma_{\Y_{r+1}} (x', 0_3') \right\|$ and $\left\| \Sigma_{\Y,\Y_r} (0_3',x') \right\|$, where $0_3:=(0,0,0)$, so that $(\| \Sigma_{\Y_r}\|)_{r\ge 1}$ is an increasing sequence and $\| \Sigma_{\Y,\Y_r}\|\le \| \Sigma_{\Y_{r+1}}\|$ for all $r\ge 1$. Thus it suffices to prove that $\sup_{r\ge 1} \| \Sigma_{\Y_r}\|<\infty$ in order for $\sup_{r\ge 1} \| \Sigma_{\Y,\Y_r}\|$ to be finite. \\
For this, we recall that the spectral density $ f_\Y$ associated to the process $ \{\Y_n,\ n\in \nbZ \}$ defined by \eqref{def_spectral_density}
is absolutely convergent for all $x\in \nbR$, thus verifies 
\begin{equation}\label{proof_lemmaA3_0}
\|f_\Y(x)\|\le K_{max},\quad x\in \nbR,
\end{equation}
for some constant $K_{max}>0$. Also, $\Sigma_{\Y_r}$ is symmetric matrix with real entries, so that the matrix norm entails that $\| \Sigma_{\Y_r}\|$ is its  largest (real) eigenvalue, that we denote by $\rho_{max}(r)$. Hence, there exists an eigenvector $\delta^{(r)}=(\delta^{(r)'}_1,...,\delta^{(r)'}_r)'\in \nbR^{3r}$, where $\delta_i\in \nbR^3$, such that
\begin{equation}\label{proof_lemmaA3_1}
\| \Sigma_{\Y_r}\|= \rho_{max}(r)=\delta^{(r)'} \Sigma_{\Y_r}\delta^{(r)}\quad\mbox{and}\quad \|\delta^{(r)}\|=1.
\end{equation}
Besides, one obtains from \eqref{def_spectral_density} and the inversion formula that $C(h)=\int_{-\pi}^\pi f_{\Y}(x) e^{-ihx}dx$ for all $h\in\nbZ$, so that we obtain from \eqref{proof_lemmaA3_1} that
\begin{multline}
\label{proof_lemmaA3_2}
\| \Sigma_{\!Y_r}\|=\sum_{j,s=1}^r \delta^{(r)'}_j\left[ \int_{-\pi}^\pi f_\Y(x) e^{i(s-j)x}dx\right]   \delta^{(r)}_s\\
= \int_{-\pi}^\pi \left( \sum_{s=1}^r  \delta^{(r)}_s e^{i(s-1)x} \right)'  f_\Y(x) \overline{\left( \sum_{s=1}^r  \delta^{(r)}_s e^{i(s-1)x} \right)} dx.
\end{multline}
Since $f_\Y(x)$ is an hermitian matrix for all $x\in [0,2\pi]$, one has $X' f_\Y(x) \bar{X}\le \|f_\Y(x)\| X'\bar{X}$ for all column vector $X\in \nbC^{3r}$. Coupled with \eqref{proof_lemmaA3_0}, \eqref{proof_lemmaA3_2} as well as some easy computation, this yields the bound
$$
\| \Sigma_{\Y_r}\|\le K_{max} \int_{-\pi}^\pi \left( \sum_{s=1}^r  \delta^{(r)}_s e^{i(s-1)x} \right)' \overline{\left( \sum_{s=1}^r  \delta^{(r)}_s e^{i(s-1)x} \right)} dx= K_{max} \|\delta^{(r)}\|^2=K_{max},
$$
a constant which is independent from $r\ge 1$, which proves $\sup_{r\ge 1} \| \Sigma_{\Y_r}\|<\infty$.\\
We now turn to $\sup_{r\ge 1} \| \Sigma_{\Y_r}^{-1}\|$. We start from the observation that $\| \Sigma_{\Y_r}^{-1}\|$ is the inverse of the smallest eigenvalue $\rho_{min}(r)$ (which is positive) of $\Sigma_{\Y_r}$. As for $\rho_{max}(r)$, there thus exists an eigenvector $\iota^{(r)}=(\iota^{(r)'}_1,...,\iota^{(r)'}_r)'\in \nbR^{3r}$, where $\iota_i\in \nbR^3$, such that
\begin{equation}\label{proof_lemmaA3_3}
\| \Sigma_{\Y_r}^{-1}\|^{-1}= \rho_{min}(r)=\iota^{(r)'} \Sigma_{\Y_r}\iota^{(r)}\quad\mbox{and}\quad \|\iota^{(r)}\|=1.
\end{equation}
Now, $\mathbf{ (A4)}$ and \eqref{relation_spectral_psi} imply that $f_\Y(x)$ is invertible for all $x\in [-\pi,\pi]$, so that $X' f_\Y(x) \bar{X} >0$ for all $X\in S(\nbC^{3r}):=\{ V\in \nbC^{3r}|\ \|V\|^2=V' \bar{V}=1\}$. Since $[-\pi,\pi]\times S(\nbC^{3r})$ is a compact set of $\nbR\times \nbC^{3r}$, we deduce that there exists some constant $K_{min}>0$ such that
$$
X' f_\Y(x) \bar{X} \ge K_{min},\quad \forall (x,X)\in [-\pi,\pi]\times S(\nbC^{3r}),
$$
from which an analysis similar to \eqref{proof_lemmaA3_2} yields from \eqref{proof_lemmaA3_3} that
\begin{multline*}
\| \Sigma_{\Y_r}^{-1}\|^{-1}=\int_{-\pi}^\pi \left( \sum_{s=1}^r  \iota^{(r)}_s e^{i(s-1)x} \right)'  f_\Y(x) \overline{\left( \sum_{s=1}^r  \iota^{(r)}_s e^{i(s-1)x} \right)} dx\\
\ge K_{min} \int_{-\pi}^\pi \left( \sum_{s=1}^r  \iota^{(r)}_s e^{i(s-1)x} \right)' \overline{\left( \sum_{s=1}^r  \iota^{(r)}_s e^{i(s-1)x} \right)} dx= K_{min} \|\iota^{(r)}\|^2=K_{min},
\end{multline*}
so that $\| \Sigma_{\Y_r}^{-1}\|\le K_{min}^{-1}$, a constant independent from $r\ge 1$, proving that $\sup_{r\ge 1} \| \Sigma_{\Y_r}^{-1}\|<\infty$.
\end{proof}
\begin{lemm}\normalfont\label{lemmaA8}
For all $s\in \nbN$, the series $\sum_{h=-\infty}^\infty \|\mbox{Cov} (\Y_0 \Y_s',\Y_h \Y_{h+s}')\|$ is convergent.
\end{lemm}
\begin{proof}
Let us put $w_h(s):=\mbox{Cov} (Y_0 Y_{k_0}Y_s Y_{s+k_0}, Y_h Y_{h+k_0}Y_{s+h} Y_{h+s+k_0})$, which is one of the entries of $\mbox{Cov} (\Y_0 \Y_s',\Y_h \Y_{h+s}')$. We proceed by establishing estimates similar to the ones in Lemmas \ref{lemma_cov} and \ref{lemma_cov2}. As in \eqref{decompo1}, we have
\begin{eqnarray*}
w_h(s)&=& \sum_{s'=0}^{k_0-1}\lambda_0^{s'} \left[ \sum_{t=0}^{s-k_0-1} \lambda_0^t \mbox{Cov}(  Y_0 Y_{k_0}Y_s Y_{s+k_0} ,  Y_h Y_{h+k_0} \epsilon_{s+h} \epsilon_{s+h+k_0-s'}) \right.\\
&& \left.  + \lambda_0^{s-k_0}
\mbox{Cov}(  Y_0 Y_{k_0}Y_s Y_{s+k_0} , Y_h Y_{h+k_0}^2  \epsilon_{s+h+k_0-s'} ) \right]  \\
&&       + \lambda_0^{k_0} \mbox{Cov}(  Y_0 Y_{k_0}Y_s Y_{s+k_0} ,  Y_h Y_{h+k_0}   Y_{s+h}^2  ),
\end{eqnarray*}
Similarly to the bound \eqref{bound_sum_Davydov1}, and since $\epsilon$ has moments of order $8$ thanks to Assumption $\mathbf{ (A2)_1}$, both quantities $\sum_{t=0}^{s-k_0-1}|\lambda_0^t \mbox{Cov}(  Y_0 Y_{k_0}Y_s Y_{s+k_0} ,  Y_h Y_{h+k_0} \epsilon_{s+h} \epsilon_{s+h+k_0-s'})|$ and\\ $| \lambda_0^{s-k_0}
\mbox{Cov}(  Y_0 Y_{k_0}Y_s Y_{s+k_0} , Y_h Y_{h+k_0}^2  \epsilon_{s+h+k_0-s'} ) |$ are less than $$K\left[ \alpha_\epsilon\left( \left\lfloor \frac{h-s-k_0-1}{2}\right\rfloor\right)^{1-2/\beta} + \lambda_0^{\left\lfloor \frac{h-s-k_0-1}{2}\right\rfloor}\right]$$ for $h \ge s+k_0+1$ and some constant $K>0$. Also, estimations similar to \eqref{decompo2} and \eqref{decompo3} leading to \eqref{bound_sum_Davydov1} may be used to obtain that $| \mbox{Cov}(  Y_0 Y_{k_0}Y_s Y_{s+k_0} ,  Y_h Y_{h+k_0}   Y_{s+h}^2  )|$ has the same upper bound, so that $\sum_{h=0}^\infty |w_h(s) |<\infty$. We prove similarly that $\sum_{h=-\infty}^0 |w_h(s) |$ and, more generally, that $\sum_{h=-\infty}^\infty \|\mbox{Cov} (\Y_0 \Y_s',\Y_h \Y_{h+s}')\|$ is convergent.
\end{proof}
\begin{lemm}\normalfont\label{lemmaA9}
The r.v. $\sqrt{r}\| \hat\Sigma_{\Y_r} - \Sigma_{\Y_r} \|$, $\sqrt{r}\| \hat\Sigma_{\Y} - \Sigma_{\Y} \|$ and $\sqrt{r}\| \hat\Sigma_{\Y,\Y_r} - \Sigma_{\Y,\Y_r} \|$ converge to $0$ in probability as $n\to \infty$ when $r=r(n)=\mathrm{o}(n^{1/3})$.
\end{lemm}
\begin{proof}
Remember from \eqref{SigmaYr} and \eqref{hatSigmaYr} that $\hat\Sigma_{\Y_r} - \Sigma_{\Y_r}$ has $(3r)^2$ coefficients. Each entry of $\hat\Sigma_{\Y_r}$ is of the form
$$
\hat Z_{r_1,r_2}^{i,j}:=\frac{1}{n}\sum_{t=1}^n \Y^i_{t-r_1}\Y^j_{t-r_2},\quad i,j=1,2,3,\ 1\le r_1,r_2 \le r ,
$$
of which variance verifies $\mbox{Var} (\hat Z_{r_1,r_2}^{i,j})=\frac{1}{n^2}\sum_{h=-n+1}^{n-1}(n-|h|)\mbox{Cov}(\Y^i_{-r_1}\Y^j_{-r_2}, \Y^i_{h-r_1}\Y^j_{h-r_2})$ which is less than $C_1/n$ for some constant $C_1$ thanks to Lemma \ref{lemmaA8}. We deduce that
\begin{multline*}
\nbE\left( \left[ \sqrt{r} \| \hat\Sigma_{\Y_r} - \Sigma_{\Y_r}\| \right]^2\right)=r\; \nbE\left( \| \hat\Sigma_{\Y_r} - \Sigma_{\Y_r}\|^2\right) \le r \sum_{i,j=1}^3 \sum_{r_1,r_2=1}^r \mbox{Var} (\hat Z_{r_1,r_2}^{i,j})\\
\le \frac{3r^3}{n}C_1 \longrightarrow 0,\quad n\to \infty, \ r=r(n)=\mathrm{o}(n^{1/3}),
\end{multline*}
which proves that $\sqrt{r} \| \hat\Sigma_{\Y_r} - \Sigma_{\Y_r}\|$ converges in $\mathbb{L}^2$, hence in probability, to $0$. We prove the same convergence for $\sqrt{r}\| \hat\Sigma_{\Y} - \Sigma_{\Y} \|$ and $\sqrt{r}\| \hat\Sigma_{\Y,\Y_r} - \Sigma_{\Y,\Y_r} \|$ in an analogous manner.
\end{proof}
\begin{lemm}\normalfont\label{lemmaA10}
The r.v. $\sqrt{r}\| \hat\Sigma_{\hat\Y_r} - \Sigma_{\Y_r} \|$, $\sqrt{r}\| \hat\Sigma_{\hat\Y} - \Sigma_{\Y} \|$ and $\sqrt{r}\| \hat\Sigma_{\hat\Y,\hat\Y_r} - \Sigma_{\Y,\Y_r} \|$ converge to $0$ in probability as $n\to \infty$ when $r=r(n)=\mathrm{o}(n^{1/3})$.
\end{lemm}
\begin{proof}
We are going to prove that $\sqrt{r}\| \hat\Sigma_{\hat\Y_r} - \hat \Sigma_{\Y_r} \|\stackrel{\nbP}{\longrightarrow} 0$ as $n\to \infty$ when $r=r(n)=\mathrm{o}(n^{1/3})$ so as to conclude the convergence of $\sqrt{r}\| \hat\Sigma_{\hat\Y_r} - \Sigma_{\Y_r} \|$ thanks to Lemma \ref{lemmaA9} and the triangular inequality. We note that the entries of $\hat\Sigma_{\hat\Y_r} - \hat \Sigma_{\Y_r}$ are of the form
$$
\frac{1}{n}\sum_{t=1}^n \left[  \Y^i_{t-r_1} \Y^j_{t-r_2}  - \hat \Y^i_{t-r_1} \hat\Y^j_{t-r_2}\right],\quad i,j=1,2,3,\ 1\le r_1,r_2 \le r ,
$$
of which most complex term is the following:
\begin{eqnarray}
&&\frac{1}{n}\sum_{t=1}^n \left\{ (Y_{t-r_1}Y_{t+k_0-r_1} - u_{k_0-1})(Y_{t-r_2}Y_{t+k_0-r_2} - u_{k_0-1}) \right. \nonumber\\
&& \left. - (Y_{t-r_1}Y_{t+k_0-r_1} -  \bar Y_{k_0,t-r_1}  )(Y_{t-r_2}Y_{t+k_0-r_2} - \bar Y_{k_0,t-r_2}) \right\}\nonumber\\
&=& \frac{1}{n}\sum_{t=1}^n \left\{ u_{k_0-1}^2 - \bar Y_{k_0,t-r_1} \bar Y_{k_0,t-r_2}+ Y_{t-r_1}Y_{t+k_0-r_1}(\bar Y_{k_0,t-r_2} -  u_{k_0-1}  ) \right. \nonumber\\
&& \left. + Y_{t-r_2}Y_{t+k_0-r_2}(\bar Y_{k_0,t-r_1} -  u_{k_0-1}  )\right\}.\label{term_lemmaA10}
\end{eqnarray}
We study each terms on the righthandside of the above equality and start with
\begin{equation}\label{1term_lemmaA10}
\frac{1}{n}\sum_{t=1}^n Y_{t-r_1}Y_{t+k_0-r_1}(\bar Y_{k_0,t-r_2} -  u_{k_0-1}  ).
\end{equation}
We first observe that a consequence of one of the steps in the proof of Theorem \ref{CLT_vector_k_0} (see convergence \eqref{convergence_Herrndorf} in Step 1 of the proof of this theorem) is $\sqrt{n} \| \bar Y_{k_0,n} -  u_{k_0-1} \|_2$ is upper bounded by some constant $K$ for all $n\ge 1$, of which consequence is obviously that $\| \bar Y_{k_0,n} -  u_{k_0-1} \|_2$ is obviously upper bounded in $n\in \nbN$. Hence, using the classical inequality $\sqrt{a+b}\le \sqrt{a} + \sqrt{b}$ for all $a,b \ge 0$, we have, for all $r_2 \le r$,%
\begin{eqnarray*}
\sqrt{t}  \| \bar Y_{k_0,t-r_2} -  u_{k_0-1} \|_2 &\le & K\sqrt{r}, \quad \forall t=1,\ldots r_2, \\
\sqrt{t}  \| \bar Y_{k_0,t-r_2} -  u_{k_0-1} \|_2 & \le & \sqrt{t-r_2}  \| \bar Y_{k_0,t-r_2} -  u_{k_0-1} \|_2 + \sqrt{r_2}  \| \bar Y_{k_0,t-r_2} -  u_{k_0-1} \|_2 \\
&\le & K(1+\sqrt{r}),\quad \forall t \ge r_2,
\end{eqnarray*}
from which one has the global bound $\sqrt{t}  \| \bar Y_{k_0,t-r_2} -  u_{k_0-1} \|_2 \le K(1+\sqrt{r})$ for some constant $K$ and for all $t\ge 0$. Coupled with the fact that $\| Y_n\|_4$ is uniformly bounded by some constant $C$ as well as Cauchy Scwharz inequality, we thus deduce that \eqref{1term_lemmaA10} is bounded as follows
$$
\left|\left| \frac{1}{n}\sum_{t=1}^n Y_{t-r_1}Y_{t+k_0-r_1}(\bar Y_{k_0,t-r_2} -  u_{k_0-1}  ) \right|\right|_1 \le \frac{1}{n}\sum_{t=1}^n C^2 \frac{K(1+\sqrt{r})}{\sqrt{t}}\le K' \frac{1+\sqrt{r}}{ \sqrt{n}}
$$
for some contant $K'$, where we used that $\sum_{t=1}^n \frac{1}{\sqrt{t}}$ is equivalent to $\sqrt{n}$ up to a factor. We thus deduce that
\begin{equation}\label{1term_lemmaA10_ineg}
\sqrt{r}\; \left|\left| \frac{1}{n}\sum_{t=1}^n Y_{t-r_1}Y_{t+k_0-r_1}(\bar Y_{k_0,t-r_2} -  u_{k_0-1}  ) \right|\right|_1 \longrightarrow 0,\quad n\to\infty,\ r=r(n)=\mathrm{o}(n^{1/3}).
\end{equation}
A similar approach yields the following limit that concerns the second term in \eqref{term_lemmaA10}:
\begin{equation}\label{2term_lemmaA10_ineg}
\sqrt{r}\; \left|\left| \frac{1}{n}\sum_{t=1}^n Y_{t-r_2}Y_{t+k_0-r_2}(\bar Y_{k_0,t-r_1} -  u_{k_0-1}  ) \right|\right|_1 \longrightarrow 0,\quad n\to\infty,\ r=r(n)=\mathrm{o}(n^{1/3}).
\end{equation}
We now turn to the last first term in \eqref{term_lemmaA10} namely ${n}^{-1}\sum_{t=1}^n (u_{k_0-1}^2 - \bar Y_{k_0,t-r_1} \bar Y_{k_0,t-r_2})$. For this, we write $u_{k_0-1}^2 - \bar Y_{k_0,t-r_1} \bar Y_{k_0,t-r_2}=u_{k_0-1} (u_{k_0-1}- \bar Y_{k_0,t-r_1}) + \bar Y_{k_0,t-r_1} (u_{k_0-1}- \bar Y_{k_0,t-r_1})$ and conclude in a similar manner that
\begin{equation}\label{3term_lemmaA10_ineg}
\sqrt{r}\; \left|\left| \frac{1}{n}\sum_{t=1}^n (u_{k_0-1}^2 - \bar Y_{k_0,t-r_1} \bar Y_{k_0,t-r_2})\right|\right|_1 \longrightarrow 0,\quad n\to\infty,\ r=r(n)=\mathrm{o}(n^{1/3}).
\end{equation}
Gathering \eqref{1term_lemmaA10_ineg}, \eqref{2term_lemmaA10_ineg} and \eqref{3term_lemmaA10_ineg} we deduce that \eqref{term_lemmaA10} multiplied by $\sqrt{r}$ tends to $0$ in $\mathbb{L}^1$ hence in probability as $n\to\infty$ when $r=r(n)=\mathrm{o}(n^{1/3})$. As announced earlier, we prove similarly that all other entries of $\sqrt{r} [ \hat\Sigma_{\hat\Y_r} - \hat \Sigma_{\Y_r}]$ also converges in probability. This concludes the proof.
\end{proof}
\begin{lemm}\normalfont \label{lemmaA11}
Let us define $\Phi_r^*:=(\Phi_1,...,\Phi _r)'$. Then $\sqrt{r}\|\Phi_r^* - \underline{\Phi_r}\|\longrightarrow 0$ as $r\to\infty$, where we recall that $\underline{\Phi_r}$ is defined in \eqref{Phi_barre}.
\end{lemm}
\begin{proof}
We follow the proof of \citet[Lemma 5]{BMCF12}. We recall from \eqref{AR_infty} and the definition of the least square regression coefficients $\Phi_{r,1},\dots,\Phi_{r,r}$ with corresponding residual $\varepsilon_{n,r} $ that $
\Y_n=\sum_{k=1}^{r}\Phi_k\mathcal{Y}_{n-k}+\sum_{k=r+1}^{\infty}\Phi_k\mathcal{Y}_{n-k}+\varepsilon_{n}= \sum_{k=1}^{r}\Phi_{r,k}\mathcal{Y}_{n-k}+\varepsilon_{n,r}
$
which entails the matrix relation
\begin{equation}\label{lemmaA11_1}
(\underline{\Phi_r}-\Phi_r^*)\Y_{r,n}=\sum_{k=r+1}^{\infty}\Phi_k\mathcal{Y}_{n-k}+\varepsilon_{n}:=\varepsilon_{n,r}^*,
\end{equation}
from which we deduce that
\begin{equation}\label{lemmaA11_2}
\underline{\Phi_r}-\Phi_r^*= \Sigma_{\varepsilon_{r}^*}\Sigma_{\Y_r}^{-1},\quad \mbox{where }\Sigma_{\varepsilon_{r}^*}:=\EE(\varepsilon_{n,r}^*\Y_{r,n}').
\end{equation}
Let us note that $\varepsilon_{n}$ is orthogonal to $\Y_{r,n}$ i.e. $\EE(\varepsilon_{n}\Y_{r,n}')=0 $, as this vector depends on $\Y_t$ for $t=n-r,...,n-1$. So that we obtain from \eqref{lemmaA11_1} and \eqref{lemmaA11_2} that
\begin{equation}\label{lemmaA11_3}
\Sigma_{\varepsilon_{r}^*}=\sum_{k=r+1}^{\infty}\Phi_k\EE(\mathcal{Y}_{n-k}\Y_{r,n}').
\end{equation}
We proved in Step 1 of the proof of Theorem \eqref{CLT_vector_k_0} that $\mathfrak{S}_{k_0}$ in \eqref{expression_cov_series} is an absolute convergent series: a byproduct of this is that $\|\Y_0\|_2 $ is finite. Also, since $ \Y_{r,n}$ is a $3r\times 3$ matrix, the matrix norm property \eqref{subm_norm} implies that $\|\Y_{r,n}\|_2\le \sqrt{Kr}$ for some constant $K>0$. Thus, we deduce from the Cauchy-Schwarz inequality that $\|\EE(\mathcal{Y}_{n-k}\Y_{r,n}')\|\le \EE(\|\mathcal{Y}_{n-k}\Y_{r,n}'\|) \le \EE(|\mathcal{Y}_{n-k}\|\;\|\Y_{r,n}'\|)\le |\mathcal{Y}_{n-k}\|_2\;\|\Y_{r,n}'\|_2=\mathrm{O}(\sqrt{r})$. So that, by \eqref{lemmaA11_3}:
$$
\|\Sigma_{\varepsilon_{r}^*}\|\le \sum_{k=r+1}^{\infty}\|\Phi_k\|\;\|\EE(\mathcal{Y}_{n-k}\Y_{r,n}')\|=\mathrm{O}(\sqrt{r})\sum_{k=r+1}^{\infty}\|\Phi_k\|,
$$
hence we obtain from \eqref{lemmaA11_2} that $\sqrt{r}\|\Phi_r^* - \underline{\Phi_r}\|=\mathrm{O}(r)\sum_{k=r+1}^{\infty}\|\Phi_k\|\longrightarrow 0$ thanks to the assumption $\|\Phi_k\| =\mathrm{o}(1/k^2)$.
\end{proof}
\begin{lemm}\normalfont \label{lemmaA12}
The r.v. $\sqrt{r}\| \hat\Sigma_{\hat\Y_r}^{-1} - \Sigma_{\Y_r}^{-1} \|$ converge to $0$ in probability as $n\to \infty$ when $r=r(n)=\mathrm{o}(n^{1/3})$.
\end{lemm}
\begin{proof}
The proof is exactly the same as the one in \citet[Lemma 6]{BMCF12}, up to a change of notation.
\end{proof}
\begin{lemm}\normalfont \label{lemmaA13}
The r.v. $\sqrt{r}\| \underline{\hat\Phi_r} - \underline{\Phi_r}  \|$ converge to $0$ in probability as $n\to \infty$ when $r=r(n)=\mathrm{o}(n^{1/3})$.
\end{lemm}
\begin{proof}
We follow the proof of \citet[Lemma 7]{BMCF12}. Similarly to \eqref{lemmaA11_2}, and thanks to \eqref{Phi_barre}, one has that 
\begin{multline*}
\underline{\hat\Phi_r} - \underline{\Phi_r} =\hat\Sigma_{\hat\Y,\hat\Y_r}\hat\Sigma_{\hat\Y_r}^{-1}-\Sigma_{\Y,\Y_r}\Sigma_{\Y_r}^{-1}= (\hat\Sigma_{\hat\Y,\hat\Y_r}-\Sigma_{\Y,\Y_r})(\hat\Sigma_{\hat\Y_r}^{-1}-\Sigma_{\Y_r}^{-1}) \\
+(\hat\Sigma_{\hat\Y,\hat\Y_r}-\Sigma_{\Y,\Y_r})\Sigma_{\Y_r}^{-1}+\Sigma_{\Y,\Y_r}(\hat\Sigma_{\hat\Y_r}^{-1}- \Sigma_{\Y_r}^{-1}),
\end{multline*}
whence one obtains for all $a>0$ that
\begin{multline}\label{lemmaA13_1}
\PP(\sqrt{r}\| \underline{\hat\Phi_r} - \underline{\Phi_r}  \|>a)\le \PP\left(
\sqrt{r}\| \hat\Sigma_{\hat\Y,\hat\Y_r}-\Sigma_{\Y,\Y_r}\|\; \|(\hat\Sigma_{\hat\Y_r}^{-1}-\Sigma_{\Y_r}^{-1})\| \right.\\
\left.  +\sqrt{r} \| \hat\Sigma_{\hat\Y,\hat\Y_r}-\Sigma_{\Y,\Y_r}\|\;\| \Sigma_{\Y_r}^{-1}\| +\sqrt{r} \| \Sigma_{\Y,\Y_r}\|\; \|\hat\Sigma_{\hat\Y_r}^{-1}- \Sigma_{\Y_r}^{-1}\|>a 
\right)  \\
\le \PP\left( \sqrt{r}\| \hat\Sigma_{\hat\Y,\hat\Y_r}-\Sigma_{\Y,\Y_r}\|\; \|(\hat\Sigma_{\hat\Y_r}^{-1}-\Sigma_{\Y_r}^{-1})\|>a/3
\right)  
+ \PP\left(
\sqrt{r} \| \hat\Sigma_{\hat\Y,\hat\Y_r}-\Sigma_{\Y,\Y_r}\|>a/(3\; \sup_{j\ge 1}\| \Sigma_{\Y_j}^{-1}\|)
\right) \\
+ \PP\left(
\sqrt{r} \| \Sigma_{\Y,\Y_r}\| \|\hat\Sigma_{\hat\Y_r}^{-1}- \Sigma_{\Y_r}^{-1}\|> a/(3\;\sup_{j\ge 1} \| \Sigma_{\Y,\Y_j}\|) 
\right).
\end{multline}
Now, each of the terms in \eqref{lemmaA13_1} tends to $0$ thanks to Lemmas \ref{lemmaA10} and \ref{lemmaA6}, and the proof is complete.
\end{proof}
The end of proof of Theorem \ref{convergence_Isp} relies on the above Lemmas \ref{lemmaA6} to \ref{lemmaA12}, and is the same as the end of proof as \citet[Theorem A.1, page 7 of Supplementary material]{BMCF12}. We give it below for the sake of presentation. Let $\nbE_r:= \mathbf{I}_3 \otimes \mathbf{1}_r$ where we recall that $\mathbf{I}_3$ is the $3\times 3$ identity matrix, $\mathbf{1}_r=(1,...,1)'$ of size $r$. Then by the submultiplicativity of the matrix norm and Lemma \ref{lemmaA12}: $\| \sum_{i=1}^r \left( \hat\Phi_{r,i} - \Phi_{r,i}\right)\| = \| (\hat\Phi_{r} - \Phi_{r}) \nbE_r\| \le \sqrt{3r}\|\hat\Phi_{r} - \Phi_{r} \|\stackrel{\nbP}{\longrightarrow} 0$ as $n\to \infty$ when $r=r(n)=\mathrm{o}(n^{1/3})$. Similarly using Lemma \ref{lemmaA11}, we have that $\| \sum_{i=1}^r \left( \Phi_{r,i} - \Phi_{i}\right)\|\|\longrightarrow 0$. We then deduce by triangular inequality that
\begin{multline}\label{convergence_spectral_estimator1}
\| \hat{\Phi}_r(1) - \Phi(1)\| \le \left|\left| \sum_{i=1}^r \left( \hat\Phi_{r,i} - \Phi_{r,i}\right)\right|\right| + \left|\left| \sum_{i=1}^r \left( \Phi_{r,i} - \Phi_{i}\right)\right|\right|\\
+ \left|\left| \sum_{i=r+1}^\infty \Phi_i \right|\right| \stackrel{\nbP}{\longrightarrow} 0,\quad n\to \infty,\ r=r(n)=\mathrm{o}(n^{1/3}).
\end{multline}
We now pass on to the convergence of the estimator $\hat{\Sigma}_{\hat{\varepsilon}_r}$. \eqref{AR_tronquee} implies that $\hat{\Sigma}_{\hat{\varepsilon}_r}=\hat\Sigma_{\hat\Y} - \underline{\hat \Phi_{r}} \hat\Sigma_{\hat\Y,\hat\Y_r} '$. Also, the expression \eqref{Wold_expressions}, the orthogonality of $\varepsilon_0$ and $P_{-1} \Y_0$ and \eqref{AR_infty} imply that
\begin{multline*}
\Sigma_\varepsilon=\nbE (\varepsilon_0 \varepsilon_0')= \nbE (\varepsilon_0 \Y_0')=\nbE \left[ \left( \Y_0 -\sum_{k=1}^{\infty}\Phi_k\mathcal{Y}_{-k} \right)  \Y_0' \right]\\
=\Sigma_\Y- \sum_{k=1}^{\infty}\Phi_k \nbE(\mathcal{Y}_{-k} \Y_0')= \Sigma_\Y-  \Phi_{r}^* \Sigma_{\Y,\Y_r} ' - \sum_{k=r+1}^{\infty}\Phi_k \nbE(\mathcal{Y}_{-k}  \Y_0') .
\end{multline*}
We deduce that
\begin{eqnarray}
\| \hat{\Sigma}_{\hat{\varepsilon}_r} - {\Sigma}_{{\varepsilon}_r}\| &= & \Big|\Big|\hat\Sigma_{\hat\Y}  -  \Sigma_{\Y} - \Big( \underline{\hat \Phi_{r}} -  \Phi_{r}^* \Big)\hat\Sigma_{\hat\Y,\hat\Y_r} '\nonumber\\
&& - \Phi_{r}^*\Big( \hat\Sigma_{\hat\Y,\hat\Y_r} ' - \Sigma_{\Y,\Y_r} ' \Big) + \sum_{k=r+1}^{\infty}\Phi_k \nbE(\mathcal{Y}_{-k}  \Y_0')\Big|\Big| \nonumber\\
&\le & \|  \hat\Sigma_{\hat\Y}  -  \Sigma_{\Y}\| + \Big|\Big| \Big( \underline{\hat \Phi_{r}} -  \Phi_{r}^* \Big)  \Big( \hat\Sigma_{\hat\Y,\hat\Y_r} '  - \Sigma_{\Y,\Y_r} '\Big) \Big|\Big| +  \Big|\Big| \Big( \underline{\hat \Phi_{r}} -  \Phi_{r}^* \Big)  \Sigma_{\Y,\Y_r}' \Big|\Big| \nonumber\\
&& + \Big|\Big| \Phi_{r}^* \Big( \hat\Sigma_{\hat\Y,\hat\Y_r} '  - \Sigma_{\Y,\Y_r} '\Big) \Big|\Big|+
\Big|\Big| \sum_{k=r+1}^{\infty}\Phi_k \nbE(\mathcal{Y}_{-k}  \Y_0') \Big|\Big|. \label{convergence_spectral_estimator2}
\end{eqnarray}
Lemma \ref{lemmaA9} entails that the first term on the right hand side of \eqref{convergence_spectral_estimator2} converges in probability to $0$. The combination of Lemmas \ref{lemmaA11}, \ref{lemmaA13} and \ref{lemmaA10} yield that the second term converges to $0$, and likewise the fourth term tends to $0$ because the $\{\Phi_i,\ i\in \nbN \}$ is a bounded sequence thanks for example to the assumption $\|\Phi_k\| =\mathrm{o}(1/k^2)$. Lemma \ref{lemmaA6} also entails that the third term tends to $0$. Finally, the last term tends to $0$ since it is the remainder of a convergent series. All in all, we thus deduce from \eqref{convergence_spectral_estimator2} that $\hat{\Sigma}_{\hat{\varepsilon}_r}\stackrel{\nbP}{\longrightarrow}0$ as $n\to\infty$, $r=r(n)=\mathrm{o}(n^{1/3})$, which combined with \eqref{convergence_spectral_estimator1}, proves \eqref{convergence_spectral_estimator}.\zak

\subsection{Proof of Proposition \ref{prop_main_theo1}}\label{subsec:proof_gradient}
The following lemma will be needed.
\begin{lemm}\label{lemma_before_prop_main_theo1}\normalfont
There exist constants $M_\varpi$, $M_H^1$ and $M_H^2$ such that
\begin{eqnarray}
|\varpi_k'(|x|) \ln (|x|)| &\le & M_\varpi \frac{k}{\lambda_-^k},\label{lemma_prop_main1}\\
|H_k'(x)| &\le & M_H^1 \frac{1}{\lambda_-^k}, \label{lemma_prop_main2}\\
|H_k''(x)| &\le & M_H^2 \frac{1}{\lambda_-^k}, \label{lemma_prop_main2_bis}
\end{eqnarray}
for all $x\in \R$. Besides, for each $k\in \N$ there exists $M_1(k)$ such that
\begin{eqnarray}
||\nabla \psi_k(a,b) || \le M_1(k),\quad \forall (a,b)\in \R^2 , \label{lemma_prop_main3}\\
||D^2 \psi_k(a,b) || \le M_2(k),\quad \forall (a,b)\in \R^2 , \label{lemma_prop_main3_bis}
\end{eqnarray}
with $M_1(k)$ and $M_2(k)$ being $\mathrm{O}\left(\frac{1}{\lambda_-^k} \right)$.
\end{lemm}
\begin{proof}
The assumption \eqref{equiv_varpi} on the regularity at $0$ of the function $\varpi$ imply that $\varpi'(x)=\mathrm{O}(x)$ as $x\to ^+0$, hence the existence of some positive $\vartheta^1$ (that we may also suppose less than $1$ w.l.o.g.) and $C_\varpi^1$ such that $0 \le \varpi'(x) \le C_\varpi^1 x$ when $x\in [0,\vartheta^1 ]$. Furthermore, the regularity conditions implies the existence of some $C_\varpi >0$ such that $0 \le \varpi'(x) \le C_\varpi$, $x\in [0,1]$; we also recall that $\varpi'(x)=0$ if $x<0$ or $x>1$. Consequently, $\varpi_k$ defined in \eqref{def_chi_Km} verifies
\begin{equation}\label{proof_lemma_estim_gen1}
|\varpi_k'(|x|) \ln (|x|)| \le  |\ln (|x|)| C_\varpi^1 \frac{2}{K_m \lambda_-^k}x \quad \mbox{if }  \frac{2}{K_m \lambda_-^k}|x| \le \vartheta^1 \iff  |x| \le \frac{K_m \lambda_-^k}{2} \vartheta^1 .
\end{equation}
We now observe that the function $x\mapsto x \ln x $ is decreasing and negative on the interval $[0, e^{-1}]$, and in particular on the interval $\left[0, {K_m \lambda_-^k} \vartheta^1/ {2}\right]$ for $k$ large enough. We thus deduce from \eqref{proof_lemma_estim_gen1} that $|\varpi_k'(|x|) \ln (|x|)| \le C_\chi^1 \left|\ln \left({K_m \lambda_-^k} \vartheta^1/ {2}\right)\right|=\mathrm{O}(k)$. Since $\lambda_-<1$, an $\mathrm{O}(k)$ is also an $\mathrm{O}(k/\lambda_-^k)$ so that we obtain
\begin{equation}\label{proof_lemma_estim_gen2}
|\varpi_k'(|x|) \ln (|x|)| = \mathrm{O}(k/\lambda_-^k),\quad x\in \left[0, \frac{K_m \lambda_-^k}{2} \vartheta^1 \right].
\end{equation}
Now, since $x\mapsto \ln x$ is increasing and negative on $(0 ,1 ]$, and picking $k$ large enough such that $({K_m \lambda_-^k})/{2} <1$, we have
\begin{multline}\label{proof_lemma_estim_gen3}
|\varpi_k'(|x|) \ln (|x|)|\le  C_\varpi^1  |\ln (|x|)| \le  C_\varpi^1 \left|\ln \left(\frac{K_m \lambda_-^k}{2} \vartheta^1 \right) \right|\\
=\mathrm{O}(k/\lambda_-^k),\quad x\in \left[ \frac{K_m \lambda_-^k}{2} \vartheta^1, \frac{K_m \lambda_-^k}{2} \right].
\end{multline}
And, since $\varpi_k'(|x|) \ln (|x|)=0$ when $|x|>({K_m \lambda_-^k})/{2}$, \eqref{proof_lemma_estim_gen2} and \eqref{proof_lemma_estim_gen3} yield \eqref{lemma_prop_main1}.\\
We now compute
\begin{equation}\label{proof_lemma_estim_gen4}
H_k'(x)=\frac{1}{k}\frac{1}{|x|} \varpi_k(|x|) + \frac{1}{k} \varpi_k'(|x|)\ln |x|,\quad x\in \nbR ,
\end{equation}
observing incidentally that $H_k'(x)$ is continuous at $x=0$ thanks to the fact that $\varpi(x)=\mathrm{o}(x^2)$ and $\varpi'(x)=\mathrm{o}(x)$ when $x\to 0$. We first note that \eqref{lemma_prop_main1} implies that the second term on the righthandside of \eqref{proof_lemma_estim_gen4} verifies $\frac{1}{k} \varpi_k'(x)\ln |x|=\mathrm{O}(1/\lambda_-^k)$. Furthermore, $\varpi$ has the property that $\varpi(x)=\mathrm{O}(x)$ so that there exists $\vartheta^0>0$ and $C_\varpi^0$ such that $0 \le \varpi(x) \le C_\varpi^0 x$ when $x\in [0,\vartheta^0 ]$. Hence the following upper bound
$$
\left| \frac{1}{k}\frac{1}{|x|} \varpi_k(|x|)\right| \le \frac{1}{k} \frac{2}{K_m \lambda_-^k} C_\chi^0=\mathrm{O}(1/\lambda_-^k)\quad \mbox{if }  \left|\frac{2}{K_m \lambda_-^k}x \right|\le \vartheta^0 \iff  |x| \le \frac{K_m \lambda_-^k}{2} \vartheta^0 .
$$
If $|x|\ge ({K_m \lambda_-^k}\vartheta^0)/{2} $ then $\left|  \varpi_k(|x|)/({k}{|x|})\right|\le \left| 1/({k}{|x|} )\right| C_\varpi \le \left| {2}/({k}{K_m \lambda_-^k \vartheta^0}) \right| C_\varpi=\mathrm{O}(1/\lambda_-^k)$ which, coupled with the above inequality, yields that the first term on the righthandside of \eqref{proof_lemma_estim_gen4} is an $\mathrm{O}(1/\lambda_-^k)$. All in all, we thus obtain the inequality \eqref{lemma_prop_main2}. With similar arguments, \eqref{lemma_prop_main2_bis} is obtained by differentiating $H_k'(x)$ and using $\varpi(x)=\mathrm{o}(x^3)$.\\
We now prove \eqref{lemma_prop_main3} from the two afore established upper bounds \eqref{lemma_prop_main1} and  \eqref{lemma_prop_main2}. We observe that $\psi$ is twice differentiable, and compute its first derivative w.r.t. $a$ as
\begin{equation}\label{derive_psi}
\partial_a \psi_k(a,b)=G'(a)G(b) H_k(a^2-b) + 2a G(a)G(b) H_k'(a^2-b).
\end{equation}
Since $\varpi(x)=\mathrm{O}(x)$ as $x\to 0$, a similar approach leading to \eqref{proof_lemma_estim_gen2} yields that $|\varpi_k(|x|) \ln (|x|)| = \mathrm{O}(k/\lambda_-^k)$ on $|x|\le ({K_m \lambda_-^k}\vartheta^0)/{2} $. Consequently we have that
\begin{equation}\label{first_uniform_bound}
|G'(a)G(b) H_k(a^2-b)|= \mathrm{O}(1/\lambda_-^k),\quad |a^2-b| \in \left[0, \frac{K_m \lambda_-^k}{2} \vartheta^0 \right].
\end{equation}
Since $G$ has a finite support and is thus zero outside of some compact interval say $I$, $(a,b)\mapsto G'(a)G(b) H_k(a^2-b)$ is $0$ when $(a,b)\notin I^2$, and is continuous hence uniformly bounded on the compact set $I^2 \cap \left\{ (a,b)\in \nbR^2|\ |a^2-b|> ({K_m \lambda_-^k} \vartheta^0)/ {2}\right\}$. Thanks to the factor $1/k$ in the definition \eqref{def_H_k} of $H_k$, it is easily verified that this uniform bound on that set is in fact an $\mathrm{O}(1/k)$, and so is also an $\mathrm{O}(1/\lambda_-^k)$. This latter fact, coupled with \eqref{first_uniform_bound}, yields that the first term on the righthandside of \eqref{derive_psi} is an $\mathrm{O}(1/\lambda_-^k)$ uniformly in $(a,b)\in \nbR^2$. Similarly, a similar analysis along with the inequality \eqref{lemma_prop_main2} proved previously yields that second term on the righthandside of \eqref{derive_psi} is an $\mathrm{O}(1/\lambda_-^k)$ uniformly in $(a,b)\in \nbR^2$. Again, a similar analysis yields similar upper bounds for $\partial_b \psi_k(a,b)$, so that \eqref{lemma_prop_main3} holds. The upper bound \eqref{lemma_prop_main3_bis} is obtained by similar arguments thanks to \eqref{lemma_prop_main2_bis}. This concludes the proof of Lemma \ref{lemma_before_prop_main_theo1}
\end{proof}
We now turn to the proof of Proposition \ref{prop_C_1_u_k}. Let us set $U(k):=(C_1,u_k)$, $k\in \nbN$. The definition \eqref{def_estimator_general} as well as the finite increment theorem yields the existence of some (random) $c_{k,n}\in (0,1)$ such that
\begin{eqnarray*}
\psi_{k}(\bar{Y}_n, \bar{Y}_{n,k}) - \psi_k(U(k))&=& \psi_k\left(U(k) + \frac{S_n(k)}{n}\right) - \psi_k(U(k))\\
&=&\nabla \psi_k\left(U(k) + c_{k,n}\frac{S_n(k)}{n}\right).\frac{S_n(k)}{n},
\end{eqnarray*}
so that, thanks to \eqref{lemma_prop_main3}, we obtain the $L^2$ bound
\begin{equation}\label{proof_main_theo1_1}
||\psi_{k}(\bar{Y}_n, \bar{Y}_{n,k}) - \psi_k(U(k))||_2 \le M_1(k) \left|\left| \frac{S_n(k)}{n}\right|\right|_2.
\end{equation}
We now observe that, for $k$ large enough we have from \eqref{expr_C_1_u_k} as well as Assumption $\mathbf{ (A6)}$ that $|C_1^2-u_k|=\lambda_0^k \left|\sum_{j=0}^k \lambda_0^{-j} \chi_j + \lambda_0(C_1^2-C_2)\right|\ge \lambda_0^k K_m/2$, so that the definition of $\chi_k$ entails that $\varpi_k(|C_1^2-u_k|)=1$. Now, by Assumption $\mathbf{ (A7)}$ we have $0\le C_1=||Y_0||_1 \le C_Y$ and $0\le u_k=\nbE(Y_0 Y_{k+1})\le ||Y_0||_2 ||Y_{k+1}||_2=||Y_0||_2^2\le C_Y^2$ (by the Cauchy Scwharz inequality) so that the definition of function $G$ entails that $G(C_1)=G(u_k)=1$. Thus $\psi_k(U(k))=G(C_1)G(u_k)H_k(C_1^2-u_k)=\frac{1}{k}\ln |C_1^2-u_k|=S_k$ for $k$ large enough. All in all, we thus have, from \eqref{observation_S_k} and \eqref{expr_C_1_u_k}, that
$$
|\psi_k(U(k)) - \ln \lambda_0|=\frac{1}{k}\left| \sum_{j=0}^k \lambda_0^{-j} \chi_j + \lambda_0(C_1^2-C_2)\right|=\mathrm{O}\left(\frac{1}{k}\right)
$$
for $k$ large enough. The upper bound \eqref{bound_S_k_n} is deduced from the above inequality, Minkowski's inequality, \eqref{proof_main_theo1_1} as well as the estimate $M_1(k)=\mathrm{O}\left({1}/{\lambda_-^k} \right)$ proved in Lemma \ref{lemma_before_prop_main_theo1}. \zak

\subsection{Proof of Proposition \ref{prop_main_theo2}}\label{sec:proof_prop_main_theo2}
The scheme somewhat has some common ideas with the proof of Theorem \ref{theo_CLT_vector_k_0} in Section \ref{sec:proof_theo_k_0} (with the notable difference that $k_0$ in this section is now replaced with $k_n$ with $\lim_{n\to\infty}k_n=\infty$) combined with a triangular central limit theorem for stationary sequences of random variables proved in \cite{FZ05}, and is decomposed in several steps. Let $(k_n)_{n\in \N}$ be a sequence of integers verifying \eqref{Cond_CLT_k_n}.\\
\paragraph{$\diamond$ Step 1:} we first prove the existence of the limit $\mathfrak{S}$ in \eqref{CLT_general_Sigma}. We first observe that
\begin{equation}\label{pr_prop_k_n_1}
\frac{1}{n} \nbE(S_n(k_n) S_n(k_n)')=\frac{1}{n} \sum_{h=-n}^n (n-|h|)\nbE \left(X_0(k_n)X_h(k_n)'\right)
\end{equation}
where we recall that $S_n(.)$ and $X_h(.)$ are defined in \eqref{def_estimators}. Since
\begin{equation}\label{pr_prop_k_n_2}
\nbE \left(X_0(k_n)X_h(k_n)'\right)= \left[
\begin{array}{cc}
\mbox{Cov}(Y_0,Y_h) & \mbox{Cov}(Y_0, Y_h Y_{h+k_n+1})\\
 \mbox{Cov}(Y_0 Y_{k_n+1}, Y_h ) & \mbox{Cov}(Y_0Y_{k_n+1}, Y_hY_{h+k_n+1})
\end{array}
\right],
\end{equation}
we then see that \eqref{pr_prop_k_n_1} converges if we prove that $\sum_{h=1}^n \nbE \left(X_0(k_n)X_h(k_n)'\right)$ and\\ $\sum_{h=1}^n h\; \nbE \left(X_0(k_n)X_h(k_n)'\right)$ converge by the dominant convergence theorem, in which case the limit of \eqref{pr_prop_k_n_1} is related to the limit of $\sum_{h=1}^n \nbE \left(X_0(k_n)X_h(k_n)'\right)$ as $n\to\infty$. For this we study the corresponding terms in \eqref{pr_prop_k_n_2}. The simplest term in the latter expression is $\mbox{Cov}(Y_0,Y_h)$, of which corresponding series converges thanks to a direct application of the bound \eqref{bound_sum_Davydov1} that yields $|\mbox{Cov}(Y_0,Y_h)|\le K \left[ \alpha_{{\epsilon}}\left( \left\lfloor ({h-1})/{2}\right\rfloor\right)^{1-2/\beta} + \lambda_0^{\left\lfloor \frac{h-1}{2}\right\rfloor}\right]$ for some constant $K>0$. In that case the corresponding limit may also be written as
\begin{equation}
\frac{1}{n}\sum_{h=-n}^n (n-|h|) \mbox{Cov}(Y_0,Y_h) \longrightarrow 2 \sum_{h=1}^\infty (C_1^2 - u_{h-1}),\quad n\to \infty .\label{pr_prop_k_n_2bis}
\end{equation}
We next study $\mbox{Cov}(Y_0, Y_h Y_{h+k_n+1})$. We get from \eqref{decompo1} that
\begin{eqnarray}
\mbox{Cov}(Y_0, Y_h Y_{h+k_n+1})&=& \sum_{s=0}^{k_n}\lambda_0^s \left[ \sum_{t=0}^{h-1} \lambda_0^t \mbox{Cov}(Y_0, \epsilon_{h-t} \epsilon_{k_n+1+h-s}) + \lambda_0^{h}
\mbox{Cov}(Y_0, Y_0 \epsilon_{k_n+1+h-s}) \right]\nonumber\\
&& + \lambda_0^{k_n+1} \mbox{Cov}(Y_0, Y_h^2).\label{pr_prop_k_n_3}
\end{eqnarray}
We wish to apply the dominant convergence theorem to establish the limit of\\ $\sum_{h=1}^{n} \mbox{Cov}(Y_0, Y_h Y_{h+k_n+1})$ as $n\to\infty$. For this we observe that, similarly to \eqref{ineg_utile}:
$$
|\mbox{Cov}(Y_0, Y_h Y_{h+k_n+1})| \le K \left[ \alpha_\epsilon \left( \left\lfloor \frac{h-1}{2}\right\rfloor\right)^{1-2/\beta} + \lambda_0^{\left\lfloor \frac{h-1}{2}\right\rfloor} + \lambda_0^{h}\right]
$$
for some constant $K>0$, of which righthandside upper bound is independent from $n$ and summable in $h\ge 1$. Thus, we now need to study the convergence of the terms on the righthandside of \eqref{pr_prop_k_n_3} as $n\to\infty$ for fixed $h\ge 1$. We start with $\sum_{s=0}^{k_n}\lambda_0^s  \sum_{t=0}^{h-1} \lambda_0^t \mbox{Cov}(Y_0, \epsilon_{h-t} \epsilon_{k_n+1+h-s})$. We write that
\begin{eqnarray*}
\mbox{Cov}(Y_0, \epsilon_{h-t} \epsilon_{k_n+1+h-s})&=&\mbox{Cov}(Y_0 \epsilon_{h-t} ,\epsilon_{k_n+1+h-s}) + \nbE(Y_0 \epsilon_{h-t})\nbE(\epsilon_{k_n+1+h-s})\\
&& - \nbE(Y_0)\nbE(\epsilon_{h-t} \epsilon_{k_n+1+h-s})\\
&=& \mbox{Cov}(Y_0 \epsilon_{h-t} ,\epsilon_{k_n+1+h-s}) +  \nbE(Y_0 \epsilon_{h-t})m_0 - C_1 \nbE(\epsilon_{h-t} \epsilon_{k_n+1+h-s}),
\end{eqnarray*}
of which terms are easier to analyse than $\mbox{Cov}(Y_0, \epsilon_{h-t} \epsilon_{k_n+1+h-s})$ directly. For fixed $h,s,t\ge 1$, the Davydov inequality entails that $\lambda_0^s \lambda_0^t\mbox{Cov}(Y_0 \epsilon_{h-t} ,\epsilon_{k_n+h-s})\longrightarrow 0$ as $n\to \infty$. Furthermore, the Cauchy Schwarz inequality yields
$$
|\lambda_0^s \mbox{Cov}(Y_0 \epsilon_{h-t} ,\epsilon_{k_n+1+h-s})| \le \lambda_0^s  ||Y_0||_4 ||\epsilon||_4 ||\epsilon||_2
$$
of which upper bound is summable w.r.t. $s\in \nbN$, so that a domination convergence theorem argument yields that
$$
\sum_{s=0}^{k_n}\lambda_0^s  \sum_{t=0}^{h-1} \lambda_0^t \mbox{Cov}(Y_0\epsilon_{h-t} ,\epsilon_{k_n+1+h-s}) \longrightarrow 0,\quad n\to \infty ,
$$
for fixed $h\ge 1$. Moreover, the following limits are easily obtained as $n\to\infty$:
\begin{eqnarray*}
\sum_{s=0}^{k_n}\lambda_0^s  \sum_{t=0}^{h-1} \lambda_0^t \nbE(Y_0 \epsilon_{h-t})m_0 & \longrightarrow & \sum_{t=0}^{h-1} \lambda_0^t \nbE(Y_0 \epsilon_{h-t}) \frac{m_0}{1-\lambda_0}= \sum_{t=0}^{h-1} \lambda_0^t v_{h-t-1} \frac{m_0}{1-\lambda_0},\\
\sum_{s=0}^{k_n}\lambda_0^s  \sum_{t=0}^{h-1} \lambda_0^t C_1 \nbE(\epsilon_{h-t} \epsilon_{k_n+1+h-s}) & \longrightarrow & \sum_{t=0}^{h-1} \lambda_0^t  C_1 m_0^2 \frac{1}{1-\lambda_0}
\end{eqnarray*}
so that the following limit is established:
\begin{eqnarray}
\sum_{s=0}^{k_n}\lambda_0^s  \sum_{t=0}^{h-1} \lambda_0^t \mbox{Cov}(Y_0, \epsilon_{h-t} \epsilon_{k_n+1+h-s}) & \longrightarrow & \frac{m_0}{1-\lambda_0} \sum_{t=0}^{h-1} \lambda_0^t \left[ v_{h-t-1} - C_1 m_0\right]\nonumber\\
 &=& - \frac{m_0}{1-\lambda_0} \sum_{t=0}^{h-1} \lambda_0^t \chi_{h-t-1},\quad n\to \infty ,
\label{pr_prop_k_n_4}
\end{eqnarray}
thanks to Relations \eqref{expr_v_k} and \eqref{expr_first_moment}. A similar analysis yields
\begin{eqnarray}
\sum_{s=0}^{k_n}\lambda_0^s \lambda_0^{h} \mbox{Cov}(Y_0, Y_0 \epsilon_{k_n+1+h-s}) &=& \sum_{s=0}^{k_n}\lambda_0^s \lambda_0^{h} \left[\mbox{Cov}(Y_0^2,\epsilon_{k_n+1+h-s}) +C_2 m_0- C_1 \nbE(Y_0 \epsilon_{k_n+1+h-s})\right]\nonumber\\
& \longrightarrow & \sum_{s=0}^\infty \lambda_0^s \lambda_0^{h} \times 0  + \sum_{s=0}^\infty \lambda_0^s \lambda_0^{h} C_2 m_0 - \sum_{s=0}^\infty \lambda_0^s \lambda_0^{h}C_1^2 m_0\nonumber\\
&=& \lambda_0^{h} m_0 \frac{C_2-C_1^2}{1-\lambda_0}\label{pr_prop_k_n_5}\\
\lambda_0^{k_n+1} \mbox{Cov}(Y_0, Y_h^2) &\longrightarrow & 0 \label{pr_prop_k_n_6}
\end{eqnarray}
as $n\to\infty$, for fixed $h\ge 1$. Combining \eqref{pr_prop_k_n_4}, \eqref{pr_prop_k_n_5}, \eqref{pr_prop_k_n_6} and \eqref{pr_prop_k_n_3}, we then obtain
\begin{multline}\label{pr_prop_k_n_7_minus1}
\sum_{h=1}^n \mbox{Cov}(Y_0, Y_h Y_{h+k_n+1}) \longrightarrow - \frac{m_0}{1-\lambda_0} \sum_{h=1}^\infty \sum_{t=0}^{h-1} \lambda_0^t \chi_{h-t-1} +   m_0 \frac{C_2-C_1^2}{1-\lambda_0}  \sum_{h=1}^\infty \lambda_0^{h}\\
= - \frac{m_0}{1-\lambda_0} \frac{1}{1-\lambda_0} \sum_{h=0}^\infty \chi_h +  m_0 \frac{C_2-C_1^2}{1-\lambda_0}\frac{\lambda_0}{1-\lambda_0}= \frac{m_0}{(1-\lambda_0)^2} \left[ - \sum_{h=0}^\infty \chi_h + \lambda_0(C_2-C_1^2) \right],\quad n\to\infty .
\end{multline}
A similar analysis yields that $\sum_{h=1}^n h\; \mbox{Cov}(Y_0, Y_h Y_{h+k_n+1})$ converges as $n\to\infty$ so that $\lim_{n\to\infty} \sum_{h=1}^n \frac{h}{n}\; \mbox{Cov}(Y_0, Y_h Y_{h+k_n})=0$, from which, along with \eqref{pr_prop_k_n_7_minus1} and the fact that $\lim_{n\to\infty} \mbox{Cov}(Y_0, Y_0 Y_{k_n+1})=0$, gives that
\begin{equation}\label{pr_prop_k_n_7_0}
\frac{1}{n}\sum_{h=1}^n (n-|h|)\mbox{Cov}(Y_0, Y_h Y_{h+k_n+1}) \longrightarrow 2 \frac{m_0}{(1-\lambda_0)^2} \left[ - \sum_{h=0}^\infty \chi_h + \lambda_0(C_2-C_1^2) \right],\quad n\to\infty .
\end{equation}
We then study the term $\sum_{h=-1}^{-n} \mbox{Cov}(Y_0, Y_h Y_{h+k_n+1})$. A change of index $h:=-h$ as well as the stationarity of the process $\{Y_n,\ n\in\nbZ \}$ yields
\begin{equation}\label{pr_prop_k_n_7_01}
\sum_{h=-n}^{-1} \mbox{Cov}(Y_0, Y_h Y_{h+k_n+1})=\sum_{h=1}^{n} \mbox{Cov}(Y_h, Y_0 Y_{k_n+1}).
\end{equation}
We split the righthandside of \eqref{pr_prop_k_n_7_01} into the two sums $\sum_{h=1}^{k_n}$ and $\sum_{h=k_n+1}^{n}$. Writing for $h=1,...,k_n$ that 
\begin{eqnarray}
\mbox{Cov}(Y_h, Y_0 Y_{k_n+1})&=&\mbox{Cov}(Y_0, Y_h Y_{k_n+1})+ \nbE(Y_0)\nbE(Y_h Y_{k_n+1})- \nbE(Y_h)\nbE(Y_0 Y_{k_n+1})\nonumber\\
&=&\mbox{Cov}(Y_0, Y_h Y_{k_n+1})+ C_1 u_{k_n-h}-C_1 u_{k_n}\nonumber\\
&=& \mbox{Cov}(Y_0, Y_h Y_{k_n+1})+ C_1 (u_{k_n-h}-C_1^2)-C_1 (u_{k_n}-C_1^2),\label{pr_prop_k_n_7_02}
\end{eqnarray}
we notice that, similarly to \eqref{ineg_utile}, $|\mbox{Cov}(Y_0, Y_h Y_{k_n+1})|\le  K \left[ \alpha_\epsilon \left( \left\lfloor {h}/{2}\right\rfloor\right)^{1-2/\beta} + \lambda_0^{\left\lfloor \frac{h}{2}\right\rfloor} + \lambda_0^{h}\right]$ which, since the mixing coefficient function $\alpha_\epsilon(\cdot)$ is decreasing and $\left \lfloor {h}/{2}\right\rfloor \ge \lfloor{(k_n}/{4})-1\rfloor$, $h=\left \lfloor {k_n}/{2}\right\rfloor,...,k_n$, yields that
\begin{equation}\label{pr_prop_k_n_7_03}
\left|\sum_{h=\left\lfloor \frac{k_n}{2}\right\rfloor}^{k_n} \mbox{Cov}(Y_0, Y_h Y_{k_n})\right|\le k_n K \left[ \alpha_\epsilon \left( \left\lfloor  \frac{k_n}{4}-1\right\rfloor\right)^{1-2/\beta} + \lambda_0^{  \left\lfloor  \frac{k_n}{4}-1\right\rfloor     } + \lambda_0^{ \left\lfloor \frac{k_n}{2}\right\rfloor       }\right] \longrightarrow 0,
\end{equation}
Next, writing $\mbox{Cov}(Y_0, Y_h Y_{k_n+1})=\mbox{Cov}(Y_0 Y_h, Y_{k_n+1})+ \nbE(Y_0 Y_h)\nbE(Y_{k_n+1})- \nbE(Y_0)\nbE(Y_h Y_{k_n+1})= \mbox{Cov}(Y_0 Y_h, Y_{k_n+1})+ C_1(u_{h-1}-C_1^2) + C_1(C_1^2-u_{k_n-h})$, and since, by an argument similar to \eqref{pr_prop_k_n_7_03}, $\sum_{h=1}^{ \left\lfloor \frac{k_n}{2}\right\rfloor -1} \mbox{Cov}(Y_0 Y_h, Y_{k_n+1}) \longrightarrow 0$ as $n\to\infty$, and $\sum_{h=1}^{ \left\lfloor \frac{k_n}{2}\right\rfloor -1} (C_1^2-u_{k_n-h})=\sum_{h=k_n-\left\lfloor \frac{k_n}{2}\right\rfloor+1}^{k_n-1} (C_1^2-u_{h})\longrightarrow 0$ as $n\to\infty$, we thus get that 
$$
\sum_{h=1}^{ \left\lfloor \frac{k_n}{2}\right\rfloor -1}\mbox{Cov}(Y_0, Y_h Y_{k_n+1})\longrightarrow C_1 \sum_{h=1}^\infty (u_{h-1}-C_1^2),\quad n\to\infty ,
$$
which, coupled with \eqref{pr_prop_k_n_7_03}, yields that
\begin{equation}\label{pr_prop_k_n_7_04}
\sum_{h=1}^{k_n} \mbox{Cov}(Y_0, Y_h Y_{k_n+1}) \longrightarrow C_1 \sum_{h=1}^\infty (u_{h-1}-C_1^2),\quad n\to\infty .
\end{equation}
Pluged into \eqref{pr_prop_k_n_7_02}, and since $k_n (u_{k_n}-C_1^2) = \mathrm{O}(k_n \lambda_0^{k_n})\longrightarrow 0$ as $n\to \infty$ thanks to Corollary \ref{prop_C_1_u_k} and $\sum_{h=1}^{k_n} (u_{k_n-h}-C_1^2)=\sum_{h=0}^{k_n-1} (u_{h}-C_1^2) \longrightarrow \sum_{h=0}^{\infty } (u_{h-1}-C_1^2)$, we thus get
\begin{equation}\label{pr_prop_k_n_7_05}
\sum_{h=1}^{k_n}  \mbox{Cov}(Y_h, Y_0 Y_{k_n+1}) \longrightarrow 2C_1 \sum_{h=1}^{\infty } (u_{h-1}-C_1^2),\quad n\to \infty .
\end{equation}
We now consider the $\sum_{h=k_n+1}^{n}$ sum in the splitting of righthandside sum of \eqref{pr_prop_k_n_7_01}. Using the stationarity of $\{Y_n,\ n\in\nbZ \}$ and a change of index $h:=h-k_n$, we first write
\begin{equation}\label{pr_prop_k_n_7_06}
\sum_{h=k_n+1}^n  \mbox{Cov}(Y_h, Y_0 Y_{k_n+1})=\sum_{h=0}^{n-k_n-1} \mbox{Cov}( Y_{-k_n-1}Y_0 ,Y_h).
\end{equation}
We use a dominated convergence argument to determine the limit of \eqref{pr_prop_k_n_7_06} as $n\to \infty$. Similarly to \eqref{ineg_utile}, the following inequality holds
\begin{equation}\label{pr_prop_k_n_7_07}
| \mbox{Cov}( Y_{-k_n-1}Y_0 ,Y_h)| \le K \left[ \alpha_\epsilon \left( \left\lfloor \frac{h}{2}\right\rfloor\right)^{1-2/\beta} + \lambda_0^{\left\lfloor \frac{h}{2}\right\rfloor} + \lambda_0^{h}\right],
\end{equation}
the upper bound being summable. Next, we decompose the summand in the righthandside of \eqref{pr_prop_k_n_7_06} in the same spirit as \eqref{pr_prop_k_n_7_02} as
\begin{equation}\label{pr_prop_k_n_7_08}
\mbox{Cov}( Y_{-k_n-1}Y_0 ,Y_h)= \mbox{Cov}( Y_{-k_n-1},Y_0 Y_h) + C_1 (u_{h-1}-C_1^2)-C_1 (u_{k_n}-C_1^2).
\end{equation}
For fixed $h\in \nbN$, a standard argument yields that $$|\mbox{Cov}( Y_{-k_n-1},Y_0 Y_h) | \le K \left[ \alpha_\epsilon \left( \left\lfloor \frac{k_n+1}{2}\right\rfloor\right)^{1-2/\beta} + \lambda_0^{\left\lfloor \frac{k_n+1}{2}\right\rfloor} + \lambda_0^{h}\right] \longrightarrow 0$$ as $n\to \infty$. We already saw that  $\lim_{n\to \infty }u_{k_n}-C_1^2=0$, so that from \eqref{pr_prop_k_n_7_08} the following limit holds
\begin{equation}\label{pr_prop_k_n_7_09}
\mbox{Cov}( Y_{-k_n-1}Y_0 ,Y_h)\longrightarrow  C_1 (u_{h-1}-C_1^2),\quad n\to \infty,\ \forall h\in \nbN .
\end{equation}
Hence, \eqref{pr_prop_k_n_7_06} along with \eqref{pr_prop_k_n_7_07} and \eqref{pr_prop_k_n_7_09} yield by the dominated convergence theorem that
$$
\sum_{h=k_n+1}^n  \mbox{Cov}(Y_h, Y_0 Y_{k_n+1}) \longrightarrow \sum_{h=0}^\infty C_1 (u_{h-1}-C_1^2)  =  C_1(C_2-C_1^2) + C_1 \sum_{h=1}^{\infty } (u_{h-1}-C_1^2),
$$
which, coupled with \eqref{pr_prop_k_n_7_05} as well as \eqref{pr_prop_k_n_7_01}, yields
$$
\sum_{h=-n}^{-1} \mbox{Cov}(Y_0, Y_h Y_{h+k_n+1}) \longrightarrow C_1(C_2-C_1^2) + 3C_1 \sum_{h=1}^{\infty } (u_{h-1}-C_1^2),\quad n\to \infty ,
$$
which in turn gives
$$
\frac{1}{n}\sum_{h=-n}^{-1} (n-|h|)\mbox{Cov}(Y_0, Y_h Y_{h+k_n+1}) \longrightarrow 2C_1(C_2-C_1^2) + 3C_1 \sum_{h=1}^{\infty } (u_{h-1}-C_1^2),\quad n\to \infty .$$
Gathering \eqref{pr_prop_k_n_7_0} and the above limit we eventually arrive after some easy calculation at
\begin{multline}\label{pr_prop_k_n_7}
\frac{1}{n}\sum_{h=-n}^n (n-|h|)\mbox{Cov}(Y_0, Y_h Y_{h+k_n+1}) \longrightarrow -2 \frac{m_0}{(1-\lambda_0)^2}  \sum_{h=0}^\infty \chi_h \\
+ \frac{C_1}{1-\lambda_0}(C_2-C_1^2) +3 C_1 \sum_{h=1}^{\infty } (u_{h-1}-C_1^2) ,\quad n\to\infty .
\end{multline}
We then study $\mbox{Cov}(Y_0Y_{k_n+1}, Y_hY_{h+k_n+1})$. The limit of $\sum_{h=1}^n \mbox{Cov}(Y_0Y_{k_n+1}, Y_hY_{h+k_n+1})$ is obtained by splitting the sum in three blocks, namely $\sum_{h=1}^{\lfloor k_n/2 \rfloor}$, $\sum_{\lfloor k_n/2 \rfloor+1}^{k_n+1}$ and $\sum_{h=k_n+2}^n$.\\
We start by the summation over $h=1,...,\lfloor k_n/2 \rfloor$. We exploit here the fact that $h$ is "far" from $k_n+1$, and write that
\begin{eqnarray}
\mbox{Cov}(Y_0Y_{k_n+1}, Y_hY_{h+k_n+1}) &=& \mbox{Cov}(Y_0 Y_h, Y_{k_n+1} Y_{h+k_n+1}) + \nbE(Y_0 Y_h) \nbE(Y_{k_n+1} Y_{k_n+1+h}) \nonumber\\
&& - \nbE(Y_0 Y_{k_n+1})\nbE(Y_{h} Y_{k_n+1+h})\nonumber\\
&=& \mbox{Cov}(Y_0 Y_h, Y_{k_n+1} Y_{h+k_n+1}) + u_{h-1}^2 - u_{k_n}^2.\label{pr_prop_k_n_8}
\end{eqnarray}
Similarly to \eqref{ineg_utile} we have $$|\mbox{Cov}(Y_0Y_{h}, Y_{k_n+1}Y_{h+k_n+1})| \le K \left[ \alpha_\epsilon\left( \left\lfloor \frac{k_n-h}{2}\right\rfloor\right)^{1-2/\beta} + \lambda_0^{\left\lfloor \frac{k_n-h}{2}\right\rfloor} + \lambda_0^{k_n+1-h}\right]$$ which, since the mixing coefficient function $\alpha_\epsilon(\cdot)$ is decreasing and $\lfloor ({k_n-h})/{2} \rfloor \ge \lfloor({k_n}/{4})-1\rfloor$, $h=1,...,\lfloor k_n/2 \rfloor$, entails that
\begin{multline}\label{pr_prop_k_n_9}
\sum_{h=1}^{\lfloor k_n/2 \rfloor} |\mbox{Cov}(Y_0Y_{h}, Y_{k_n+1}Y_{h+k_n+1})| \le K \lfloor k_n/2 \rfloor \left[ \alpha_\epsilon\left( \left\lfloor\frac{k_n}{4}-1\right\rfloor \right)^{1-2/\beta} + \lambda_0^{\left\lfloor\frac{k_n}{4}-1\right\rfloor} + \lambda_0^{(k_n+1)/2}\right]\\ 
\longrightarrow 0,\quad n\to \infty .
\end{multline}
Remembering that $C_1^2 -u_k = \mathrm{O}(\lambda_0^k)$ from Corollary \ref{prop_C_1_u_k}, so that $\lim_{k\to \infty}k(C_1^2 -u_k)=0$, we also have, writing $u_{h-1}^2 - u_{k_n-1}^2= (u_{h-1}^2 -C_1^4) + (C_1^
4- u_{k_n-1}^2)$:
\begin{equation}\label{pr_prop_k_n_10}
\sum_{h=1}^{\lfloor k_n/2 \rfloor} (u_{h-1}^2 - u_{k_n-1}^2)= \sum_{h=1}^{\lfloor k_n/2 \rfloor} (u_{h-1}^2 -C_1^4) + \lfloor k_n/2 \rfloor (C_1^4- u_{k_n-1}^2) \longrightarrow  \sum_{h=1}^{\infty} (u_{h-1}^2 -C_1^4),\quad n\to \infty ,
\end{equation}
as indeed $\lfloor k_n/2 \rfloor (C_1^ 4- u_{k_n-1}^2)= \mathrm{O}\left(\lfloor k_n/2 \rfloor (C_1^2- u_{k_n-1})\right)\longrightarrow 0$ as $n\to\infty$. Gathering \eqref{pr_prop_k_n_9} and \eqref{pr_prop_k_n_10} in \eqref{pr_prop_k_n_8} thus yields
\begin{equation}\label{pr_prop_k_n_11}
\sum_{h=1}^{\lfloor k_n/2 \rfloor} \mbox{Cov}(Y_0Y_{k_n}, Y_hY_{h+k_n}) \longrightarrow \sum_{h=1}^{\infty} (u_{h-1}^2 -C_1^4),\quad n\to \infty.
\end{equation}
We now consider the summation over $h=\lfloor k_n/2 \rfloor +1,..., k_n+1$. We this time exploit the fact that $h$ and $k_n+1+h$ are respectively "far" from $0$ and $k_n+1$, and write
\begin{eqnarray}
\mbox{Cov}(Y_0Y_{k_n+1}, Y_hY_{h+k_n+1}) &=& \mbox{Cov}(Y_0 Y_h Y_{k_n+1} ,Y_{h+k_n+1}) + [\mbox{Cov}(Y_0, Y_h Y_{k_n+1} ) \nonumber\\
&&  + \nbE(Y_0)\nbE(Y_h Y_{k_n+1})]\nbE(Y_{h+k_n+1})- \nbE(Y_{0}Y_{k_n+1}) \nbE(Y_{h}Y_{h+k_n+1}) \nonumber\\
&=& \mbox{Cov}(Y_0 Y_h Y_{k_n+1} ,Y_{h+k_n+1}) + [\mbox{Cov}(Y_0, Y_h Y_{k_n+1} ) \nonumber\\
&& + C_1 u_{k_n-h}]C_1 - u_{k_n}^2 .\label{pr_prop_k_n_12}
\end{eqnarray}
Similarly to \eqref{pr_prop_k_n_9}:
\begin{eqnarray*}
\sum_{\lfloor k_n/2 \rfloor +1 }^{k_n+1}| \mbox{Cov}(Y_0 Y_h Y_{k_n+1} ,Y_{h+k_n}+1)| &\le & K \lfloor k_n/2 \rfloor \left[ \alpha_\epsilon\left( \left\lfloor\frac{k_n}{4}-1\right\rfloor \right)^{1-2/\beta} + \lambda_0^{\left\lfloor\frac{k_n}{4}-1\right\rfloor} + \lambda_0^{  \frac{k_n+1}{2}    }\right]\\&&\longrightarrow 0,\quad n\to\infty ,\\
\sum_{\lfloor k_n/2 \rfloor +1}^{k_n+1} |\mbox{Cov}(Y_0, Y_h Y_{k_n+1} ) |  &\le & K \lfloor k_n/2 \rfloor \left[ \alpha_\epsilon\left( \left\lfloor\frac{k_n}{4}-1\right\rfloor \right)^{1-2/\beta} + \lambda_0^{\left\lfloor\frac{k_n}{4}-1\right\rfloor} + \lambda_0^{ \frac{k_n+1}{2} }\right]\\&&\longrightarrow 0, \quad n\to\infty .
\end{eqnarray*}
And, writing $C_1^2 u_{k_n-h} -  u_{k_n}^2 = C_1^2 (u_{k_n-h}- C_1^2) + C_1^4 -  u_{k_n}^2$, we obtain, similarly to \eqref{pr_prop_k_n_10}:
\begin{multline*}
\sum_{h=\lfloor k_n/2 \rfloor +1}^{k_n+1} C_1^2 u_{k_n-h} -  u_{k_n}^2 =   C_1^2 \sum_{h=0}^{k_n- \lfloor k_n/2 \rfloor - 1} (u_{h-1}- C_1^2) + \left( k_n - \lfloor k_n/2 \rfloor \right)(C_1^4 -  u_{k_n}^2)\\
\longrightarrow C_1^2 \sum_{h=0}^\infty (u_{h-1}- C_1^2),\quad n\to \infty ,
\end{multline*}
so that we obtain from \eqref{pr_prop_k_n_12} the following limit
\begin{equation}\label{pr_prop_k_n_13}
\sum_{h=\lfloor k_n/2 \rfloor +1 }^{k_n+1} \mbox{Cov}(Y_0Y_{k_n}, Y_hY_{h+k_n}) \longrightarrow C_1^2 \sum_{h=0}^\infty (u_{h-1}- C_1^2),\quad n\to\infty .
\end{equation}
Then, we consider the summation over $h=k_n+2,...,n$. The expansion established in \eqref{proof_step_1_CTL_k0} reads here
\begin{multline}\label{pr_prop_k_n_14}
\mbox{Cov}(Y_0Y_{k_n+1}, Y_hY_{h+k_n+1})=\sum_{s=0}^{k_n}\lambda_0^s \left[ \sum_{t=0}^{h-k_n-2} \lambda_0^t \mbox{Cov}(Y_0Y_{k_n+1}, \epsilon_{h-t} \epsilon_{h+k_n+1-s})\right. \\ \left. + \lambda_0^{h-k_n-1}
\mbox{Cov}(Y_0Y_{k_n+1}, Y_{k_n+1} \epsilon_{h+k_n+1-s}) \right]
 + \lambda_0^{k_n+1} \mbox{Cov}(Y_0Y_{k_n+1}, Y_h^2).
\end{multline}
We study the summation over $h=k_n+2,...,n$ of each term in \eqref{pr_prop_k_n_14}, and start with the first term, which reads, thanks to a change of variable $r:=h-k_n-1$,
\begin{eqnarray}
\sum_{h=k_n+2}^n \sum_{s=0}^{k_n}\  \sum_{t=0}^{h-k_n-2} \lambda_0^{s+t} \mbox{Cov}(Y_0Y_{k_n+1}, \epsilon_{h-t} \epsilon_{h+k_n+1-s})&=& \sum_{r=1}^{n-k_n-1} J_1(r,k_n),\label{pr_prop_k_n_15}\\
J_1(r,k_n)&:=& \sum_{s=0}^{k_n} \sum_{t=0}^{r-1} \lambda_0^{s+t} \mbox{Cov}(Y_0Y_{k_n+1}, \epsilon_{r+k_n+1-t} \epsilon_{r+2k_n+2-s}).\nonumber
\end{eqnarray}
Studying the limit of \eqref{pr_prop_k_n_15} is done by applying the dominant convergence theorem. We see from \eqref{bound_sum_Davydov0} that $\left|  \sum_{t=0}^{r-1} \lambda_0^{t} \mbox{Cov}(Y_0Y_{k_n+1}, \epsilon_{r+k_n+1-t} \epsilon_{r+2k_n+2-s})\right|\le K\left[ \alpha_\epsilon\left( \left\lfloor ({r-1})/{2}\right\rfloor\right)^{1-2/\beta} + \lambda_0^{\left\lfloor \frac{r-1}{2}\right\rfloor}\right]$ for all $s=0,...,k_n$, so that $|J_1(r,k_n)|$ is dominated as follows
$$
|J_1(r,k_n)|\le  K\left[ \alpha_\epsilon\left( \left\lfloor \frac{r-1}{2}\right\rfloor\right)^{1-2/\beta} + \lambda_0^{\left\lfloor \frac{r-1}{2}\right\rfloor}\right]
$$
of which righhandside is summable. As to the pointwise convergence of $J_1(r,k_n)$ for fixed $r$ as $n\to\infty$, we use again a domination argument in the variables $s$ and $t$. Indeed we have that $|\lambda_0^{s+t} \mbox{Cov}(Y_0Y_{k_n+1}, \epsilon_{r+k_n+1-t} \epsilon_{r+2k_n+2-s})|\le K \lambda_0^{s+t}$ for some constant $K$ thanks to the Cauchy Schwarz inequality, and we have
\begin{eqnarray}
&&\mbox{Cov}(Y_0Y_{k_n+1}, \epsilon_{r+k_n+1-t} \epsilon_{r+2k_n+2-s}) = \mbox{Cov}(Y_0,Y_{k_n+1} \epsilon_{r+k_n+1-t} \epsilon_{r+2k_n+2-s}) \nonumber\\
&&+ \nbE(Y_0) \nbE(Y_{k_n+1} \epsilon_{r+k_n+1-t} \epsilon_{r+2k_n+2-s})  - \nbE(Y_0Y_{k_n+1}) \nbE(\epsilon_{r+k_n+1-t} \epsilon_{r+2k_n+2-s}) \nonumber\\
&&= \mbox{Cov}(Y_0,Y_{k_n+1} \epsilon_{r+k_n+1-t} \epsilon_{r+2k_n+2-s})\nonumber\\
&& +C_1 \nbE(Y_{0} \epsilon_{r-t} \epsilon_{r+k_n+1-s})- u_{k_n}\nbE (\epsilon_{r-t} \epsilon_{r+k_n+1-s}) \label{pr_prop_k_n_16}
\end{eqnarray}
Standard arguments yields the following limits as $n\to \infty$:
\begin{eqnarray*}
\mbox{Cov}(Y_0,Y_{k_n+1} \epsilon_{r+k_n+1-t} \epsilon_{r+2k_n+2-s})&\longrightarrow & 0,\\
\nbE(Y_{0} \epsilon_{r-t} \epsilon_{r+k_n+1-s}) &\longrightarrow & \nbE(Y_{0} \epsilon_{r-t}) m_0,\\
u_{k_n}\nbE (\epsilon_{r-t} \epsilon_{r+k_n+1-s}) &\longrightarrow & C_1^2 m_0^2 ,
\end{eqnarray*}
which, plugged into \eqref{pr_prop_k_n_16}, yields the following thanks to Relation \eqref{expr_v_k}
\begin{eqnarray*}
J_1(r,k_n) &\longrightarrow & \sum_{s=0}^\infty \sum_{t=0}^{r-1}\lambda_0^{s+t} [C_1 \nbE(Y_{0} \epsilon_{r-t}) m_0 -C_1^2 m_0^2]\\
&=& \frac{C_1 m_0}{1-\lambda_0}\sum_{t=0}^{r-1} \lambda_0^t [v_{r-t-1}-C_1 m_0] = - \frac{C_1 m_0}{1-\lambda_0}\sum_{t=0}^{r-1} \lambda_0^t \chi_{r-t-1}\\
&=& - \frac{C_1 m_0}{1-\lambda_0}\lambda_0^r \sum_{t=1}^{r} \lambda_0^{-t} \chi_{t-1},\quad n\to \infty .
\end{eqnarray*}
Hence the following limit from \eqref{pr_prop_k_n_15}:
\begin{multline}\label{pr_prop_k_n_17}
\sum_{h=k_n+2}^n \sum_{s=0}^{k_n} \sum_{t=0}^{h-k_n-2} \lambda_0^{s+t} \mbox{Cov}(Y_0Y_{k_n+1}, \epsilon_{h-t} \epsilon_{h+k_n+1-s}) \longrightarrow - \frac{C_1 m_0}{1-\lambda_0}\sum_{r=1}^\infty \lambda_0^r \sum_{t=1}^{r} \lambda_0^{-t} \chi_{t-1}\\
=- \frac{C_1 m_0}{(1-\lambda_0)^2}\sum_{t=1}^{\infty} \lambda_0^{-t} \chi_{t-1}.
\end{multline}
We now pass on to the second term in \eqref{pr_prop_k_n_14} and we write
\begin{eqnarray}
\sum_{h=k_n+2}^n \sum_{s=0}^{k_n}\  \lambda_0^{s+h}
\mbox{Cov}(Y_0, Y_0 \epsilon_{k_n+1+h-s})&=& \sum_{r=1}^{n-k_n-1} J_2(r,k_n),
\label{pr_prop_k_n_18}\\
J_2(r,k_n) &=& \sum_{s=0}^{k_n}\  \lambda_0^{s+r+k_n+1}
\mbox{Cov}(Y_0, Y_0 \epsilon_{r+2k_n+2-s}).\nonumber
\end{eqnarray}
From the Cauchy Schwarz inequality as well as crude bounding we get that $|J_2(r,k_n)|$ is upper bounded by $K\lambda_0^{r+k_n}$ for some constant $K$, so that from \eqref{pr_prop_k_n_17} we obtain that
\begin{equation}
\left| \sum_{h=k_n+2}^n \sum_{s=0}^{k_n}\  \lambda_0^{s+h}
\mbox{Cov}(Y_0, Y_0 \epsilon_{k_n+1+h-s})\right| =\mathrm{O}(\lambda_0^{k_n})\longrightarrow 0,\quad n\to\infty .\label{pr_prop_k_n_19}
\end{equation}
As to the third term in \eqref{pr_prop_k_n_14}, standard estimates yield that $\sum_{h=1}^\infty \mbox{Cov}(Y_0, Y_h^2)$ is a convergent series, so that we get
\begin{equation}
\sum_{h=k_n+2}^n \lambda_0^{k_n+1} \mbox{Cov}(Y_0, Y_h^2) =\mathrm{O}(\lambda_0^{k_n})\longrightarrow 0,\quad n\to\infty ,\label{pr_prop_k_n_20}
\end{equation}
so that, gathering \eqref{pr_prop_k_n_17}, \eqref{pr_prop_k_n_19} and \eqref{pr_prop_k_n_20}, we get from \eqref{pr_prop_k_n_14} that
\begin{equation}
\sum_{h=k_n+2}^n \mbox{Cov}(Y_0Y_{k_n+1}, Y_hY_{h+k_n+1}) \longrightarrow - \frac{C_1 m_0}{(1-\lambda_0)^2}\sum_{t=1}^{\infty} \lambda_0^{-t}\chi_{t-1},\quad n\to \infty .\label{pr_prop_k_n_21}
\end{equation}
Finally, gathering \eqref{pr_prop_k_n_11}, \eqref{pr_prop_k_n_13} and \eqref{pr_prop_k_n_21} we arrive at the following limit as $n\to \infty $:
\begin{equation}
\sum_{h=1}^n \mbox{Cov}(Y_0Y_{k_n}, Y_hY_{h+k_n+1}) \longrightarrow  \sum_{h=1}^{\infty} (u_{h-1}^2 -C_1^4) +  C_1^2 \sum_{h=0}^\infty (u_{h-1}- C_1^2) - \frac{C_1 m_0}{(1-\lambda_0)^2}\sum_{t=1}^{\infty} \lambda_0^{-t}\chi_{t-1}.\label{pr_prop_k_n_22}
\end{equation}
A similar analysis yields that $\sum_{h=1}^n h\; \mbox{Cov}(Y_0Y_{k_n}, Y_h Y_{h+k_n})$ converges as $n\to\infty$ so that $\lim_{n\to\infty} \sum_{h=1}^n \frac{h}{n}\; \mbox{Cov}(Y_0Y_{k_n}, Y_h Y_{h+k_n})=0$. Also, it is not difficult to check that\\ $\lim_{n\to\infty} \mbox{Cov}(Y_0Y_{k_n+1}, Y_0 Y_{k_n+1})=\lim_{n\to\infty} \mbox{Cov}(Y_0^2,Y_{k_n+1}^2) + C_2^2 - C_1^4 + C_1^2 - u_{k_n}^2=C_2^2 - C_1^4$. 

As to  $\sum_{h=-n}^{-1} \mbox{Cov}(Y_0Y_{k_n+1}, Y_h Y_{h+k_n+1})$, we use again by a stationarity argument and a change of index that it is equal to $\sum_{h=1}^{n} \mbox{Cov}(Y_0Y_{k_n+1}, Y_h Y_{h+k_n+1})$, so that its limit as $n\to\infty$ is also given by the righthandside of \eqref{pr_prop_k_n_22}. Thus, we obtain the following limit
\begin{multline}
\frac{1}{n}\sum_{h=-n}^n (n-|h|)\mbox{Cov}(Y_0Y_{k_n+1}, Y_h Y_{h+k_n+1}) \longrightarrow C_2^2 - C_1^4 + 2 \sum_{h=1}^{\infty} (u_{h-1}^2 -C_1^4) +  2 C_1^2 \sum_{h=0}^\infty (u_{h-1}- C_1^2)\\
 - 2 \frac{C_1 m_0}{(1-\lambda_0)^2}\sum_{t=1}^{\infty} \lambda_0^{-t}\chi_{t-1},\quad n\to \infty .\label{pr_prop_k_n_23}
\end{multline}
Hence, in view of \eqref{pr_prop_k_n_1}, \eqref{pr_prop_k_n_2} and the limits \eqref{pr_prop_k_n_2bis}, \eqref{pr_prop_k_n_7} and \eqref{pr_prop_k_n_23}, we thus proved the expression \eqref{expression_big_Sigma} for $\mathfrak{S}$.\\
\paragraph{$\diamond$ Step 2:} Recalling for all $m\in \nbN$ the processes $\{Y_{n,m},\ n\in \nbZ \}$ and $\{Z_{n,m},\ n\in \nbZ \}$ in \eqref{proof_step_2_CTL_k0_2_0},  as well as introducing
\begin{eqnarray}
X_{j,m}(k_n)&=& X_{j,m}^1(k_n) + X_{j,m}^2(k_n)\label{def_X_j_k_n}\\
X_{j,m}^1(k_n)&=& [Y_{j,m}- \nbE(Y_{j,m}), Y_{j+k_n+1,m}- \nbE(Y_{j+k_n+1,m})]' \nonumber\\
X_{j,m}^2(k_n)&=& [Z_{j,m}- \nbE(Z_{j,m}), Z_{j+k_n+1,m}- \nbE(Z_{j+k_n+1,m})]',\nonumber
\end{eqnarray}
we now follow the scheme in Step 2 of the proof of Theorem \ref{theo_CLT_vector_k_0} with again the non negligible difference that $k_0$ is now replaced by $k_n$, and show that
\begin{equation}\label{Step2_k_n}
\frac{1}{\sqrt{n}}\sum_{i=1}^n X_{i,m}^1(k_n)\stackrel{\cal D}{\longrightarrow} {\cal N}(0,\mathfrak{S}^m),\quad n\to\infty ,
\end{equation}
for some semi-definite positive matrix $\mathfrak{S}^m\in \nbR^{2\times 2}$. 
For this we need to verify Condition (2) in \cite{FZ05} which says that the following limit exists and defines that matrix:
\begin{multline}\label{Step_2_CTL_kn_def_sigma2}
\frac{1}{n}\nbE\left( \left( \sum_{i=1}^n X_{i,m}^1(k_n)\right)\left( \sum_{i=1}^n X_{i,m}^1(k_n)\right)'\right)=\frac{1}{n}\nbE\left( \sum_{1\le i,j \le n}X_{i,m}^1(k_n)'X_{j,m}^1(k_n)\right)\\
=\frac{1}{n}\sum_{h=-n}^n (n-|h|) \nbE (X_{0,m}^1(k_n),X_{h,m}^1(k_n))
\longrightarrow \mathfrak{S}^m
\end{multline}
as $n\to\infty$, which we prove now. For this we first express the expectation above as covariances as follows 
\begin{equation}\label{Cov_X0m}
\nbE(X_{0,m}^1(k_n),X_{h,m}^1(k_n))=
\left[
\begin{array}{cc}
\mbox{Cov}(Y_{0,m},Y_{h,m}) & \mbox{Cov}(Y_{0,m},Y_{h+k_n+1,m})\\
\mbox{Cov}(Y_{k_n+1,m},Y_{h,m}) & \mbox{Cov}(Y_{k_n+1,m},Y_{h+k_n+1,m})
\end{array}
\right] ,
\end{equation}
where $\mbox{Cov}(Y_{k_n+1,m},Y_{h+k_n+1,m})=\mbox{Cov}(Y_{0,m},Y_{h,m})$ by stationarity. We then study the convergence of the series with general terms given by the entries of the above matrix, i.e. convergence of $\sum_{h=-n}^n \mbox{Cov}(Y_{0,m},Y_{h,m})$, $\sum_{h=-n}^n \mbox{Cov}(Y_{0,m},Y_{h+k_n+1,m})$ and $\sum_{h=-n}^n \mbox{Cov}(Y_{k_n+1,m},Y_{h,m}) $ as $n\to \infty$. Standard arguments yield the convergence of\\ $\sum_{h=-n}^n \mbox{Cov}(Y_{0,m},Y_{h,m})$, as the summand does not depend on $n$, so that we are going to focus on the convergence of $\sum_{h=-n}^n \mbox{Cov}(Y_{0,m},Y_{h+k_n+1,m})$, the one with the summand $\mbox{Cov}(Y_{k_n+1,m},Y_{h,m})$ being dealt with similarly. Remembering the definition \eqref{proof_step_2_CTL_k0_2_0} for $Y_{0,m}$ and $Y_{h,m}$, we have
\begin{equation}\label{Cov_X0m_0}
\mbox{Cov}(Y_{0,m},Y_{h+k_n+1,m})=\sum_{r,r'=0}^m \mbox{Cov}(\theta^{(r)}_0\circ\epsilon_{-r}, \theta^{(r')}_{h+k_n+1}\circ\epsilon_{h+k_n+1-r'}).
\end{equation}
It thus suffices to prove that $\sum_{h=-n}^n \mbox{Cov}(\theta^{(r)}_0\circ\epsilon_{-r}, \theta^{(r')}_{h+k_n+1}\circ\epsilon_{h+k_n+1-r'})$ admits a limit as $n \to \infty$ for all $r$ and $r'$ in $0,...,m$ to be able to conclude as to the limit of $\sum_{h=-n}^n \mbox{Cov}(Y_{0,m},Y_{h+k_n+1,m})$. We prove that the limit exists when $r$ and $r'$ are both $0$ w.l.o.g. Remembering that $\theta^{(0)}_n=\mbox{Id}$, we thus need to study the limit of
$\sum_{h=-n}^n \mbox{Cov}(\epsilon_{0}, \epsilon_{h+k_n+1})$
which we again split into $\sum_{h=0}^n$ and $\sum_{h=-n}^{-1}$. A dominated convervgence argument gives that 
\begin{equation}\label{Cov_X0m_1_0}
\sum_{h=0}^n \mbox{Cov}(\epsilon_{0}, \epsilon_{h+k_n+1}) \longrightarrow 0,\quad n\to \infty.
\end{equation}
Next, writing that
\begin{equation}\label{Cov_X0m_2}
\sum_{h=-n}^{-1}\mbox{Cov}(\epsilon_{0}, \epsilon_{h+k_n+1})=\sum_{h=1}^{n}\mbox{Cov}(\epsilon_{0}, \epsilon_{k_n+1-h})=\sum_{h=1}^{k_n+1}+ \sum_{h=k_n+2}^{n},
\end{equation}
of which two latter sums have the following limits (remember that $k_n=\mathrm{o}(n)$ by assumption, so that $n-k_n\longrightarrow \infty$ as $n\to \infty$):
\begin{eqnarray*}
\sum_{h=1}^{k_n+1}\mbox{Cov}(\epsilon_{0}, \epsilon_{k_n+1-h}) &=& \sum_{h=0}^{k_n}\mbox{Cov}(\epsilon_{0}, \epsilon_{h})=\sum_{h=0}^{k_n} \nu_h\\
&\longrightarrow & \sum_{h=0}^{\infty} \nu_h ,\quad n\to \infty ,\\
 \sum_{h=k_n+2}^{n} \mbox{Cov}(\epsilon_{0}, \epsilon_{k_n+1-h}) &=&  \sum_{h=k_n+2}^{n} \mbox{Cov}(\epsilon_{0}, \epsilon_{h-k_n-1})=\sum_{h=1}^{n-k_n-1} \mbox{Cov}(\epsilon_{0}, \epsilon_{h})\\
 &=& \sum_{h=1}^{n-k_n-1} \nu_h \longrightarrow  \sum_{h=1}^{\infty} \nu_h ,\quad n\to \infty 
\end{eqnarray*}
which, plugged into \eqref{Cov_X0m_2} and \eqref{Cov_X0m_1_0}, yields the following limit
$$
\sum_{h=-n}^n \mbox{Cov}(Y_{0,m},Y_{h+k_n+1,m}) \longrightarrow \nu_0 + 2 \sum_{h=1}^{\infty} \nu_h ,\quad n\to \infty .
$$
Note again that what we wanted was simply the existence of the limit, of which explicit expression is just a bonus. All in all, we thus have from \eqref{Cov_X0m_0} that $\sum_{h=-n}^n \mbox{Cov}(Y_{0,m},Y_{h+k_n+1,m})$ admits a limit as $n\to \infty$. Similarly, one proves the existence of $\sum_{h=-n}^n |h \mbox{Cov}(Y_{0,m},Y_{h+k_n+1,m})|$ so that, by \eqref{Cov_X0m} and the dominated convergence theorem the limit $\mathfrak{S}^m$ indeed exists in \eqref{Step_2_CTL_kn_def_sigma2}.

We now conclude by saying that Condition (1) in \cite{FZ05} is satisfied with $\nu^*=2\beta-2$. As discussed in the beginning of \citet[Section 2]{FZ05}, and using the notation therein, Condition (3) is verified here with $T_n=2h_n$, and thanks to the assumption \eqref{Cond_CLT_k_n}. Hence the convergence in distribution \eqref{Step2_k_n} holds.\\
\paragraph{$\diamond$  Step 3:} As in Step 3 of the proof of Proposition \ref{prop_main_theo2}, we prove that 
\begin{equation}\label{step_3_CTL_kn}
\lim_{m\to \infty}\limsup_{n\to \infty}\nbP\left( \left| \left| \frac{1}{\sqrt{n}} \sum_{i=1}^n X_{i,m}^2(k_n) \right|\right|> \epsilon\right)=0,\quad \forall \epsilon >0.
\end{equation}
The equivalent of \eqref{proof_step_2_CTL_k0_18} is here
\begin{multline}\label{step3_to_prove1}
\nbP\left( \left| \left| \frac{1}{\sqrt{n}} \sum_{i=1}^n X_{i,m}^2(k_n) \right|\right|> \epsilon\right) \le \frac{1}{n\epsilon^2}\nbE\left( \left| \left| \sum_{i=1}^n X_{i,m}^2(k_n) \right|\right|^2\right)\\
= \frac{1}{n\epsilon^2}\sum_{h=-n}^n(n-|h|)\nbE\left( X_{0,m}^2(k_n)'X_{h,m}^2(k_n)\right).
\end{multline}
We see that it suffices to prove that%
\begin{equation}\label{step3_to_prove2_0}
\sum_{h=-n}^n\| \nbE\left( X_{0,m}^2(k_n)'X_{h,m}^2(k_n)\right)\| \le K \upsilon^m,\quad \sum_{h=-n}^n |h|.\|\nbE\left( X_{0,m}^2(k_n)'X_{h,m}^2(k_n)\right)\|\le K \upsilon^m
\end{equation}
in order to prove \eqref{step_3_CTL_kn}. We only prove this property for $\sum_{h=-n}^n\nbE\left( X_{0,m}^2(k_n)'X_{h,m}^2(k_n)\right)$, the proof being similar for $\sum_{h=-n}^n h\nbE\left( X_{0,m}^2(k_n)'X_{h,m}^2(k_n)\right)$. For this, and since, as in \eqref{Cov_X0m}, one has the expression
\begin{equation}\label{Cov_X02m}
\nbE(X_{0,m}^2(k_n),X_{h,m}^2(k_n))=
\left[
\begin{array}{cc}
\mbox{Cov}(Z_{0,m},Z_{h,m}) & \mbox{Cov}(Z_{0,m},Z_{h+k_n+1,m})\\
\mbox{Cov}(Z_{k_n+1,m},Z_{h,m}) & \mbox{Cov}(Z_{0,m},Z_{h,m})
\end{array}
\right] ,
\end{equation}
we saw that it suffices to prove that
\begin{equation}\label{step3_to_prove2}
\sum_{h=-n}^n |\mbox{Cov}(Z_{0,m}, Z_{h+k_n+1})| \le K\upsilon^m,\quad \forall n\in \nbN ,
\end{equation}
the other terms in \eqref{Cov_X02m} being dealt with similarly. Similarly to the decomposition \eqref{proof_step_2_CTL_k0_8}, 
we write thanks to the definition \eqref{proof_step_2_CTL_k0_2_0}, and for $h= 0,...,n$:
\begin{eqnarray}
\mbox{Cov}(Z_{0,m}, Z_{h+k_n+1}) 
&=& I_1(h,m,k_n)+I_2(h,m,k_n),\label{Step3_decomp}\\
I_1(h,m,k_n)&:=& \sum_{m+1\le i <\lfloor h/2 \rfloor +k_n+1} \mbox{Cov} \left(Z_{0,m}, [\theta^{(i)}_{h+k_n+1}\circ\epsilon_{h+k_n+1-i}]\right), \nonumber\\
I_2(h,m,k_n)&:=& \sum_{  \max(\lfloor h/2 \rfloor +k_n+1 ,m+1 ) \le i } \mbox{Cov} \left(Z_{0,m}, [\theta^{(i)}_{h+k_n+1}\circ\epsilon_{h+k_n+1-i}]\right), \nonumber
\end{eqnarray}
with the convention that the sum over an empty set is equal to $0$ (which is the case for example for $ I_1(h,m,k_n)$ when $\lfloor h/2 \rfloor \ge m+1$). An argument similar to \eqref{proof_step_2_CTL_k0_9} and \eqref{proof_step_2_CTL_k0_11} yields the following inequalities
\begin{eqnarray*}
|I_1(h,m,k_n)| &\le & K \sum_{m+1\le i <\lfloor h/2 \rfloor +k_n+1} \left\| Z_{0,m} - \nbE (Z_{0,m})\right\|_\beta \\
&& .\left\| \theta^{(i)}_{h+k_n+1}\circ\epsilon_{h+k_n+1-i} - \nbE \left(\theta^{(i)}_{h+k_n+1}\circ\epsilon_{h+k_n+1-i}\right) \right\|_\beta \;
\alpha_\epsilon(\lfloor h/2 \rfloor)^{1-2/\beta},\\
&& h=0,...,n,\\
\left\| Z_{0,m} - \nbE (Z_{0,m})\right\|_\beta &\le & K \nu^m ,
\end{eqnarray*}
leading to the bound
\begin{equation}\label{Step3_I1}
\left|I_1(h,m,k_n)\right| \le K \upsilon^m \alpha_\epsilon\left(\lfloor h/2 \rfloor\right)^{1-2/\beta},\quad h=0,...,n.
\end{equation}
As to $I_2(h,m,k_n)$, we need to split the sum further according to whether $\lfloor h/2 \rfloor +k_n+1$ is larger or less than $m+1$. When it is larger than $m+1$ then an upper bound as in \eqref{proof_step_2_CTL_k0_13} holds, so that we obtain
\begin{eqnarray*}
|I_2(h,m,k_n)| &\le & K \nu^{[\lfloor h/2 \rfloor +k_n+1} \nu^m \nbu_{[\lfloor h/2 \rfloor +k_n +1\ge m+1]} + K \nu^m \nu^m\nbu_{[\lfloor h/2 \rfloor +k_n +1< m+1]},\\
&\le & K \nu^{[\lfloor h/2 \rfloor} \nu^m \nbu_{[\lfloor h/2 \rfloor +k_n +1\ge m+1]} + K \nu^{2m}\nbu_{[h < 2(m+1)]},\ h=0,...,n
\end{eqnarray*}
where we used the fact that $\lfloor h/2 \rfloor +k_n +1< m+1$ implies that $h< 2(m+1)$. Summing \eqref{Step3_I1} and the above inequality on $h=0,...,n$ leads to 
\begin{multline}\label{Step3_sumI}
\sum_{h=0}^n |\mbox{Cov}(Z_{0,m}, Z_{h+k_n+1})| \le  K \nu^m \left[\sum_{h=0}^\infty \alpha_\epsilon\left(\lfloor h/2 \rfloor\right)^{1-2/\beta} \right] \\
+ K \nu^m \left[\sum_{h=0}^\infty \nu^{\lfloor h/2 \rfloor} \right] + 2(m+1)\nu^{2m} \le K \nu^m .
\end{multline}
We now consider $\sum_{h=-n}^{-1} |\mbox{Cov}(Z_{0,m}, Z_{h+k_n+1})|$. We first observe that, by a change of index and stationarity,
\begin{equation}\label{Step3_change_index}
\sum_{h=-n}^{-1} |\mbox{Cov}(Z_{0,m}, Z_{h+k_n+1})| = \sum_{h=1}^{n} |\mbox{Cov}(Z_{k_n+1,m}, Z_{h,m})|=\sum_{h=1}^{k_n+1} + \sum_{h=k_n+2}^{n}.
\end{equation}
Let us first notice that the following bound holds thanks to standard arguments:
$$
|\mbox{Cov}(Z_{h,m}, Z_{0,m})| \le K \nu^m \left[  \alpha_\epsilon\left(\lfloor h/2 \rfloor\right)^{1-2/\beta}  + \upsilon^{\lfloor h/2 \rfloor}  \right],
$$
so that the two sums on the righthandside of \eqref{Step3_change_index} verify the bounds
\begin{eqnarray*}
\sum_{h=1}^{k_n+1} |\mbox{Cov}(Z_{k_n+1,m}, Z_{h,m})| &=& \sum_{h=1}^{k_n+1} |\mbox{Cov}(Z_{k_n+1-h,m}, Z_{0,m})|= \sum_{h=0}^{k_n} |\mbox{Cov}(Z_{h,m}, Z_{0,m})|\\
&\le & \sum_{h=0}^{\infty} |\mbox{Cov}(Z_{h,m}, Z_{0,m})| \le K \nu^m \sum_{h=0}^{\infty}\left[  \alpha_\epsilon\left(\lfloor h/2 \rfloor\right)^{1-2/\beta}  + \upsilon^{\lfloor h/2 \rfloor}  \right]\\
\sum_{h=k_n+2}^{n} |\mbox{Cov}(Z_{k_n+1,m}, Z_{h,m})| &=& \sum_{h=k_n+2}^{n} |\mbox{Cov}(Z_{0,m}, Z_{h-k_n+1,m})| = \sum_{h=1}^{n-k_n-1} |\mbox{Cov}(Z_{0,m}, Z_{h,m})| \\
&\le & \sum_{h=0}^{\infty} |\mbox{Cov}(Z_{h,m}, Z_{0,m})| \le K \nu^m \sum_{h=0}^{\infty}\left[  \alpha_\epsilon\left(\lfloor h/2 \rfloor\right)^{1-2/\beta}  + \upsilon^{\lfloor h/2 \rfloor}  \right]
\end{eqnarray*}
which, plugged into \eqref{Step3_change_index}, yields that $\sum_{h=-n}^{-1} |\mbox{Cov}(Z_{0,m}, Z_{h+k_n+1})|  \le K \nu^m$. In turn, this latter bound associated to \eqref{Step3_sumI} yields \eqref{step3_to_prove2} hence \eqref{step3_to_prove2_0}.
Thus, all in all, \eqref{step_3_CTL_kn} holds.\\
\paragraph{$\diamond$ Step 4:} As in the final step of the proof of Theorem \ref{theo_CLT_vector_k_0}, \eqref{Step2_k_n} and \eqref{step_3_CTL_kn} yield that $\lim_{m\to \infty}\mathfrak{S}^m=\mathfrak{S}$ defined by  \eqref{expression_big_Sigma}, as well as the convergence in disrtibution \eqref{CLT_general}. \zak

\subsection{Proof of Theorem \ref{main_theo_general1}}\label{proof_gen1}
Setting $k=k_n$ in \eqref{bound_S_k_n}, and remembering from \eqref{def_estimator_general} that $\hat{S}_n=\psi_{k_n}(\bar{Y}_n, \bar{Y}_{k_n+1,n})$, we get
$$
||\hat{S}_n - \ln \lambda_0||_2 \le K_S \left( \frac{1}{\lambda_-^{k_n}} \left|\left| \frac{\hat{S}_n}{n}\right|\right|_2 +  \frac{1}{k_n}\right) = K_S \left( \frac{1}{ \lambda_-^{k_n}\sqrt{n}} \left|\left| \frac{\hat{S}_n}{\sqrt{n}}\right|\right|_2 +  \frac{1}{k_n}\right).
$$
The convergence \eqref{CLT_general_Sigma} in Proposition \ref{prop_main_theo1} implies that $\left|\left| {\hat{S}_n}/{\sqrt{n}}\right|\right|_2=\mathrm{O}(1)$. The trick here is to choose $k_n=\lfloor c \ln(n)\rfloor$ where $c<-{1}/({2 \ln \lambda_-})$ which implies that $\lim_{n\to \infty }\lambda_-^{k_n}\sqrt{n}/k_n=\infty $, so that the upper bound on the righthanside of the above inequality is an $\mathrm{O}(1/k_n)=\mathrm{O}(1/\ln(n))$, proving the inequality \eqref{conv_gen_estimator_quad}. Finally, the convergence in probability of $\hat{N}_n$ comes from the a.s. convergence of $\bar{Y}_n$ towards $C_1$ and the convergence in probability of $e^{\hat{S}_n}$ (the latter stemming from the convergence in $\mathbb{L}^2$ of $\hat{S}_n$) towards $\lambda_0$, coupled with Relation \eqref{expr_first_moment}.
\zak
\subsection{Proof of Theorem \ref{main_theo_genera2}}\label{proof_gen2}
The increment theorem of function $\psi_{k_n}$ at point $U(k_n):=(C_1, u_{k_n})$ gives
\begin{equation}\label{Taylor_CLT}
\hat{S}_n-\psi_{k_n}(U(k_n))=\psi_{k_n}\left(U(k_n) + \frac{S_n(k_n)}{n}\right)-\psi_{k_n}(U(k_n))= \nabla \psi_{k_n}\left(U(k_n) + c_n\frac{S_n(k_n)}{n}\right).\frac{S_n(k_n)}{n}
\end{equation}
for some r.v. $c_n\in(0,1)$. Using again the increment theorem to $\nabla \psi_{k_n}$ at some point $U(k_n)$, we get, thanks to the upper bound \eqref{lemma_prop_main3_bis} for $D^2 \psi_{k_n}$ obtained in the upcoming Lemma \ref{lemma_before_prop_main_theo1}, that
\begin{multline}\label{Taylor_CLT_1}
\left| \left| k_n \lambda_0^{k_n}\left[ \nabla \psi_{k_n}\left(U(k_n) + c_n\frac{S_n(k_n)}{n}\right)- \nabla \psi_{k_n}\left(U(k_n) \right) \right] \right|\right|\\
\le k_n \lambda_0^{k_n} \sup_{(a,b)\in \nbR^2} ||D^2 \psi_{k_n}(a,b)||.\; \left| \left| c_n\frac{S_n(k_n)}{n} \right|\right|\\
\le K k_n \lambda_0^{k_n}\frac{1}{\lambda_-^{k_n}}\left| \left| c_n\frac{S_n(k_n)}{n} \right|\right|\le  K k_n \lambda_0^{k_n}\frac{1}{\lambda_-^{k_n}\sqrt{n}}\left| \left| \frac{S_n(k_n)}{\sqrt{n}} \right|\right|
\end{multline}
for some constant $K>0$. Furthermore, for $n$ large enough we have from \eqref{expr_C_1_u_k} as well as Assumption $\mathbf{ (A6)}$ that $|C_1^2-u_{k_n}|=\lambda_0^{k_n} \left|\sum_{j=0}^{k_n} \lambda_0^{-j} \chi_j + \lambda_0(C_1^2-C_2)\right|\ge \lambda_0^{k_n} K_m/2$, so that the definition \eqref{def_chi_Km} of $\varpi_k$ entails that $\varpi_{k_n}(|C_1^2-u_{k_n}|)=1$ and $\varpi_{k_n}'(|C_1^2-u_{k_n}|)=0$. This in turn entails that $\nabla \psi_{k_n}\left(U(k_n) \right)=\left( {2C_1}/({k_n|C_1^2 - u_{k_n}|}), -{1}/({k_n|C_1^2 - u_{k_n}|}) \right)'$ (see the upcoming calculation of $\partial_a\psi_k$ in \eqref{derive_psi} and a similar calculation for $\partial_b\psi_k$ for all $k$ in the proof of Lemma \ref{lemma_before_prop_main_theo1} in Section \ref{subsec:proof_gradient}).
Since $k_n$ obviously satisfies \eqref{Cond_CLT_k_n}, the central limit theorem \eqref{CLT_general} in Proposition \ref{prop_main_theo2} holds, and thanks to the expression \eqref{expr_C_1_u_k} we thus have
\begin{equation}\label{Taylor_CLT_2}
k_n \lambda_0^{k_n} \nabla \psi_{k_n}\left(U(k_n) \right)\frac{S_n(k_n)}{\sqrt{n}}\stackrel{\cal D}{\longrightarrow} {\cal N}
\left( 0, V
 \mathfrak{S}V'
\right),\quad n\to \infty
\end{equation}
where $\mathfrak{S}$ is given by \eqref{expression_big_Sigma} and $V$ is given by \eqref{def_sigma_gen_V}. Now, the choice of $k_n:=\lfloor c \ln n\rfloor$ where $c<-{1}/({2 \ln \lambda_-)}$ implies that $\lim_{n\to \infty} \lambda_-^{k_n}\sqrt{n}=\infty$, and $\left| \left| {S_n(k_n)}/{\sqrt{n}} \right|\right|$ is bounded in $\mathbb{L}^2$ thanks to \eqref{CLT_general_Sigma}, so that \eqref{Taylor_CLT_1} implies that
\begin{multline*}\label{ineq_later}
\left| \left| k_n \lambda_0^{k_n}\left[ \nabla \psi_{k_n}\left(U(k_n) + c_n\frac{S_n(k_n)}{n}\right)- \nabla \psi_{k_n}\left(U(k_n) \right) \right] .\frac{S_n(k_n)}{\sqrt{n}}\right|\right|\\
\le 
\left| \left| k_n \lambda_0^{k_n}\left[ \nabla \psi_{k_n}\left(U(k_n) + c_n\frac{S_n(k_n)}{n}\right)- \nabla \psi_{k_n}\left(U(k_n) \right) \right] \right|\right|.\left| \left| \frac{S_n(k_n)}{\sqrt{n}} \right|\right|\\
\le K k_n \lambda_0^{k_n}\frac{1}{\lambda_-^{k_n}\sqrt{n}}\left| \left| \frac{S_n(k_n)}{\sqrt{n}} \right|\right|^2 \stackrel{\mathbb{L}^1}{\longrightarrow} 0
\end{multline*}
as $n\to\infty$ (in the above inequalities, it is convenient to take a multiplicative matrix norm $||\cdot||$ which verifies $||Mv|| \le ||M||.||v||$ for all matrix $M$ and vector $v$). Coupled with \eqref{Taylor_CLT_2},  we then get the following asymptotic normality
\begin{equation}\label{Taylor_CLT_3}
k_n \lambda_0^{k_n} \nabla \psi_{k_n}\left(U(k_n) + c_n\frac{S_n(k_n)}{n}\right).\frac{S_n(k_n)}{\sqrt{n}}\stackrel{\cal D}{\longrightarrow} {\cal N}
\left( 0, V
 \mathfrak{S}V'
\right),\quad n\to \infty .
\end{equation}
Having considered the righthandside in \eqref{Taylor_CLT}, we now study its lefthandside. Since we saw that $\varpi_{k_n}(|C_1^2-u_{k_n}|)=1$ for $n$ large enough, we thus have that $ \psi_{k_n}(U(k_n))={k^{-1}_n}\ln |C_1^2- u_{k_n} |$ so that, from \eqref{expr_C_1_u_k}:
\begin{multline}
k_n \lambda_0^{k_n} [\psi_{k_n}(U(k_n))- \ln \lambda_0]= \lambda_0^{k_n} \ln \left| \sum_{j=0}^{k_n} \lambda_0^{-j} \chi_j + \lambda_0(C_1^2-C_2)\right|\\
= \lambda_0^{k_n} \ln \left| \sum_{j=0}^{\infty} \lambda_0^{-j} \chi_j + \lambda_0(C_1^2-C_2)\right| + \lambda_0^{k_n} \ln \left|1 -  \frac{\sum_{j=k_n+1}^{\infty} \lambda_0^{-j} \chi_j}{\left| \sum_{j=0}^{\infty} \lambda_0^{-j} \chi_j + \lambda_0(C_1^2-C_2)\right|}
\right|.\label{Taylor_CLT_4}
\end{multline}
The assumption $\nu_h= \mathrm{O}(\zeta^h)$ entails from the definition \eqref{def_chi} and $\zeta<\lambda_-\le \lambda_0$ that $\chi_j=\mathrm{O}(\zeta^j)$ so that $\sum_{j=k_n+1}^{\infty} \lambda_0^{-j} \chi_j=\mathrm{O}\left(  \left( {\zeta}/{\lambda_0}\right)^{k_n} \right)$, a quantity that tends to $0$ as $n\to\infty$. We thus deduce from \eqref{Taylor_CLT_4} and the classical expansion $\ln (1+u)= \mathrm{O}(u)$ as $u\to 0$ that
\begin{equation}\label{Taylor_CLT_5}
k_n \lambda_0^{k_n} [\psi_{k_n}(U(k_n))- \ln \lambda_0]=\lambda_0^{k_n} \ln \left| \sum_{j=0}^{\infty} \lambda_0^{-j} \chi_j + \lambda_0(C_1^2-C_2)\right| + \mathrm{O}(\zeta^{k_n}),\quad n\to\infty .
\end{equation}
Gathering \eqref{Taylor_CLT} and \eqref{Taylor_CLT_5} yields
\begin{multline*}
\hat{S}_n=\ln \lambda_0 + \frac{1}{k_n} \ln \left| \sum_{j=0}^\infty \lambda_0^{-j} \chi_j + \lambda_0(C_1^2-C_2)\right| \\
+ \frac{1}{\sqrt{n} k_n \lambda_0^{k_n}}\left[ k_n \lambda_0^{k_n} \nabla \psi_{k_n}\left(U(k_n) + c_n\frac{S_n(k_n)}{n}\right).\frac{S_n(k_n)}{\sqrt{n}} +  \mathrm{O}(\sqrt{n}\zeta^{k_n})\right].
\end{multline*}
Now, the assumption $c>-{1}/({2 \ln \zeta})$ implies that $\lim_{n\to\infty} \sqrt{n}\zeta^{k_n}=0$, so that the expansion \eqref{expansion_Sn} holds with $Z_n:= k_n \lambda_0^{k_n} \nabla \psi_{k_n}\left(U(k_n) + c_n{S_n(k_n)}/{n}\right).{S_n(k_n)}/{\sqrt{n}} +  \mathrm{O}(\sqrt{n}\zeta^{k_n})$, which converges in distribution towards ${\cal N}(0,\sigma)$ thanks to  \eqref{Taylor_CLT_3}, with $\sigma$ defined as \eqref{def_sigma_gen}. Also note that the assumption $c<-{1}/({2 \ln \lambda_-})$ implies that $ Z_n/{\sqrt{n}k_n \lambda_0^{k_n}} $ is indeed as second asymptotic term, as indeed we have $\lim_{n\to \infty} \sqrt{n} \lambda_0^{k_n}=\infty$.
\zak

\begin{center}
{\bf Acknowledgements}
\end{center}
The authors thank the anonymous referee and the Associate Editor  for the very careful reading of the paper.

\bibliographystyle{apalike}
\bibliography{biblio-INAR}



\end{document}